%% file: main.tex
\pdfoutput=1
\documentclass[12 pt]{amsart}
\title{Linear Response for Contracting on Average Iterated Function Systems}
\author{Jianning Fu}
\date{\today}
\usepackage[letterpaper,margin=1in]{geometry}
\usepackage[normalem]{ulem}
\usepackage{xcolor}
\usepackage{graphicx}
\usepackage[utf8]{inputenc}
\usepackage[T1]{fontenc}
\usepackage{csquotes}
\usepackage{amsmath}
\usepackage{amsthm}
\usepackage{amsfonts}
\usepackage{mathrsfs}
\usepackage{amssymb} 
\usepackage{fancyhdr}
\usepackage{enumerate}
\usepackage{tgschola} 
\usepackage[
backend=biber,
style=alphabetic,
sorting=ynt
]{biblatex}
\renewbibmacro{in:}{}
\AtEveryBibitem{%
  \clearfield{number}%
  \clearfield{doi}%
  \clearfield{url}%
  \clearfield{isbn}%
  \clearfield{issn}%
  \clearfield{note}%
  \clearlist{publisher}%
  \clearlist{location}%
}
\usepackage{xspace}

\newtheorem{thm}{Theorem}[section]

\newtheorem{remark}{Remark}[section]

\newtheorem{prop}{Proposition}[section]
\newtheorem{lm}{Lemma}[section]

\numberwithin{equation}{section}

\def\DS{\displaystyle}

\usepackage[
colorlinks,
breaklinks,
unicode
]{hyperref}

\begin{document}
\maketitle	
\begin{abstract}
   Consider the following probabilistic contracting on average iterated function system 
$$\Phi = \left\{f_i (x) = \lambda_i x + d_i,\;i=1,2 ;\;\; p = \left(\frac{1}{2} , \frac{1}{2}\right) \right\},$$
where the contraction ratios $\lambda_1 , \lambda_2$ are such that $0<\lambda_1<1<\lambda_2$ and  $\lambda_1\lambda_2<1$. Denote by $\mu_{\lambda_1,\lambda_2}$ its stationary measure. We study the differentiability of 
$$(\heartsuit)\quad\quad\quad\quad\quad
\lambda_1\mapsto\int_{\mathbb{R}}\phi(x)\,d\mu_{\lambda_1,\lambda_2}(x), $$
 where $\phi$ is a suitable test function. We establish three cases where $(\heartsuit)$
 is differentiable and show the derivative coincides with the one obtained by taking formal derivative, which can be generalized to the case of multiple maps with different probabilities.
 We also present sufficient conditions under which there exists a smooth, bounded test function $\phi$ so that  $(\heartsuit)$ is not differentiable. 
\end{abstract}

	\bigskip

\section{Introduction and Statement of Results}\label{Section: Introduction}

To state the main results, let us first recall some basic facts of self-similar sets and measures. Let $m\ge2 $ and $d\ge 1$ be positive integers and consider $m$ maps $\mathcal{F} = (f_1, \cdots f_m)$ acting on $\mathbb{R}^d$, with each map $f_i$ having a contraction ratio $0<r_i<1$. That is,

$$\forall x,y\in\mathbb{R}^d , \quad ||f_i (x) - f_i (y)|| = r_i ||x-y||,\quad i=1,\cdots m.$$
Such a tuple of maps $\mathcal{F}$ is called a \textbf{self-similar iterated function system} (self-similar \textbf{IFS} or \textbf{IFS} for short). 

It is well known \cite{hutchinson1981fractals} that there exists a unique compact set $K = K_{\mathcal{F}}$, called the \textbf{attractor} of $\mathcal{S}$, such that
$$K = \bigcup_{i=1}^{m} f_i (K).$$

Throughout, we always assume that not all $f_i$'s share the same fixed point, so that the attractor of $\mathcal{S}$ is not a single point. 

One can view the attractor via symbolic dynamics as follows. Define $\mathscr{A} = \{1,\cdots m \}$ to be the alphabet, $\Sigma = \mathscr{A}^{\mathbb{N}}$ the shift space and $\sigma:\Sigma\to \Sigma$ the left shift. Consider the map $\Pi = \Pi_{\mathcal{S}}: \Sigma\to \mathbb{R}^d$ that sends any word $i = (i_1,i_2,\cdots)\in\Sigma$ to
$$\Pi (i) = \lim_{n\to\infty} f_{i_1} \circ \cdots f_{i_n} (0) =: \lim_{n\to\infty} f_{i_1 \cdots i_n} (0).$$

Note that since each $f_i$ is contracting, $\DS \lim_{n\to\infty} f_{i_1} \circ \cdots f_{i_n} (x)$ exists and is independent of the choice of $x$, justifying the above definition. The attractor $K$ is exactly the image of the projection $\Pi$.

An additional probability structure can be put on IFS. Namely, if $\mathcal{F} = (f_1, \cdots,f_m)$ is a self-similar IFS and $p = (p_1,\cdots , p_m)$ is a probability vector, then we call $\Phi = (\mathcal{F},p)$ a \textbf{probabilistic IFS}. A Borel probability measure $\mu$ on $\mathbb{R^d}$ is called the \textbf{stationary measure} for the IFS $\Phi$, if 
$$\mu = \sum_{j=1}^m p_j (f_j)_{*} \mu.$$
That is, if we consider the Markov process where for any $x\in \mathbb{R}^d$, it is mapped to $f_j(x)$ with probability $p_j$, then $\mu$ is the corresponding stationary measure. It is well known that such measure is unique, so it is indeed ``the'' stationary measure. Moreover, it can also be described symbolically. Define $\nu = p^{\mathbb{N}}$ to be the Bernoulli measure on $\Sigma$. Then 
$\DS\mu = \Pi_* \nu.$
We refer to \cite{BSSBook}, \cite{Falc} for the aforementioned properties as well as more discussions on (probabilistic) IFS.

The notion of attractor and stationary measure can be generalized to the case where not all $r_i$ are less than $1$. However, we require that the IFS \textbf{contracts on average}, meaning that
$$\sum_{j=1}^m p_j \log r_j <0.$$

In this case, there is also a unique invariant measure $\mu$ \cite{kifer2012ergodic}, which is still characterized by $\mu = \Pi_* \nu$. One slight difference is that in the contracting on average case, the projection $\Pi$ is only defined for $\nu$-almost every words. In addition, its support is no longer compact. 

Many results in the contracting case can be generalized to this stationary measure, including but not restricted to its dimension, density, order of tail and absolute continuity (with respect to Hausdorff measure of suitable dimension) \cite{guivarc2016spectral}\cite{kittle2025polynomial} \cite{gelfert2023contracting} \cite{kittle2024absolute} \cite{kittle2025dimension} \cite{Linde14Random}.

In the present work, we  study the linear response problem for stationary measures of a contracting on average IFS. In order to keep the notation simple,
we restrict ourselves to the following system 
$$\Phi = \left\{f_i (x) = \lambda_i x + d_i,\;i=1,2 ;\;\; p = \left(\frac{1}{2} , \frac{1}{2}\right) \right\},$$
where the contraction ratios $\lambda_1 , \lambda_2$ are such that $0<\lambda_1<1<\lambda_2$ and  $\lambda_1\lambda_2<1$.  We discuss in  Section \ref{Section: Extensions and open problems} more general setting where our approach works. 

Denote by $\mu_{\lambda_1,\lambda_2}$ the stationary measure of this IFS. Consider the map
\begin{equation}
\label{SMAv}
h = h_\phi:\epsilon\mapsto \int \phi \,d\mu_{\lambda_1+\epsilon,\lambda_2} ,
\end{equation}
defined for $\epsilon$ satisfying $(\lambda_1+\epsilon) \lambda_2<1$, where $\phi$ is a function of a certain regularity. By linear response  theory, we mean studying sufficient regularity conditions on $\phi$ so that $h_{\phi}'(0)$ exists 
and calculating the value of the derivative.

 For many smooth  or piecewise smooth dynamical systems, there is a well-developed linear response theory, originated in Ruelle's work \cite{ruelle1997differentiation} and developed by many others including but not restricted to \cite{dolgopyat2004differentiability, gouezel2006banach, baladi2008linear, baladi2007anisotropic}. There are also examples of failure of linear response, see  for example
\cite{baladielse, BBS15, dLS18}. However, not much is known about the analogous question for contracting (on average) IFS, and it is tempting to develop an analogous theory. Our method is based on the tail property of $\mu_{\lambda_1,\lambda_2}$ and large deviation theory, and 
may work in a more general setting.

In contrast, in a separate work we will show that for Bernoulli convolutions, to get (almost everywhere) differentiability of $h_\phi$ we only need H\"older continuous observable $\phi$. Moreover, differentiability with H\"older observable is related to the dimension of the measure.

   To state the main results we need additional notation. By the symbolic characterization, one get the symbolic space $\Sigma = \{1,2\}^{\mathbb{N}}$ equipped with the probability measure $\nu = (\frac{1}{2},\frac{1}{2})^{\mathbb{N}}$ and the projection
   $$\Pi(i_1,i_2,\cdots) = \sum_{j=0}^{\infty} d_{i_j} \lambda_{i_1}\cdots \lambda_{i_{j-1}}=:X.$$

   Moreover, $\mu_{\lambda_1,\lambda_2} = \Pi_{*} \nu$ will be the law of the random variable $X$.

If $\epsilon$ is sufficiently small so that $(\lambda_1+\epsilon)\lambda_2<1$,the random variable
$$X(\epsilon):=\sum_{j=0}^{\infty} d_{i_j} \lambda_{i_1}(\epsilon)\cdots \lambda_{i_{j-1}}(\epsilon) $$
is well-defined almost everywhere, where $\lambda_1(\epsilon) = \lambda_1+\epsilon$ and $\lambda_2(\epsilon) = \lambda_2$, and has law $\mu_{\lambda_1+\epsilon,\lambda_2}$.
With these notations, we can write
$$\int_{\mathbb{R}}\phi\,d\mu_{\lambda_1+\epsilon,\lambda_2} = \int_{\Sigma}\phi(X(\epsilon))\,d\nu.$$
Since $X(\epsilon)$ is a formal power series in $\epsilon$, one can take derivative formally. It is therefore natural to ask whether
\begin{equation}\label{Formal derivative equals actual derivative}
    h^{(l)}(0): = \frac{d^l h}{d\epsilon^l}(0)  =  \int_{\Sigma} \frac{d^l}{d\epsilon^l}(\phi\circ
 X (\epsilon))|_{\epsilon=0}\,d\nu ,\; l=1,2,\cdots
\end{equation}
under suitable assumptions on the test function $\phi$ and $\lambda_1,\lambda_2$. We will show three instances of equality \eqref{Formal derivative equals actual derivative} being true,.

 On the other hand, if $\mu_{\lambda_1,\lambda_2}$ is too concentrated on some intervals with small Lebesgue measure, it is possible to construct a $C^\infty$, bounded test function $\phi$ whose derivative is large around these points, so that $h_\phi$ is not differentiable at $0$ and, in particular,  equation \eqref{Formal derivative equals actual derivative} fails. We will give two sufficient conditions on $\lambda_1,\lambda_2$ for such $\phi$ to exist.

 Now we state the main results. We start with three cases where equation \eqref{Formal derivative equals actual derivative} is true.

\begin{thm}\label{Differentiability of Stationary measure}
Suppose $r\in\mathbb{N}$ and $0<\lambda_1<1<\lambda_2$ are such that $\DS \frac{\lambda_1^r+\lambda_2^r}{2}<1$. Then for any function $\phi\in C^r(\mathbb{R})$ with $||\phi^{(r)}||_{C^0}<\infty$, the average $h(\epsilon):= \int\phi\,d\mu_{\lambda_1+\epsilon,\lambda_2}$ is differentiable at $0$ up to order $r$ and the derivatives are given by    

\begin{align}\label{Equation: LR with moment condition and bounded derivative}
    \nonumber    h^{(l)}(0) &= \int_{\Sigma} \frac{\,d^l}{\,d \epsilon^l} (\phi\circ X(\epsilon))|_{\epsilon=0}\,d\nu \\
        &= \int_{\Sigma} \sum_{k_1 + 2k_2 + \cdots lk_l = l} L(k_1, \cdots, k_l) \phi^{(k)} (X) \prod_{j=1}^{l} (X^{(j)})^{k_j}\,d\nu,\; l=1,2,\cdots ,r.
    \end{align}

 The second equality above is simply the Faa di Bruno's formula for higher derivative of composition of two functions: the summation is over all non negative integers $k_1, \cdots ,k_l$ with $k_1 + 2k_2 + \cdots lk_l = l$. Moreover, $L(k_1,\cdots,k_l) := \frac{l!}{k_1 ! (1!)^{k_1} \cdots k_l! (l!)^{k_l}}$, $k := k_1 + \cdots + k_l$ 
 and $X^{(j)}$ is defined as 
   $$X^{(j)} := \frac{\,d^j}{\,d \epsilon^j}X(\epsilon)|_{\epsilon=0} = \frac{j!}{\lambda_1^j}\sum_{m=1}^{\infty} \lambda_{i_1} \cdots \lambda_{i_m} \binom{ o(m)}{j} $$
   where $o(m)$ denotes the number of $ 1$'s among $i_1,\cdots,i_m$.
\end{thm}

In the above theorem, the set of assumptions $\DS \frac{\lambda_1^r+\lambda_2^r}{2}<1$ and $||\phi^{(r)}||_{C^0}<\infty$ serve as some kind of ``moment'' condition to make sure that the right hand side of \eqref{Equation: LR with moment condition and bounded derivative} is finite. Another instance of such sufficient condition is the following, in which one gets infinite differentiability:
   
   \begin{thm}\label{Smoothness for Finite Moment Condition}
    Suppose $d_1=1,d_2=1$, $ 0<\lambda_1 <1<\lambda_2$ and $t>0$ are such that $\frac{\lambda_1^t+ \lambda_2^t}{2}<1$. Note that this implies $\lambda_1\lambda_2 <1$ and the stationary measure $\mu = \mu_{\lambda_1,\lambda_2}$ is well defined. Take the test function to be $\phi(x) = x^t,x\ge 1$ and define $h(\epsilon) = \int \phi(x) \,d\mu_{\lambda_1 + \epsilon , \lambda_2}$ for $\epsilon$ small. Then $h$ is infinitely differentiable at $0$ and the derivative coincide with the  formal derivative:
    \begin{align*}
        h^{(l)}(0) &= \int_{\Sigma} \frac{\,d^l}{\,d \epsilon^l} (\phi\circ X(\epsilon))|_{\epsilon=0}\,d\nu \\
        &= \int_{\Sigma} \sum_{k_1 + 2k_2 + \cdots lk_l = l} L(k_1, \cdots, k_l) \phi^{(k)} (X) \prod_{j=1}^{l} (X^{(j)})^{k_j}\,d\nu,\quad \forall l\in \mathbb{N}.
    \end{align*}
    Here, all the notations are the same as in Theorem \ref{Differentiability of Stationary measure}.
\end{thm}

Apart from moment conditions as in the above two results, that $\phi$ being compactly supported can also lead to the formal derivative being integrable and therefore validity of linear response, as the following result shows:

\begin{thm}\label{infinite differentiability for compactly supported test function}0,1
    Assume $0<\lambda_1 <1 <\lambda_2$ are such that $\lambda_1 \lambda_2 <1$ and $\phi\in C_c^{\infty}(\mathbb{R})$
    (the space of compactly supported smooth function). Then $h:\epsilon \to \int \phi \,d \mu_{\lambda_1 +\epsilon , \lambda_2}$ is 
    infinitely differentiable at $0$ with
    \begin{align*}
        h^{(l)}(0) = \int_{\Sigma} \frac{\,d^l}{\,d \epsilon^l} (\phi\circ X(\epsilon))|_{\epsilon=0}\,d\nu 
        = \int_{\Sigma} \sum_{k_1 + 2k_2 + \cdots lk_l = l} L(k_1, \cdots, k_l) \phi^{(k)} (X) \prod_{j=1}^{l} (X^{(j)})^{k_j}\,d\nu.
    \end{align*}
\end{thm}

Note that in the above three theorems, the assumptions remain valid if $\lambda_1,\lambda_2$ are perturbed a little bit. Thus the linear response formula also holds at sufficiently small $\epsilon$. 

Next are the two instances of non-differentiability.

\begin{thm}\label{Section: non-diff example; Main Theorem}
    Assume $0<\lambda_1<1<\lambda_2$ are such that $\lambda_1 \lambda_2 <\frac{1}{4}$. Then there exists a smooth, bounded function $\phi$ such that $h(\epsilon) := \int \phi \,d_{\lambda_1+\epsilon , \lambda_2}$ is not differentiable at $0$.
\end{thm}

\begin{thm}\label{Section: Non-diff Examples; Main Theorem 2}
    Assume $0<\lambda_1<1<\lambda_2$ are such that $\lambda_1\lambda_2<1$ and  $\log_2\lambda_2 >1+ \log_{\frac{1}{\lambda_1}} \lambda_2$. Then there exists a smooth, bounded function $\phi$ such that $h(\epsilon) := \int \phi \,d_{\lambda_1+\epsilon , \lambda_2}$ is not differentiable at $0$. 
    \end{thm}

Note that the assumption of Theorem \ref{Section: Non-diff Examples; Main Theorem 2} 
is satisfied if $\lambda_2$ is large, like $\lambda_2\ge 4$.    

\begin{remark}
   The above give rise to examples of  (partially) hyperbolic solenoidal attractor where linear response fails for smooth observables. To see this, consider the skew product representation of the IFS, which is the map $F:[-1,1]\times \mathbb{R}\to [-1,1]\times \mathbb{R}$ ,
\begin{equation}
\nonumber F(x,y) = 
    \begin{cases}
    (2x+1,f_1(y)), \text{ if } x\in[-1,0]\\
    (2x-1,f_2(y)),\text{ if } x\in (0,1]
\end{cases}
.
\end{equation}
Then $\rho = \rho_{\lambda_1,\lambda_2} :=Leb \times \mu_{\lambda_1,\lambda_2}$ is the unique SRB measure of $F$. And $\Phi(x,y):= \phi(y)$ is a $C^\infty$ observable that fails linear response.

\end{remark}

\medskip

 The proof of the three affirmative results follow the same strategy. We will show that under the assumption of these theorems, $h(\epsilon)$ is exponentially close to its finite approximation $\int_\Sigma\phi(X_n)\,d\nu$, where  
$$X_n := \sum_{m=0}^n  \lambda_{i_1}\cdots\lambda_{i_m}  d_{i_{m+1}} .$$ Then we show that the derivative can be computed using finite approximations. As for the two non-differentiable examples, the construction takes advantage of the fact that the partial sum $\DS X_N = \sum_{m=0}^{N} \lambda_{i_1}\cdots\lambda_{i_m} { d_{i_{m+1}}}$ can be large while the scaling factor 
$\Lambda_N = \lambda_{i_1}\cdots \lambda_{i_m}$ is tiny, and that the probability of this happening is not too small because of the self-similar structure of $X$. These will be explained in more detail in individual sections below.\\

Here is how the paper is organized. Theorem \ref{infinite differentiability for compactly supported test function} is proved in \S \ref{Section: Infinite Differentiability for Smooth Compactly Supported Test Function}. The two examples on non-differentiability are constructed in \S \ref{Section: Examples of Non-Differentiability}, where part of the construction relies on Theorem \ref{infinite differentiability for compactly supported test function}. Then we prove Theorem \ref{Differentiability of Stationary measure}  for $r=1$ and Theorem \ref{Smoothness for Finite Moment Condition} in \S 
\ref{Section: Differentiability via Moment Condition and Bound on Derivative }
and \S \ref{Section: Smoothness of Moments}. 
The higher derivatives in Theorem \ref{Differentiability of Stationary measure}
are computed in \S \ref{ScHD}.
 \S \ref{Section: Extensions and open problems} contains discussion 
of extensions and open problems.

\input{Infinite_Differentiability_for_SCSO}
\input{Examples_of_Non-Differentiability}
\input{Differentiability_via_Moment_and_Derivative}

\input{Smoothness_of_Moments}

\input{Extensions_and_open_problems}

\printbibliography[title = {\centering{References}}]

\end{document}

%% file: Infinite_Differentiability_for_SCSO.tex
\section{Infinite Differentiability for Smooth, Compactly Supported Test Function}\label{Section: Infinite Differentiability for Smooth Compactly Supported Test Function}

\subsection{Outline of the Proof}
In this section  we prove Theorem \ref{infinite differentiability for compactly supported test function}
which establish the differentiability of \eqref{SMAv}
when the test function $\phi$ is smooth and compactly supported. 



Recall that $\mu_{\lambda_1+\epsilon,\lambda_2}$ is the stationary measure for the IFS 
$$\Phi(\epsilon) = \left\{f_1(x) = (\lambda_1 +\epsilon)x+d_1, \quad f_2(x) = \lambda_2x+d_2;\quad p = \left(\frac{1}{2},\frac{1}{2}\right)\right \}.$$
Before starting the proof, we note that we may assume $d_1\!=\!d_2\!=\!1$ without loss of generality. Indeed,  define $f_1(x) = \lambda_1 x+1,f_2(x)=\lambda_2x +1$ and let $\tilde{f}_i(x) = \lambda_i x+d_i,i=1,2$ be any other nontrivial IFS (recall that an IFS is nontrivial if the functions in question do not have a common fixed point). That is, $\frac{d_1}{1-\lambda_1}\neq \frac{d_2}{1-\lambda_2}$. Then a simple calculation yields that the function $c:\mathbb{R}\to \mathbb{R},x\mapsto ax+b$ 
with
$$ a:=\frac{(1-\lambda_1)d_2 -(1-\lambda_2)d_1}{\lambda_2-\lambda_1},\quad 
b:=\frac{d_2-d_1}{\lambda_1-\lambda_2}$$
satisfies $c\circ f_i = \tilde{f}_i\circ c,i=1,2$. If we write $\tilde{\Pi}:\Sigma\to \mathbb{R}$ for the symbolic projection of the IFS $\{\tilde{f}_1,\tilde{f}_2\}$, then for 
almost every
$i=(i_1,i_2\cdots)\in \Sigma$,
\begin{align*}
    \tilde{\Pi}(i) &= \lim_{n\to \infty} \tilde{f}_{i_1} \circ\cdots\circ \tilde{f}_{i_n}(0) =\lim_{n\to \infty} \tilde{f}_{i_1} \circ\cdots\circ \tilde{f}_{i_n}(c(0)) \\
    &= \lim_{n\to \infty} c\circ f_{i_1} \circ \cdots \circ f_{i_n}(0) 
    = c(\lim_{n\to \infty}  f_{i_1} \circ \cdots \circ f_{i_n}(0) )
    =c\circ \Pi (i),
\end{align*}
where $\Pi:\Sigma\to \mathbb{R}$ is the symbolic projection of $\{f_1,f_2\}$. Consequently, if we write $\tilde{\mu} = \tilde{\mu}_{\lambda_1,\lambda_2}$ for the stationary measure of $\{(\tilde{f}_1,\tilde{f}_2),p=(\frac{1}{2},\frac{1}{2})\}$ and $\mu = \mu_{\lambda_1,\lambda_2}$ for that of $\{(f_1,f_2),p=(\frac{1}{2},\frac{1}{2})\}$, then for any test function  $\phi\in C_c^{\infty}(\mathbb{R})$, 
$$\int_{\mathbb{R}} \phi(x)\,d\tilde{\mu}(x) = \int_{\Sigma} \phi(\tilde{\Pi}(i))\,d\nu = \int_{\Sigma} \phi\circ c(\Pi(i))\,d\nu = \int_{\mathbb{R}} \phi\circ c (x)\,d\mu(x) = \int_{\mathbb{R}}\phi(ax+b)\,d\mu(x).$$
That is, the integral of $\tilde{\mu}$ against a test function $\phi$ is the same as the integral of $\mu$ 
against the test function $\phi\circ c$. 

Clearly if $\phi\in C^\infty_{c}(\mathbb{R})$, then so is $\phi\circ c$. Moreover 
$a,b$ are smooth functions in $\lambda_1,\lambda_2$. It is thus not difficult to see that, if the conclusion of Theorem \ref{infinite differentiability for compactly supported test function} is true for $\{(f_1,f_2),p=(\frac{1}{2},\frac{1}{2})\}$, then for  $\{(\tilde{f}_1,\tilde{f}_2),p=(\frac{1}{2},\frac{1}{2})\}$,

\begin{align*}
    \frac{d}{d\lambda_1} \int_\mathbb{R} \phi(x)\,d\tilde{\mu}_{\lambda_1,\lambda_2} & = \frac{d}{d\lambda_1}\int_{\mathbb{R}} \phi(ax+b)\,d\mu_{\lambda_1,\lambda_2} \\
    &= \int_{\mathbb{R}} \phi'(ax+b) (\frac{\partial a}{\partial\lambda_1}+\frac{\partial b}{\partial\lambda_1})\,d\mu_{\lambda_1,\lambda_2}(x) + \frac{d}{d\lambda_1} \int_{\mathbb{R}} \phi\circ c\,d\mu_{\lambda_1,\lambda_2}. 
\end{align*}

Indeed, the second equality is due to chain rule and that interchanging derivative and integral is valid, as $\phi'(ax+b) (\frac{\partial a}{\partial\lambda_1}+\frac{\partial b}{\partial\lambda_1})$ is uniformly bounded for $\lambda_1,\lambda_2$ in a small neighborhood. Moreover, the term $\DS  \frac{d}{d\lambda_1} \int_{\mathbb{R}} \phi\circ c\,d\mu_{\lambda_1,\lambda_2}$ is well-defined because we assume Theorem \ref{infinite differentiability for compactly supported test function} holds with $d_1=d_2=1$.

With the above discussion, let us assume $d_1=d_2=1$ till the end of this section. Then $\mu_{\lambda_1+\epsilon , \lambda_2}$ is the law of
$$X(\epsilon) =  \sum_{m=0}^{\infty} \lambda_{i_1 }(\epsilon)\cdots \lambda_{i_m}(\epsilon).$$
where $\lambda_1(\epsilon):= \lambda_1 +\epsilon , \lambda_2(\epsilon): = \lambda_2$.

The proof will be induction on $l$ and the inductive step consists of two parts: first, we approximate the actual integral by finite truncate, and then we estimate the difference of finite approximation to yield the derivative. 

More precisely, the base case $l=0$ is trivial. Assume that the theorem is true for some $l\ge 0$. 
The following two lemmas to carry out the inductive step.

    Denote by $X_n$ and $X_n^{(j)}$'s the finite truncates of $X,X^{(j)}$'s at $n$. That is,
    $$X_n =  \sum_{m=1}^n \lambda_{i_1 }\cdots \lambda_{i_m},\quad 
    X_n^{(j)} = \frac{j!}{\lambda_1^j}\sum_{m=1}^{n} \lambda_{i_1} \cdots \lambda_{i_m} \binom{o(m)}{j} .$$

\begin{lm}\label{finite approximation}
    There exists $c>0 $ and $0<\theta<1$ that depends on $\lambda_1,\lambda_2$, $\phi$ and $l$ such that 
    \begin{align*}
        |&\int \sum_{k_1 + 2k_2 + \cdots lk_l = l} L(k_1, \cdots, k_l) \phi^{(k)} (X) \prod_{j=1}^{l} (X^{(j)})^{k_j} \\
        &-\int \sum_{k_1 + 2k_2 + \cdots lk_l = l} L(k_1, \cdots, k_l) \phi^{(k)} (X_n) \prod_{j=1}^{l} (X^{(j)}_n)^{k_j}| < c\theta ^n.
    \end{align*}
\end{lm}

\begin{lm}\label{second step:finite difference converges to derivative}
    For any $\epsilon\neq0$, let $n = n(\epsilon)\in\mathbb{N} $ be such that there exists some constant d satisfying $\frac{1}{d} \log\frac{1}{|\epsilon|} \le n \le d \log\frac{1}{|\epsilon|},\forall \epsilon\neq0$. Then the following limit holds:
    \begin{align}
        \lim_{\epsilon\to 0} \frac{1}{\epsilon} (&\int \sum_{k_1 + 2k_2 + \cdots lk_l = l} L(k_1, \cdots, k_l) \phi^{(k)} (X_n(\epsilon)) \prod_{j=1}^{l} (X^{(j)}_n (\epsilon))^{k_j}\\
        &-\int \sum_{k_1 + 2k_2 + \cdots lk_l = l} L(k_1, \cdots, k_l) \phi^{(k)} (X_n) \prod_{j=1}^{l} (X^{(j)}_n)^{k_j}) \\
        &=\int \sum_{k_1 + 2k_2 + \cdots (l+1)k_l = l+1} L(k_1, \cdots, k_{l+1}) \phi^{(k)} (X) \prod_{j=1}^{l+1} (X^{(j)})^{k_j}.
    \end{align}
    Here $X_n (\epsilon):= \DS \sum_{m=1}^n \lambda_{i_1}(\epsilon) \cdots \lambda_{i_m}(\epsilon)$ and 
    $\DS X_n^{(j)} (\epsilon) := \frac{\,d^j}{\,d \epsilon^j}X_n (\epsilon) =  \frac{j!}{\lambda_1^j}\sum_{m=1}^{n} \lambda_{i_1}(\epsilon) \cdots \lambda_{i_m}(\epsilon) \binom{o(m)}{j}.$ 
 \end{lm}

The remainder of this section is organized as follows. We first list several preliminary lemmas in \S \ref{Smoothness for Compactly Suppoerted Functions; Subsection; Preliminary}. Then we prove Lemma \ref{finite approximation} and \ref{second step:finite difference converges to derivative} in \S \ref{Smoothness for Compactly Suppoerted Function;Subsection;Proof of Finite approximation Lemma} and \S \ref{Smoothness for Compactly Supported Function; Subsection; Proof of lemma on Convergence of finite difference} respectively. The last subsection \S\ref{SSPfCompact}
finishes the proof of Theorem \ref{infinite differentiability for compactly supported test function}.

\subsection{Preliminaries}\label{Smoothness for Compactly Suppoerted Functions; Subsection; Preliminary}
We list several elementary lemmas that will be used later.

\begin{lm}[Generalized Dominated Convergence Theorem]\label{Generalized Dominated Convergence Lemma}
  Let $\{f_n\},\{g_n\}$ be two sequences of measurable functions in a measure space $(E,\mathfrak{B},\nu)$. Suppose that $|f_n|\le g_n\forall n$, and $g_n$ converges both a.e. and in $L^1(E,\nu)$. Then if $f_n$ converges a.e., it also converges in $L^1(E,\nu)$.
\end{lm}
This is a standard fact from real analysis and we omit its proof.

\begin{lm}\label{Convergence of Product of Two Functions}
      Let $\{f_n\},\{g_n\}$ be two sequences of measurable functions in a measure space $(E,\mathfrak{B},\nu)$. Suppose 
      \begin{itemize}
          \item $\{f_n\}$ converges a.e. and there exists $M>0$ such that $|f_n|\le M ,\forall n$. 
          \item $\{g_n\}$ converges in $L^1$ to $g\in L^1(E,\nu)$.
      \end{itemize}
      Then $f_n g_n\to fg$ in $L^1$.
\end{lm}
\begin{proof}[Proof of Lemma \ref{Convergence of Product of Two Functions}]
    We have 
    $\DS
        \int|f_n g_n-f g| \le \int |f_n (g_n-g)| + \int |(f_n - f)g|.
    $
    The first integral converges to $0$ because
    $\DS \int |f_n (g_n-g)| \le  M ||g-g_n||_{L^1} \to 0.$
   The second integral converges to $0$ by the dominated convergence theorem.
\end{proof}

We need a result from large deviation theory.
 Define $\DS f(s) = \frac{\lambda_1^s+\lambda_2^s}{2}$, $x\ge 0$. A simple analysis using the convexity and asymptotics at infinity
 of $f$ implies that there exists $s_1>0$ such that $f$ is decreasing on $[0,s_1)$ and increasing on $[s_1,\infty)$. Since $f(0) = 1$ and $f(\infty) = \infty$, there exists a unique $s_0>s_1$ such that $f(s_0) = 1$.

\begin{lm}\label{Lemma decay of tail}

        There is $c \!\!=\!\! c(\lambda_1,\lambda_2)>0$ such that  for any $R>0$, one has $\mathbb{P}(X\ge R) < cR^{-s_0}.$

    \end{lm}
Lemma \ref{Lemma decay of tail} is proved in, for example,  \cite{guivarc2016spectral}.

\subsection{Proof of Lemma \ref{finite approximation}}\label{Smoothness for Compactly Suppoerted Function;Subsection;Proof of Finite approximation Lemma}
Note that in the left hand side of the target inequality, the number of summands is bounded by some constants depending on $l$ only, and the coefficient $L(k_1,\cdots,k_l)$ is also uniformly bounded in terms of $l$. Therefore, it suffices to deal with terms of the form
$$|\int  \phi^{(k)} (X) \prod_{j=1}^{l} (X^{(j)})^{k_j} -\int  \phi^{(k)} (X_n) \prod_{j=1}^{l} (X^{(j)}_n)^{k_j}| ,$$

    which is less or equal to
   $$
        |\int  \phi^{(k)} (X) (\prod_{j=1}^{l} (X^{(j)})^{k_j}-\prod_{j=1}^{l} (X^{(j)}_n)^{k_j})| 
        +|\int  (\phi^{(k)} (X)-\phi^{(k)} (X_n)) \prod_{j=1}^{l} (X^{(j)}_n)^{k_j}|
        := J_1 + J_2 . $$
    We first estimate $J_1$. 

    Note that $\displaystyle X^{(j)} - X_n^{(j)} = \frac{j!}{\lambda_1^j} \sum_{m=n+1}^{\infty} \lambda_{i_1}\cdots \lambda_{i_m}\binom{o(m)}{j} =: R_n^{(j)}$, and we compute
    \begin{align*}
        &\prod_{j=1}^{l} (X^{(j)})^{k_j} - \prod_{j=1}^{l} (X_n^{(j)})^{k_j} =\prod_{j=1}^{l} (X_n^{(j)} + R_n^{(j)})^{k_j} - \prod_{j=1}^{l} (X_n^{(j)})^{k_j}\\
        &=(\sum_{s_1=0}^{k_1} \binom{k_1}{s_1} (X_n^{(1)})^{k_1-s_1} (R_n^{(1)})^{s_1} )\cdots(\sum_{s_l=0}^{k_l} \binom{k_l}{s_l} (X_n^{(l)})^{k_l-s_l} (R_n^{(l)})^{s_l} ) - \prod_{j=1}^{l} (X_n^{(j)})^{k_j}\\
        &= \sum_{\begin{matrix}
            s_1,\cdots,s_l \\
            0\le s_i \le k_i,\forall i\\
            s_1 + \cdots +s_l >0
        \end{matrix}}  \prod_{j=1}^{l}\binom{k_j}{s_j} (X_n^{(j)})^{k_j-s_j}(R^{(j)}_n)^{s_j}
    \end{align*}
  Again, note that in the last summation, the number of summands is uniformly bounded by some constant depending on $l$, and the coefficients $\displaystyle \prod_{j=1}^{l}\binom{k_j}{s_j}$ are also uniformly bounded in terms of $l$. Therefore we can further reduce  the problem to  estimating
  $$\int \phi^{(k)}(X) \prod_{j=1}^{l} (X_n^{(j)})^{k_j-s_j}(R^{(j)}_n)^{s_j}.$$

  Since $\phi$ is compactly supported, there exists $M\ge 0$ such that supp$\phi\subset[-M,M]$.  Hence the above integral is in fact
  \begin{equation}
      \int \phi^{(k)}(X) \prod_{j=1}^{l} (X_n^{(j)})^{k_j-s_j} (R^{(j)}_n)^{s_j} 1_{\{X\le M\}}
  \end{equation}
  where $1_A$ is the characteristic function of set $A$.

  Note that 
  \begin{align}\label{estimate 1}
      \phi^{(k)}(X)1_{\{X\le M\}} X^{(j)}_n &= \phi^{(k)}(X)\frac{j!}{\lambda_1^j}\sum_{m=1}^{n}\lambda_{i_1}\cdots \lambda_{i_m} \binom{o(m)}{j} 1_{\{X\le M\}}
      \\& \notag
      \le   ||\phi^{(k)}||_{C^0}\frac{j!}{\lambda_1^j}n^jM { \le} K n^l .
  \end{align}
  
  Here we used the fact that $\binom{o(m)}{j}\le (o(m))^j \le n^j$ and $K$ is a constant depending only on $\phi$ and $l$.

  To estimate the remaining terms, we have for
  \begin{align}\label{estimate 2}
      & \prod_{j=1}^{l} (R^{(j)}_n)^{s_j} =  \prod_{j=1}^{l} \left(\frac{j!}{\lambda_1^j}\sum_{m=n+1}^{ \infty}\lambda_{i_1} \cdots \lambda_{i_m}\binom{o(m)}{j}\right)^{s_l} \\
      &\le L  \prod_{j=1}^l(\sum_{m=n+1}^{ \infty} C(m) \lambda_{i_1}\cdots \lambda_{i_m})^{s_j}  
      = L(\sum_{m=n+1}^{\infty} C(m) \lambda_{i_1}\cdots \lambda_{i_m}) ^s 
  \end{align}
  where $\DS C(m) := \sum_{j=1}^{l} \binom{o(m)}{j}$, $s:=s_1+\cdots +s_l $. 
  
  We wish to show exponential smallness of $\DS  \mathbb{E}\left[\left(\sum_{m=n+1}^\infty C(m) \lambda_{i_1}\cdots\lambda_{i_m}\right)^s\right]$. To this end for any large $N\in\mathbb{N}$, we will find a uniform upper bound for 
  $\DS \mathbb{E}\left[\left(\sum_{m=n+1}^N C(m) \lambda_{i_1}\cdots\lambda_{i_m}\right)^s\right]$ that is exponentially small.
  We need to distinguish the cases $s=1$  and $s\ge 2$. Below we do the case $s\ge 2$. The case $s=1$ is similar but easier,and we will sketch its proof after we finish the case $s\ge 2$.

  Let $s^*>1$ be such that $\frac{1}{s}+\frac{1}{s^*}=1$.  By H\"older inequality, for any $0<\rho<1$
  \begin{align}\label{estimate 3}
     & (\sum_{m=n+1}^{N} C(m) \lambda_{i_1}\cdots \lambda_{i_m}) ^s \\
     &\le (\sum_{m=n+1}^{N}(\rho^m)^{s^*})^{\frac{s}{s^*}} \left[\sum_{m=n+1}^{N}
     \left(\frac{\lambda_{i_1}}{\rho}\right)^s \cdots\left(\frac{\lambda_{i_m}}{\rho}\right)^s(C(m))^s\right] \\
    & \le I(\rho) \left[C^s(m) \sum_{m=n+1}^{N}
    \theta_{i_1} \cdots\theta_{i_m} \right],
  \end{align}
where $\DS I(\rho) := \left(\sum_{m=n+1}^{\infty}(\rho^m)^{s^*}\right)^{\frac{s}{s^*}} <\infty$ and $\DS \theta_{i}:= \frac{\lambda_{i}^s}{\rho^s}$, $t \in\mathbb{N}$.

\textbf{Claim:} one can choose $0<\rho<1$ so that there exists $C,\gamma>0$ satisfying

\begin{equation}\label{Equation: Varadhan's Type of Estimate}
   \nonumber \mathbb{E}[\theta_{i_1} \cdots \theta_{i_m} 1_{\{X\le M\}}] \le  Ce^{-\gamma m}.
\end{equation}

  \textbf{Proof of Claim:} Let $\kappa = \kappa(\rho)$ be a positive real to be specified later. We have
  \begin{align*}
      &\mathbb{E}[\theta_{i_1} \cdots \theta_{i_m} 1_{\{\theta_{i_1}\cdots\theta_{i_m} \le e^{\kappa m}\}}]\\
      & = \mathbb{E}[\theta_{i_1} \cdots \theta_{i_m} 1_{\{\theta_{i_1}\cdots\theta_{i_m} \le e^{-\kappa m}\}}] + \mathbb{E}[\theta_{i_1} \cdots \theta_{i_m} 1_{\{e^{-\kappa m} <\theta_{i_1}\cdots\theta_{i_m} \le e^{\kappa m}\}}] \\
      &\le  e^{-\kappa m} + e^{\kappa m} \mathbb{P}\left(\frac{\sum_{t=1}^{m}\log \theta_{i_t}}{m} \ge -\kappa\right)
  \end{align*}
  Note that $\theta_{i_t}$ are i.i.d with $\mathbb{E}[\log\theta_{i_t}] = \frac{s}{2} \log \lambda_1 \lambda_2  +s\log\frac{1}{\rho}$. Thus by standard large deviation theory,  if  $\rho$ is close enough to $1$ and $\kappa$ is so small that $-\kappa > \mathbb{E}[\log\theta_{i_t}]$, then there exists $\delta>0$ depending on $-\kappa -( \frac{s}{2} \log \lambda_1 \lambda_2  +s\log\frac{1}{\rho} )$ such that
  \begin{equation}\label{large deviation}
      \mathbb{P}\left(\frac{\sum_{t=1}^{m}\log \theta_{i_t}}{m} \ge -\kappa\right) \le e^{-\delta m}.
  \end{equation}
Therefore, 
\begin{equation}
    \mathbb{E}[\theta_{i_1} \cdots \theta_{i_m} 1_{\{\theta_{i_1}\cdots\theta_{i_m} \le e^{\kappa m}\}}] \le e^{-\kappa m} + e^{-(\delta - \kappa) m}.
\end{equation}
Note that $\delta$ increases as $\kappa$ decreases. Therefore, we can first fix a set of $\kappa$ and $\rho$ satisfying
$$ -\kappa -\left( \frac{s}{2} \log \lambda_1 \lambda_2  +s\log\frac{1}{\rho} \right) >0$$
and then shrink $\kappa$ to make $ -\delta+\kappa < \frac{\kappa}{2}$, because inequality \eqref{large deviation}
 remains valid with the same $\delta$ as $\kappa$ becomes smaller.    

Then
$$ \mathbb{E}[\theta_{i_1} \cdots \theta_{i_m} 1_{\{\theta_{i_1}\cdots\theta_{i_m} \le e^{\kappa m}\}}] \le e^{-\kappa m} + e^{-(\delta - \kappa) m} \le 2e^{-\frac{\kappa}{2} m} := 2e^{-\gamma m}.$$

Also note that, $\delta$ increases as $\rho$ increases, so if we can fix the above choice of $\kappa$ and increase $\rho$ to make $e^\kappa \rho^s >1$, while the large deviation inequality \eqref{large deviation} remains true, so for this $\rho$ we still have
\begin{equation}
    \mathbb{E}[\theta_{i_1} \cdots \theta_{i_m} 1_{\{\theta_{i_1}\cdots\theta_{i_m} \le e^{\kappa m}\}}] \le  2e^{-\gamma m}.
\end{equation}
with the same $\gamma$.

We are now ready to prove the claim. First we show that $X\le M$ implies 
$$\theta_{i_1} \cdots \theta_{i_m} \le e^{\kappa m}, \forall m\ge n_0 := [s \log_{e^\kappa \rho^s} M]+1.$$
Indeed, if there exists $m>n_0$ with $\frac{\lambda^s_{i_1}}{\rho^s}\cdots \frac{\lambda^s_{i_m}}{\rho^s}=\theta_{i_1} \cdots \theta_{i_m} > e^{\kappa m}$, then
\begin{equation}\label{bounded gives finite expectation}
    M> \lambda_{i_1} \cdots \lambda_{i_m} > (e^\kappa \rho^s)^{\frac{m}{s}}> M
\end{equation}

a contradiction. 

Therefore, for all $m\ge n_0$, 
\begin{equation}
     \mathbb{E}[\theta_{i_1} \cdots \theta_{i_m} 1_{\{X\le M\}}] 
    \le \mathbb{E}[\theta_{i_1} \cdots \theta_{i_m} 1_{\{\theta_{i_1}\cdots\theta_{i_m} \le e^{\kappa m}\}}] \le  2e^{-\gamma m}
\end{equation}
and the claim is proved.

Now fix a $\rho$ as in the above claim. For all $n\ge n_0$,
\begin{align}\label{estimate 4}
    \int\sum_{m=n+1}^{N}\left(\frac{\lambda_{i_1}}{\rho}\right)^s \cdots
    \left(\frac{\lambda_{i_m}}{\rho}\right)^s (C(m))^s 1_{\{X\le M\}} &= C(m)^s \mathbb{E}[\sum_{m=n+1}^N\theta_{i_1} \cdots \theta_{i_m}  1_{\{X\le M\}}] \\
    \le 2\sum_{m=n+1}^{N}C(m)^s e^{-\gamma m} &
    \le e^{-\gamma' n} \quad \text{for some } 0<\gamma' < \gamma,
\end{align}
where in the last inequality we used the fact that $C(m) = \mathcal{O}(m^{l+1})$ . Note that $\gamma'$ depends on $s = s_1+\cdots +s_l$, which is bounded by $l$. Thus we may choose a uniform $\gamma'$ for all choices of $s_1,\cdots s_l$.

In the case $s=1$, we can skip the H\"older inequality step, and directly show  that there exists $C = C(\lambda_1,\lambda_2),\kappa = \kappa(\lambda_2,\lambda_2)>0$ such that for any $m\in\mathbb{N}$, 
$$\mathbb{E}[\lambda_{i_1}\cdots \lambda_{i_m} 1_{\{\lambda_{i_1}\cdots \lambda_{i_m} \le e^{\kappa m} \}} ]\le Ce^{-\frac{\kappa}{2}m}.$$
For any $m\in\mathbb{N}$ with $m> \frac{1}{\kappa}\log M$, $X\le M$ implies $\lambda_{i_1}\cdots\lambda_{i_m} \le e^{\kappa m}$, so that 
\begin{equation*}
     \mathbb{E}[\lambda_{i_1} \cdots \lambda_{i_m} 1_{\{X\le M\}}] 
    \le \mathbb{E}[\lambda_{i_1} \cdots \lambda_{i_m} 1_{\{\lambda_{i_1}\cdots\lambda_{i_m} \le e^{\kappa m}\}}] \le  2e^{-\frac{\kappa}{2} m}.
\end{equation*}

Therefore, for any $N>n>\frac{1}{\kappa}\log M$, we have

\begin{align}\label{Estimate 5}
  \int\!\!\!\!\sum_{m=n+1}^{N}\lambda_{i_1} \cdots
   \lambda_{i_m} C(m) 1_{\{X\le M\}} &= C(m) \mathbb{E}\left[\sum_{m=n+1}^N\lambda_{i_1} \cdots \lambda_{i_m}  1_{\{X\le M\}}\right] \\
    &\nonumber \le C\!\!\!\sum_{m=n+1}^{N}C(m) e^{-\frac{\kappa}{2} m} 
    \le e^{-\gamma' n} \quad 
\end{align}
for some  $0<\gamma' < \frac{\kappa}{2}$ which also gives exponential smallness in the case $s=1$.

Combining \eqref{estimate 1},\eqref{estimate 2},\eqref{estimate 3},\eqref{estimate 4}  and letting $N\to \infty$, we have 
\begin{equation}
    |\int \phi^{(k)}(X) \prod_{j=1}^{l} (X_n^{(j)})^{k_j-s_j}(R^{(j)}_n)^{s_j}|  =|\int \phi^{(k)}(X) \prod_{j=1}^{l} (X_n^{(j)})^{k_j-s_j}(R^{(j)}_n)^{s_j}1_{\{X\le M\}}|
    \end{equation}
    $$ \le Kn^l |\int \prod_{j=1}^{l} (R^{(j)}_n)^{s_j} 1_{\{X\le M\}}| 
    \le Kn^l  e^{-\gamma' n} 
    \le e^{-\gamma'' n} \quad \text{ for some } 0<\gamma'' <\gamma' .
$$
    Here the same $K$ may refer to different constants that are independent of $n$. And the same convention applies to the rest part of the proof.
    
    Consequently 
    \begin{equation}\label{finite approximation; intermediate step; 1}
        |J_1| \le  A e^{-\gamma'' n},
    \end{equation}
    where $A$ is a constant that depends on $\phi,l,\lambda_1,\lambda_2$. 

    Next we turn to $J_2$. Observe that 
    \begin{equation}
         X = X_n + \lambda_{i_1} \cdots \lambda_{i_n} ( \lambda_{i_{n+1}} + \lambda_{i_{n+1}}\lambda_{i_{n+2}} +\cdots) =: X_n + \Lambda_n Y
    \end{equation}
       and $Y$ has the same distribution as $X$. In particular, $Y$ also satisfies the tail property of  Lemma \ref{Lemma decay of tail}. For any $R>0$, we have
      $$
       \int  (\phi^{(k)} (X)-\phi^{(k)} (X_n)) \prod_{j=1}^{l} (X^{(j)}_n)^{k_j} = \int  (\phi^{(k)} (X_n+\Lambda_n Y)-\phi^{(k)} (X_n)) \prod_{j=1}^{l} (X^{(j)}_n)^{k_j}=: \Gamma_1 + \Gamma_2 .$$
       where 
       $\displaystyle \Gamma_j=\int_{\{\mathcal{Y}_j\}} \int_{\mathcal{A}^n} (\phi^{(k)} (X_n+\Lambda_n Y)-\phi^{(k)} (X_n)) \prod_{j=1}^{l} (X^{(j)}_n)^{k_j} \,dp^n \,d\nu_{n+1}
       $.  

\noindent
Here $R$ is a large constant chosen later,
$\mathcal{Y}_1=\{Y\geq R\}$, $\mathcal{Y}_2=\{Y< R\}$ are subsets of
$\Sigma_{n+1}:= \{(i_{n+1}, i_{n+2} , \cdots)| i_j\in \mathcal{A} = \{1,2\},\forall j\ge n+1\}$ and we write $\nu_{n+1}$ for the $(\frac{1}{2},\frac{1}{2})$-Bernoulli measure on $\Sigma_{n+1}$ .
Now
\begin{align*}
   &\int_{\mathcal{A}^n} (\phi^{(k)} (X_n+\Lambda_n Y)-\phi^{(k)} (X_n)) \prod_{j=1}^{l} (X^{(j)}_n)^{k_j} \,dp^n \quad p^n = \left(\frac{1}{2},\frac{1}{2}\right)^n, \text{ the product measure on } \mathcal{A}^n  \\
   =&\int_{\mathcal{A}^n} (\phi^{(k)} (X_n+\Lambda_n Y)-\phi^{(k)} (X_n)) \prod_{j=1}^{l} (X^{(j)}_n)^{k_j} 1_{\{X\le M \text{ or } X_n \le M\} }\,dp^n\\
  \le  & 2||\phi^{(k)}||_{C^0}\int_{\mathcal{A}^n}  \prod_{j=1}^{l} (X^{(j)}_n)^{k_j}1_{\{X\le M \text{ or } X_n \le M\} } \,dp^n  \le D \sum_{m=1}^n m^{l+1} e^{-\alpha m} \quad \text{ for some }\alpha>0 \\
  \le & D(\phi,l,\lambda_1 , \lambda_2)
\end{align*}
where in the penultimate inequality, we use the same argument as for estimating $\DS \prod_{j=1}^{l} (R^{(j)}_n)^{s_j}$.
Namely

\begin{align*}
    |\prod_{j=1}^{l} (X^{(j)}_n)^{k_j}1_{\{X\le M \text{ or } X_n \le M\} }| &\le  L  \prod_{j=1}^l\left(\sum_{m=1}^{n} C(m) \lambda_{i_1}\cdots \lambda_{i_m}\right)^{s_j}   1_{\{X\le M \text{ or } X_n \le M\} }\\
      &= L\left(\sum_{m=1}^{n} C(m) \lambda_{i_1}\cdots \lambda_{i_m}\right) ^s   1_{\{X\le M \text{ or } X_n \le M\} }
\end{align*}
where $\DS C(m) := \sum_{j=1}^{l} \binom{o(m)}{j}$, $s:=s_1+\cdots +s_l $ and $L:=\frac{l!}{\lambda_1^l}$.  Taking the expectation yields

\begin{align*}
    &\int_{\mathcal{A}^n}  \prod_{j=1}^{l} (X^{(j)}_n)^{k_j}1_{\{X\le M \text{ or } X_n \le M\} } \,dp^n \le  L \sum_{m=1}^n \mathbb{E} [ C(m) \lambda_{i_1}^s\cdots \lambda_{i_m} ^s   1_{\{X\le M \text{ or } X_n \le M\} }]\\
  &  \le D\sum_{m=1}^{n} m^{l+1} e^{-\alpha m} \quad \text{for some }\alpha>0, D>0
\end{align*}
where the last inequality will follow from the same  argument we used to derive 
\eqref{Estimate 5} and \eqref{estimate 4}. We only need to notice  here the condition $1_{\{X\le M \text{ or }X_n \le M\}}$ works equally well as $1_{\{X\le M \}}$ because we are only dealing with the first $n$ terms.

Thus \begin{equation}\label{finite approximation; intermediate step; 2}
     |\Gamma_1| \le  D \mathbb{P}(Y\ge R) \le D R^{-s}.
\end{equation}
 where $s=s_0(\lambda_1 , \lambda_2)>0$ comes from Lemma \ref{Lemma decay of tail}. 

Next, we divide $\Gamma_2$ into two parts.
\begin{align*}
    &\int_{\{Y< R\}} \int_{\mathcal{A}^n} (\phi^{(k)} (X_n+\Lambda_n Y)-\phi^{(k)} (X_n)) \prod_{j=1}^{l} (X^{(j)}_n)^{k_j} \,dp^n \,d\nu_{n+1} \\
    =& \int_{\{Y< R\}} \int_{\{\Lambda_n \le r^n\}} (\phi^{(k)} (X_n+\Lambda_n Y)-\phi^{(k)} (X_n)) \prod_{j=1}^{l} (X^{(j)}_n)^{k_j} \,dp^n \,d\nu_{n+1} \\
    +& \int_{\{Y< R\}} \int_{\{\Lambda_n > r^n\}} (\phi^{(k)} (X_n+\Lambda_n Y)-\phi^{(k)} (X_n)) \prod_{j=1}^{l} (X^{(j)}_n)^{k_j} \,dp^n \,d\nu_{n+1}\\
    =:& Q_1 + Q_2
\end{align*}
Here $r$ is any fixed real number satisfying  $(\lambda_1 \lambda_2)^{\frac{1}{2}}<r<1$.  Then there exists $\eta = \eta(r),c = c(r)>0$ so that $\mathbb{P}(\Lambda_n > r^n) \le ce^{-\eta n}$.

We have
\begin{equation}\label{finite approximation; intermediate step; 3}
    |Q_1| \le \int_{\{Y< R\}} \int_{\{\Lambda_n \le r^n\}} ||\phi^{(k+1)}||_{C^0} Rr^n |\prod_{j=1}^{l} (X^{(j)}_n)^{k_j} 1_{\{X\le M \text{ or } X_n \le M\}} \,dp^n \,d\nu_{n+1}
    \le DRr^n
\end{equation}
where the last inequality holds for the same reason as when we estimated $\Gamma_1$.

Moreover, \begin{align}\label{finite approximation; intermediate step; 4}
    |Q_2| &\le \int_{\{Y< R\}} \int_{\{\Lambda_n > r^n\}} 2||\phi^{(k)}||_{C^0} \prod_{j=1}^{l} (X^{(j)}_n)^{k_j} 1_{\{X\le M \text{ or } X_n \le M\}} \,dp^n \,d\nu_{n+1} \\
    &\le D n^{l+1} \mathbb{P}(\Lambda_n >r^n) 
    \le Dn^{l+1} e^{-\eta n}.
\end{align}

Summing \eqref{finite approximation; intermediate step; 1}, \eqref{finite approximation; intermediate step; 2},
\eqref{finite approximation; intermediate step; 3} and \eqref{finite approximation; intermediate step; 4} we have
\begin{align*}
    |&\int \!\!\!\! \sum_{k_1 + 2k_2 + \cdots lk_l = l} \!\!\!\! \!\!\!\! L(k_1, \cdots, k_l) \phi^{(k)} (X) \prod_{j=1}^{l} (X^{(j)})^{k_j} 
        - \int \!\!\!\! \sum_{k_1 + 2k_2 + \cdots lk_l = l} \!\!\!\!\!\!\!\! L(k_1, \cdots, k_l) \phi^{(k)} (X_n) \prod_{j=1}^{l} (X^{(j)}_n)^{k_j}|\\
        &\le D(n^{l+1} e^{-\gamma'' n} + Rr^n + R^{-s} + e^{-\eta n}), \label{finite approximation; last but one step}
\end{align*}
    where $D$ depends only on $\phi , l,\lambda_1\lambda_2$.  

    Now fix $R = \beta^{-\frac{n}{s}}$ where $\beta$ is any real number satisfying $0<\beta<1$ and $r \beta^{-\frac{1}{s}}<1$. Then fix $0<\theta<1$ so that
    $$\theta > \max\{ e^{-\gamma''}, e^{-\eta}, \beta , r\beta^{-\frac{1}{s}}\}.$$
    
    Then there exists $c>0$ depending only on $\phi,l,\lambda_1 \lambda_2$ so that 
\begin{align*}
     \left|\int \sum_{k_1 + 2k_2 + \cdots lk_l = l} L(k_1, \cdots, k_l) \left[\phi^{(k)} (X) \prod_{j=1}^{l} (X^{(j)})^{k_j} 
     -\phi^{(k)} (X_n) \prod_{j=1}^{l} (X^{(j)}_n)^{k_j}\right]\right|
        &\le c \theta^n.
\end{align*}

The lemma is proven.\qed

\subsection{Proof of Lemma \ref{second step:finite difference converges to derivative}}\label{Smoothness for Compactly Supported Function; Subsection; Proof of lemma on Convergence of finite difference}

    Since the summation $\sum_{k_1 + 2k_2 + \cdots lk_l = l}$ and the coefficients $L(k_1, \cdots, k_l)$ are governed by $l$, which is fixed, it suffices to prove that on each summand, the limit can be computed by taking the formal derivative:
    \begin{equation}\label{finite difference;intermediate step; 1}
          \lim_{\epsilon\to 0} \frac{1}{\epsilon} (\int  \phi^{(k)} (X_n(\epsilon)) \prod_{j=1}^{l} (X^{(j)}_n (\epsilon))^{k_j} -\int  \phi^{(k)} (X_n) \prod_{j=1}^{l} (X^{(j)}_n)^{k_j}) 
          \end{equation}
        $$=\int  \phi^{(k+1)} (X)X^{(1)} \prod_{j=1}^{l} (X^{(j)})^{k_j} + 
        \int  \phi^{(k)} (X) \sum_{j=1}^{l} (k_j (X^{(j)})^{k_j -1} X^{(j+1)}\prod_{i\neq j} (X^{(i)})^{k_i}) .$$

To that end, we similarly divide the difference into two parts.
\begin{align*}
    &\int  \phi^{(k)} (X_n(\epsilon)) \prod_{j=1}^{l} (X^{(j)}_n (\epsilon))^{k_j}-\int  \phi^{(k)} (X_n) \prod_{j=1}^{l} (X^{(j)}_n)^{k_j})\\
    =&\int  (\phi^{(k)} (X_n(\epsilon))-\phi^{(k)}(X_n))
    \prod_{j=1}^{l} (X^{(j)}_n )^{k_j} 
    +\int  \phi^{(k)} (X_n(\epsilon)) (\prod_{j=1}^{l} (X^{(j)}_n (\epsilon))^{k_j}-\prod_{j=1}^{l} (X^{(j)}_n)^{k_j})) \\
    =:& G_1 + G_2.
\end{align*}

First we compute
\begin{align*}
    \frac{1}{\epsilon} G_1&= \int \phi^{(k+1)} (X_n + \xi_n (X_n(\epsilon) - X_n)) \frac{X_n (\epsilon) - X_n}{\epsilon}\prod_{j=1}^{l} (X_n^{(j)})^{k_j} \\
    &= \int \phi^{(k+1)} (X_n + \xi_n (X_n(\epsilon) - X_n)) (X_n^{(1)})+\sum_{t=1}^{n-1}S_{n,t} \epsilon^t)\prod_{j=1}^{l} (X_n^{(j)})^{k_j} 1_{\{X_n\le M \text{ or } X_n (\epsilon) \le M\}}.
\end{align*}
Here $\DS S_{n,t} = \frac{1}{\lambda_1^{t+1}} \sum_{m=1}^{n} \lambda_{i_1} \cdots\lambda_{i_m} \binom{m}{t+1}$. 

We first show that the terms containing $S_{n,t}$'s converge to zero. For each $1\le m\le n $, since $n = \mathcal{O}(\log\frac{1}{|\epsilon|})$ and $o(m)\le n$,  if $|\epsilon|$ is small enough, 

$$\frac{\lambda_{i_1}(\epsilon) \cdots\lambda_{i_m}(\epsilon)}{\lambda_{i_1} \cdots\lambda_{i_m}} = \frac{(\lambda_1 + \epsilon)^{o(m) }\lambda_2^{m-o(m)}}{\lambda_1 ^{o(m) }\lambda_2^{m-o(m)}} = 
\left(1+\frac{\epsilon}{\lambda_1}\right)^{o(m)}\to 1,\text{ as }\epsilon\to 0,$$
Hence $X_{n}(\epsilon)\le M$ implies $X_n \le 2M$, and we have
$$S_{n,t}\le \frac{1}{\lambda_1^t}  \sum_{m=1}^n \lambda_{i_1}\cdots\lambda_{i_m}n^{t+1} \le \frac{2n^{t+1}}{\lambda_1^t}M.  $$
Since $n$ is of order $-\log |\epsilon |$, it is not hard to show that
\begin{equation}\label{higher order remainder is zero, intermediate step}
    \phi^{(k+1)}(X_n + \xi_n (X_n(\epsilon) - X_n))\sum_{t=1}^{n-1}S_{n,t} \epsilon^t \to 0,  \text{ a.e. as }\epsilon\to 0, \text{ and it is bounded.}
\end{equation}

Also, note that
$$\prod_{j=1}^{l} (X_n^{(j)})^{k_j} 1_{\{X_n\le M \text{ or } X_n (\epsilon) \le M\}}\le \prod_{j=1}^{l} (X_n^{(j)})^{k_j} 1_{\{X_n\le 2M \}}$$
and the right hand side converges both a.e. and in $L^1$. To see the convergence in $L^1$, observe that it is increasing in $n$, and the pointwise limit 
$\DS \prod_{j=1}^l (X^{(j)})^{k_j} 1_{\{ X\le 2M\}}$ is integrable, by the argument used for deriving \eqref{estimate 4}. Thus it is convergent in $L^1$ by the Monotone Convergence Theorem. Moreover, the left hand side converges a.e., so by Lemma \ref{Generalized Dominated Convergence Lemma}, 

\begin{equation}
    \prod_{j=1}^{l} (X^{(j)})^{k_j} 1_{\{X_n\le M \text{ or } X_n (\epsilon) \le M\}} \text{ converges in } L^1,
\end{equation}
which combined with \eqref{higher order remainder is zero, intermediate step} allows us to apply Lemma \ref{Convergence of Product of Two Functions} to conclude
\begin{equation}\label{higher order remainder converges to zero,final step}
    \int(\phi^{(k+1)}(X_n + \xi_n (X_n(\epsilon) - X_n))(\sum_{t=1}^{n-1}S_{n,t} \epsilon^t )\prod_{j=1}^{l} (X_n^{(j)})^{k_j} 1_{\{X_n\le M \text{ or } X_n (\epsilon) \le M\}}\to 0.
\end{equation}

On the other hand, we have
$$\begin{cases}
    \phi^{(k+1)}(X_n + \xi_n (X_n(\epsilon) - X_n))\to \phi^{(k+1)}(X) \text{ a.e. and is bounded} \\
   (X_n^{(1)})^{k_1+1} \prod_{j=2}^{l} (X^{(j)})^{k_j} 1_{\{X_n\le M \text{ or } X_n (\epsilon) \le M\}}\xrightarrow{L^1}  (X_n^{(1)})^{k_1+1}\prod_{j=2}^{l} (X_n^{(j)})^{k_j} 1_{\{X\le M\} }
\end{cases}$$
where the second convergence follows from the same computations as in \eqref{estimate 4}
 and the generalized dominated convergence lemma \ref{Generalized Dominated Convergence Lemma}. Therefore by Lemma \ref{Convergence of Product of Two Functions}, 
 \begin{align}\label{taking derivative on phi, obtaining an extra first order derivative}
      &\lim_{\epsilon\to 0} \int\phi^{(k+1)}(X_n + \xi_n (X_n(\epsilon) - X_n)) (X_n^{(1)})\prod_{j=1}^{l} (X_n^{(j)})^{k_j} 1_{\{X_n\le M \text{ or } X_n (\epsilon) \le M\}}\\
      \nonumber
      &=\int\phi^{(k+1)}(X) (X^{(1)})^{k_1 +1}\prod_{j=2}^{l} (X^{(j)})^{k_j} 1_{\{X\le M \}}
      = \int\phi^{(k+1)}(X) (X^{(1)})^{k_1 +1}\prod_{j=2}^{l} (X^{(j)})^{k_j}.
 \end{align}
 Combining \eqref{higher order remainder converges to zero,final step} and \eqref{taking derivative on phi, obtaining an extra first order derivative} gives \begin{equation}\label{estimate of G1}
     \lim_{\epsilon\to 0} \frac{G_1}{\epsilon}  = \int\phi^{(k+1)}(X) (X_n^{(1)})^{k_1 +1}\prod_{j=2}^{l} (X^{(j)})^{k_j}.
 \end{equation}
 
 Finally, we compute $G_2$. First a simple calculation shows 
 $$X_n^{(j)}(\epsilon) = X_n^{(j)} + \frac{j!}{\lambda_1^j}\left(\sum_{m=1}^n \binom{o(m)}{j}  \frac{(o(m)-j)}{\lambda_1}  \lambda_{i_1} \cdots\lambda_{i_m}\epsilon\right) + P_n^{(j)}=:X_n^{(j)} + T_n^{(j)},$$
 where the term $P_n^{(j)}$ contains all terms with power of $\epsilon$ at least $2$. Now
 \begin{align*}
   \prod_{j=1}^{l} (X^{(j)}_n (\epsilon))^{k_j}&-\prod_{j=1}^{l} (X^{(j)}_n)^{k_j}=\prod_{j=1}^{l} (X^{(j)}_n +T_n^{(j)})^{k_j}-\prod_{j=1}^{l} (X^{(j)}_n)^{k_j}\\
   &=\prod_{j=1}^{l} (\sum_{s_j=0}^{k_j} \binom{k_j}{s_j} (X_n^{(j)})^{k_j-s_j} (T_n^{(j)})^{s_j})^{k_j}-\prod_{j=1}^{l} (X^{(j)}_n)^{k_j}\\
   &= \sum_{\begin{matrix}
       0\le s_i\le k_i,i=1,\cdots,l\\
       s_1 +\cdots +s_l>0
   \end{matrix}} \prod_{j=1}^l \binom{k_j}{s_j} (X_n^{(j)})^{k_j -s_j} (T_n^{(j)})^{s_j} .  \quad \quad \quad (*)
 \end{align*}

The terms in the right hand side of $(*)$ that contains $\epsilon^t,t\ge 2$ will contribute $0$ in the limit 
$\displaystyle \lim_{\epsilon\to 0}\frac{G_2}{\epsilon}$, for the same reason as in \eqref{higher order remainder converges to zero,final step}. Therefore
\begin{align*}
& \lim_{\epsilon\to 0}\frac{G_2}{\epsilon}=\\
&\lim_{\epsilon\to 0}\frac{1}{\epsilon} \int\phi^{(k)}(X_n(\epsilon))  \!\!\!\! \sum_{\begin{matrix}
       0\le s_i\le k_i\\
       s_1 +\cdots +s_l=1
   \end{matrix}}  \!\!\!\!\prod_{j=1}^l \binom{k_j}{s_j} (X_n^{(j)})^{k_j -s_j} \left[\frac{j!}{\lambda_1^j}\sum_{m=1}^n \binom{o(m)}{j} \frac{(o(m)-j)}{\lambda_1}  \lambda_{i_1} \cdots\lambda_{i_m}\epsilon\right]^{s_j}\\
   =& \lim_{\epsilon\to 0} \int\phi^{(k)}(X_n(\epsilon))\sum_{1\le j\le l,k_j\ge 1} k_j \frac{\prod_{i=1}^l (X_n^{(i)})^{k_i} }{X_n^{(j)}}  X_n^{(j+1)} \quad\quad\quad  (**)  \\ 
  = & \lim_{\epsilon\to 0}\int \phi^{(k)}(X_n(\epsilon))\sum_{1\le j\le l,k_k\ge 1} (\prod_{i=1,i\neq j}^{l} (X_n^{(i)})^{k_i} \frac{\,d}{\,d\epsilon}(X_n^{(j)}(\epsilon))^{k_j}|_{\epsilon=0}) \\
 =&   \sum_{1\le j\le l,k_j\ge 1} \int\phi^{(k)}(X) k_j \prod_{i=1,i\neq j}^{l} (X^{(i)})^{k_i}(X^{(j)})^{k_j-1} X^{(j+1)}.
 \quad\quad\quad  (***) 
\end{align*}

Here to get $(**)$ we used the fact that
$$\frac{j!}{\lambda_1^j}\sum_{m=1}^n \binom{o(m)}{j} \frac{1}\lambda_1 (o(m)-j) \lambda_{i_1} \cdots\lambda_{i_m} = \frac{(j+1)!}{\lambda_1^{j=1}}\sum_{m=1}^n \binom{o(m)}{j+1}  \lambda_{i_1} \cdots\lambda_{i_m} = X_n^{(j+1)}, $$
and $(***)$ holds for the same reason as before. This together with \eqref{estimate of G1} completes the proof. \qed
\subsection{Proof of Theorem \ref{infinite differentiability for compactly supported test function}}
\label{SSPfCompact}

 We are now ready to prove Theorem \ref{infinite differentiability for compactly supported test function}. 

 \begin{proof}
 \label{Smoothness for Compactly Supported Function: Final Proof}

Note that in Lemma \ref{finite approximation}, the constant $c$ and $\theta$ depends continuously on the parameters $\lambda_1 , \lambda_2$. Therefore, we may choose a pair of $c,\theta$ such that Lemma \ref{finite approximation} is true for $\epsilon$ sufficiently small.

Next, choose $n = n(\epsilon) = 2\log_{\theta} |\epsilon|$ in Lemma \ref{finite approximation}. Note that 
$\displaystyle\lim_{\epsilon\to 0} \frac{c\theta^n}{\epsilon} = 0$.
Hence
\begin{align*}
   & \lim_{\epsilon\to 0} \frac{h^{(l)}(\epsilon) - h^{(l)}(0)}{\epsilon} \\
   =& \lim_{\epsilon\to 0} \left[\frac{h^{(l)}(\epsilon) -\mathbb{E}[ \sum_{k_1 + 2k_2 + \cdots lk_l = l} L(k_1, \cdots, k_l) \phi^{(k)} (X_n(\epsilon)) \prod_{j=1}^{l} (X^{(j)}_n(\epsilon))^{k_j}]}{\epsilon} \right. \\
   -&\frac{h^{(l)}(0) -\mathbb{E}[ \sum_{k_1 + 2k_2 + \cdots lk_l = l} L(k_1, \cdots, k_l) \phi^{(k)} (X_n)) \prod_{j=1}^{l} (X^{(j)}_n)^{k_j}] }{\epsilon} \\
   +&\left. \frac{\mathbb{E}[ \sum_{k_1 + 2k_2 + \cdots lk_l = l} L(k_1, \cdots, k_l) (\phi^{(k)} (X_n(\epsilon)) \prod_{j=1}^{l} (X^{(j)}_n (\epsilon))^{k_j}-\phi^{(k)} (X_n) \prod_{j=1}^{l} (X^{(j)}_n)^{k_j})]}{\epsilon}\right]\\
   =& 0+0 + \mathbb{E}[ \sum_{k_1 + 2k_2 + \cdots (l+1)k_l = l+1} L(k_1, \cdots, k_{l+1}) \phi^{(k)} (X) \prod_{j=1}^{l+1} (X^{(j)})^{k_j}].
\end{align*}
The theorem is proved.
 \end{proof}

%% file: Examples_of_Non-Differentiability.tex
\section{Examples of Non-Differentiability}\label{Section: Examples of Non-Differentiability}
\subsection{Overview}
In this section we prove Theorem \ref{Section: non-diff example; Main Theorem} and \ref{Section: Non-diff Examples; Main Theorem 2}. We begin by recalling the precise statements and roughly describing how the assumptions will come into use.\\

\noindent
{\bf Theorem \ref{Section: non-diff example; Main Theorem}.}
  {\em   Assume $0<\lambda_1<1<\lambda_2$ are such that $\lambda_1 \lambda_2 <\frac{1}{4}$. Then there exists a smooth, bounded function $\phi$ such that $h(\epsilon) := \int \phi \,d_{\lambda_1+\epsilon , \lambda_2}$ is not differentiable at $0$.}
\\

\noindent
{\bf Theorem \ref{Section: Non-diff Examples; Main Theorem 2}.}
{\em     Assume $0<\lambda_1<1<\lambda_2$ are such that $\lambda_1\lambda_2<1$ and  $\log_2\lambda_2 >1+ \log_{\frac{1}{\lambda_1}} \lambda_2$. Then there exists a smooth, bounded function $\phi$ such that $h(\epsilon) := \int \phi \,d_{\lambda_1+\epsilon , \lambda_2}$ is not differentiable at $0$.}
\\

The assumptions in the last two results reflect the local thickness of $\mu$. In particular, they imply that the partial sum $\DS X_N = \sum_{m=0}^{N} \lambda_{i_1}\cdots\lambda_{i_m}$ can be large while the scaling factor $\Lambda_N = \lambda_{i_1}\cdots \lambda_{i_m}$ is tiny, and that the probability of this happening is not too small, thanks to self-similar structure of the random variable $X$. This leads to the  stationary measure of certain small intervals being large compared to their length, and in turn produces non-differentiability.  

In addition, the assumption $\lambda_1\lambda_2<\frac{1}{4}$ implies that the \textit{Lyapunov dimension} of $\mu$, which in our setting is simply equal to {\it entropy over Lyapunov exponent:}\\ $\DS \frac{h_\nu}{\chi} = \frac{\log2}{-\frac{1}{2}\log\lambda_1 - \frac{1}{2}\log\lambda_2}$, is less than $1$. On the other hand, the Lyapunov dimension is an upper bound for the Hausdorff dimension of the stationary measure of an iterated function system. In particular, $\mu_{\lambda_1,\lambda_2}$ has Hausdorff dimension less than one and is therefore singular (with respect to Lebesgue measure). For more details of Lyapunov dimension, we refer to \cite[theorem~3.2.1]{BSSBook} as well as Chapter $10$ of the same reference. On the other hand, having Lyapunov dimension less than $1$ is not a necessary condition for such thickness, as Theorem \ref{Section: Non-diff Examples; Main Theorem 2} shows, since we can take, for instance, $\lambda_1 = \frac{1}{11},\lambda_2 =10$. 

A key step  of the proof is to reduce both theorems to the existence of a family of special test functions 
$\{\phi_N\}_{N\in\mathbb{N}}$.
\begin{prop}\label{Section: Non-diff example; Intermediate Step}
    Under either assumptions of Theorem \ref{Section: non-diff example; Main Theorem}  and \ref{Section: Non-diff Examples; Main Theorem 2}, there exists a family of smooth, nonnegative functions $\{\phi_N\}$ with the following properties.
    \begin{enumerate}[(1)]
        \item There exists $\epsilon_0 >0$ such that for all $N$ and all $\epsilon$ with $|\epsilon|<\epsilon_0$, $h_N'(\epsilon)$ exist and are nonnegative. Here $ h_N(\epsilon) = h_{\phi_N}(\epsilon):=\int \phi_N \,d\mu_{\lambda_1+\epsilon , \lambda_2}$.
        \item $\DS \sum_{N} h_N'(0) = \infty$.
        \item  $\DS \phi:=  \sum  \phi_N$ is a smooth and $\mu_{\lambda_1+\epsilon,\lambda_2}$-integrable for small enough $\epsilon$.
    \end{enumerate}
\end{prop}

In the rest of this section, we first show how to deduce the main results from Proposition \ref{Section: Non-diff example; Intermediate Step} in \S \ref{Subsection: NOn-diff; use intermediate step to conclude}. Then we list several preliminary lemmas in \S \ref{Subsection: Non-diff; Preliminaries} before proving Proposition \ref{Section: Non-diff example; Intermediate Step} under the assumption of each theorem in \S \ref{Subsection: Non-diff; Prop for thm1} and \S \ref{Subsection: Non-diff; Prop for thm2} respectively.

\subsection{Proposition \ref{Section: Non-diff example; Intermediate Step} implies the main Theorems}\label{Subsection: NOn-diff; use intermediate step to conclude}

\begin{proof}
Given $\{\phi_N\}$ and $\phi$ satisfying the assumptions of Proposition \ref{Section: Non-diff example; Intermediate Step},  define $h(\epsilon)  = \int \phi\,d\mu_{\lambda_1+\epsilon,\lambda_2}, 0<|\epsilon|<\epsilon_0$ . We will show that $h$ is not differentiable at $0$. 

Let $ F(\epsilon,N) := \frac{h_N(\epsilon) - h_N(0)}{\epsilon}, N\in\mathbb{N},0<|\epsilon|<\epsilon_0$. Note that 
$$\lim_{\epsilon\to 0} F(\epsilon,N) = h_N'(0),\;\; \forall N\in\mathbb{N}$$
and 
$$F(\epsilon,N)\ge 0,  \;\; \forall 0<|\epsilon|<\epsilon_0, \; N\in\mathbb{N}$$
since $h_N'(\epsilon)\ge 0$ for all $0<|\epsilon|<\epsilon_0$. 
Thus 

\begin{align}
    \liminf_{\epsilon\to 0} \frac{h(\epsilon) - h(0)}{\epsilon} &= \liminf_{\epsilon
    \to 0} \frac{\int \sum_{N=1}^\infty \phi_N \,d\mu_{\lambda_1+\epsilon,\lambda_2} - \int\sum_{N=1}^\infty \phi_N \,d\mu_{\lambda_1,\lambda_2} }{\epsilon}\nonumber \\
    &= \liminf_{\epsilon
    \to 0}  \frac{\sum_{N=1}^\infty  (\int  \phi_N \,d\mu_{\lambda_1+\epsilon,\lambda_2} - \int \phi_N \,d\mu_{\lambda_1,\lambda_2}) }{\epsilon} \nonumber\\
    &= \liminf_{\epsilon\to 0} \int F(\epsilon,N) \,dc(N) \quad c \text{ is the counting measure on }\mathbb{N}\nonumber\\
    &\ge \int\liminf_{\epsilon\to 0} F(\epsilon,N) \,dc(N) \quad \text{ by Fatou's Lemma }\nonumber \\
    &= \sum_{N=1}^\infty h_N'(0) = \infty,
\end{align}
which implies $h$ is not differentiable at $0$. Note that the second equality holds because 
$\DS \sum_N \int \phi_N\,d\mu_{\lambda_1+\epsilon,\lambda_2}$ is absolutely convergent.
\end{proof}

\subsection{Preliminary Lemmas}\label{Subsection: Non-diff; Preliminaries}

In this section we collect several lemmas needed for the proof of Proposition \ref{Section: Non-diff example; Intermediate Step}. Recall that for a fixed choice of $0<\lambda_1<1<\lambda_2$ with $\lambda_1\lambda_2 < 1$,  
we defined the random variable
    $$X(\epsilon) = \sum_{m=0}^\infty  \lambda_{i_1}(\epsilon)\cdots\lambda_{i_m}(\epsilon) $$
  where  $\lambda_1(\epsilon) = \lambda_1+\epsilon$ and $\lambda_2(\epsilon) = \lambda_2$, so that $\mu_{\lambda_1+\epsilon,\lambda_2}$ is the law of $X(\epsilon)$. We also defined
    $$X^{(1)}(\epsilon) = \frac{1}{\lambda_1(\epsilon)} \sum_{m=1}^\infty \lambda_{i_1}(\epsilon)\cdots\lambda_{i_m}(\epsilon) o(m)$$
    to be the formal derivative of $X(\epsilon)$ with respect to $\epsilon$. Here $o(m)$ stands for the number of $1$'s in $i_1, \cdots,i_m$. For simplicity, $X$ and $X^{(1)}$ will stand for $X(0)$ and $X^{(1)}(0)$ respectively.

\begin{lm}\label{Section: Non-diff Example; Lemma on first order derivative}
    Let $\lambda_1,\lambda_2$ be such that $\lambda_1\lambda_2 <1$. Suppose $\phi\in C^1 (\mathbb{R})$ and its derivative $\phi'$ is compactly supported. Then the function $h(\epsilon) : = \int \phi \,d\mu_{\lambda_1+\epsilon,\lambda_2}$ is differentiable at $\epsilon = 0$ and the derivative is given by
    $$h'(0) = \int_{\Sigma} \phi'(X)X^{(1)}\,d\nu.$$

Here  $\nu = (\frac{1}{2},\frac{1}{2})^{\mathbb{N}}$ is the Bernoulli measure on $\Sigma = \{0,1\}^{\mathbb{N}}$. 
\end{lm}

\begin{proof}
    This is essentially Theorem \ref{infinite differentiability for compactly supported test function}, with the only difference being that only $\phi'$ is assumed to be compactly supported, not $\phi$ itself. But this matters little since we can prove the lemma following the same scheme.
    
    More precisely, suppose $M>0$ is such that $\phi'|_{(M,\infty)}\equiv0$. Then $\phi$ restricted to $(M,\infty)$ is constant. Recall that in the proof of Theorem
    \ref{infinite differentiability for compactly supported test function}, we defined 
    $$X(\epsilon) = \sum_{m=0}^\infty  \lambda_{i_1}(\epsilon)\cdots\lambda_{i_m}(\epsilon) , \quad X = X(0)$$
    as well as the finite truncate
    $$X_N(\epsilon) = \sum_{m=0}^N  \lambda_{i_1}(\epsilon)\cdots\lambda_{i_m}(\epsilon), \; X_N = X_N(0)\quad N\in\mathbb{N}.$$
Then we showed that 
$$ \int_{\Sigma}\phi(X)\,d\nu - \int_\Sigma \phi(X_N) \,d\nu$$ 
is exponentially small and 
$$ \frac{\int_\Sigma \phi(X_N(\epsilon)) \,d\nu - \int_\Sigma \phi(X_N) \,d\nu}{\epsilon}$$
converges to the actual derivative, if $N$ is chosen appropriately, see lemma \ref{finite approximation} and \ref{second step:finite difference converges to derivative} respectively. Moreover, in that proof, the only place we have used $\phi$ being compactly supported is to reduce the above two quantities into
$$\int_{\Sigma}\phi(X)- \phi(X_N) \,d\nu = \int_{\Sigma}\left(\phi(X)- \phi(X_N)\right) 1_{\{X_N\le M \text{ or } X\le M\}} \,d\nu  $$
and 
$$\int_\Sigma \phi(X_N(\epsilon)) -\phi(X_N) \,d\nu = \int_\Sigma \left(\phi(X_N(\epsilon)) -\phi(X_N)\right)1_{\{X_N(\epsilon)\le M \text{ or } X_N\le M\}} \,d\nu$$
for some $M>0$. But these are still true under our current setting of Lemma \ref{Section: Non-diff Example; Lemma on first order derivative}, since $\phi$ is constant on $(M,\infty)$. Therefore the proof of Theorem \ref{infinite differentiability for compactly supported test function}  goes through to give the lemma.
\end{proof}

\begin{lm}\label{Section: Non-diff Example; Lemma of X(1) bigger than X}
    There exists a constant $c>0$ such that with probability $1$, $X^{(1)}\ge cX$.
\end{lm}
\begin{proof}
    With probability $1$, $1$ will appear among $i_1,i_2,\cdots$.  Assume that the first occurrence is at $i_k$. Then 
    $$X = \sum_{j=0}^{k-1}\lambda_2^j + \lambda_2^{k-1}\lambda_1 (1+\lambda_{i_{k+1}} + \lambda_{i_{k+1}}\lambda_{i_{k+2}} +\cdots)$$
    and 
  \begin{align*}
      X^{(1)} &= \frac{1}{\lambda_1}\lambda_2^{k-1}\lambda_1 (1+\lambda_{i_{k+1}}o(k+1) + \lambda_{i_{k+1}}\lambda_{i_{k+2}}o(k+2) +\cdots)\\
      &\ge \lambda_2^{k-1}(1+\lambda_{i_{k+1}} + \lambda_{i_{k+1}}\lambda_{i_{k+2}} +\cdots) 
  \end{align*}  
  Note that 
  $$\sum_{j=0}^{k-1}\lambda_2^j \le \lambda_2^{k-1} \sum_{j=0}^\infty \lambda_2^{-j} = \frac{\lambda_2}{\lambda_2-1}\lambda_2^{k-1}$$
  and 
  $$S_k:=1+\lambda_{i_{k+1}} + \lambda_{i_{k+1}}\lambda_{i_{k+2}} +\cdots \ge \sum_{j=0}^\infty \lambda_1^j = \frac{1}{1-\lambda_1}.$$

  Therefore, \begin{align*}
      X&\le\lambda_2^{k-1} (\frac{\lambda_2}{\lambda_2-1}+ \lambda_1 S_k)
    \le \lambda_2^{k-1} (\frac{\lambda_2}{\lambda_2-1} \frac{S_k}{\frac{1}{1-\lambda_1}}+ S_k) \\
      &=\lambda_2^{k-1}S_k (\frac{\lambda_2 (1-\lambda_1)}{\lambda_2-1}+1)
      =  \left(\frac{\lambda_2 (1-\lambda_1)}{\lambda_2-1}+1\right) X^{(1)}.
  \end{align*}
  Hence $c:=  (\frac{\lambda_2 (1-\lambda_1)}{\lambda_2-1}+1)^{-1}$ meets the requirement of the lemma.
\end{proof}

\subsection{Proof of Proposition \ref{Section: Non-diff example; Intermediate Step}, under the assumption $\lambda_1\lambda_2<\frac{1}{4}$} \label{Subsection: Non-diff; Prop for thm1}

\begin{proof}
    Let us define the random variables
    $\DS X = \sum_{m=0}^\infty  \lambda_{i_1}\cdots\lambda_{i_m} $
    and for each $N\in \mathbb{N}$, 
    $$X_N =  \sum_{m=0}^N  \lambda_{i_1}\cdots\lambda_{i_m} , \: \Lambda_N = \lambda_{i_1}\cdots\lambda_{i_N},\quad
    Y_N =  \sum_{m=N+1}^\infty  \lambda_{i_{N+1}}\cdots\lambda_{i_m}.  $$
    Note that 
    \begin{itemize}
        \item $\DS X = X_N + \Lambda_N Y_N $;
        \item $X_N$ and $\Lambda_N$ are independent of $Y_N$;
        \item $Y_N$ and $X-1$ has the same distribution.
    \end{itemize}

    Let $\theta:=\sqrt{\lambda_1\lambda_2}$ and fix $\rho\in (\theta,\frac{1}{2})$ (this is where we used the assumption $\lambda_1\lambda_2 <\frac{1}{4}$). By Cram\'er's theorem, there exists $\delta>0$ such that
    \begin{equation}\label{Cramer's theorem}
        \mathbb{P}\{ \Lambda_N \le \rho^N\} \ge 1-e^{-\delta N}\: \text{ for all sufficiently large }N.
    \end{equation}

    For each $N\in \mathbb{N}$, consider positive integers $M$ satisfying the following two properties:
    \begin{equation}
    \label{Section: Non-diff Example; Definition of M(N)}
        \begin{cases}
        M\mathbb{P}\{X_N \ge M\} \ge N^{-\frac{1}{2}}, \\
        \mathbb{P}\{X_N \ge M\} \ge 2e^{-\delta N}.
    \end{cases}
    \end{equation}
    It is not hard to see that for sufficiently large $N\in\mathbb{N}$, the set of positive integers satisfying 
    \eqref{Section: Non-diff Example; Definition of M(N)} is non-empty and finite, so we are able to define $M(N)$ to be the largest one of them. 
    
 Next, define $A_N$ to be the set of values that $X_N$ can take that are not less than $ M(N)$. Then by construction,
        \begin{equation}
            \mathbb{P}\{X_N\ge M(N)\} = \mathbb{P}\{X_N \in A_N\} \ge 
            \max\left(2e^{-\delta N} , \frac{1}{M\sqrt{N}}\right).
        \end{equation}

Next, let $r>0$ be such that $\mathbb{P}\{X-1\le r\} = \frac{1}{2}$. For each $N\in\mathbb{N}$, define
$$B_N = \{x\in\mathbb{R}: \text{dist}(x,A_N) \le r\rho^N\}.$$

Finally, let $\phi_N$ be a smooth function such that $\phi_N(-2r) = 0$ and that its derivative $\phi_N'$ satisfies
\begin{itemize}
    \item $0\le \phi_N\le 1 $;
    \item $\phi_N|_{B_N} \equiv 1$; 
    \item supp$\phi_N$ is contained in the $2r\rho^N$-neighborhood of $A_N$. That is, $\phi'_N$ is zero on the set $\{x\in\mathbb{R}: \text{dist}(x,A_N) > 2r\rho^N\}=:B(A_N , 2r\rho^N)$. 
    \end{itemize}
Note that since supp$\phi_N'$ is contained in $[-2r,\infty)$, we have $\phi_N (x) = \int_{-\infty}^x \phi_N'(t)\,dt$.
    
    Suppose $M(N)$ is well-defined for all $N\ge N_0$. We claim that $\{\phi_N\}_{N\ge N_0}$ satisfies the assumptions of Proposition \ref{Section: Non-diff example; Intermediate Step}. These are clearly smooth. We verify conditions $(1)-(3)$ 
   below.

\textbf{Proof of $(1)$}. Since $\phi_N'$ is compactly supported, by Lemma \ref{Section: Non-diff Example; Lemma on first order derivative} $h_N$ is differentiable at all $\epsilon$ satisfying $(\lambda_1+\epsilon)\lambda_2 <1$ and the derivative is given by

    $$h_N'(\epsilon) = \int_{\Sigma} \phi'_N (X(\epsilon)) X^{(1)}(\epsilon)\,d\nu,$$
which is clearly non-negative since $\phi_N'\ge 0$ and $X^{(1)}(\epsilon)>0$. This establishes property $(1)$ of Proposition \ref{Section: Non-diff example; Intermediate Step}.

\textbf{Proof of $(2)$} We wish to derive a lower bound for $\phi_N'(0)$ so that $\DS \sum_N \phi_N'(0) = \infty$.
For simplicity, write $ p_N := \mathbb{P}\{X_N \ge M(N)\}$. We have
  \begin{align}
      \mathbb{P}\{X_N \ge M(N) , \Lambda_N <\rho^n \} &=  \mathbb{P}\{X_N \ge M(N) \} + \mathbb{P}\{ \Lambda_N <\rho^n \} -  \mathbb{P}\{X_N \ge M(N) \text{ or } \Lambda_N <\rho^n \}\nonumber \\
      &\ge  p_N + 1-e^{-\delta N} -1 \nonumber\\
      &\ge \frac{1}{2} p_N \quad \text{ since }  p_N=\mathbb{P}\{X_N \ge M(N)\} \ge 2e^{-\delta N}.\label{Section: Non-diff Example; Probability of XN being large and LambdaN being small}
  \end{align}

It follows that 
\begin{align*}
    \mathbb{P}\{X\in B_N\} &= \mathbb{P}\{X_N + \Lambda_N Y_N \in B_N\}
    \ge \mathbb{P}\{X_N\in A_N , \Lambda_N \le \rho^N , Y_N\le r\}\\
    &= \mathbb{P}\{X_N\in A_N , \Lambda_N \le \rho^N\}  \mathbb{P}\{ Y_N\le r\}
    \ge  \frac{1}{2} p_N \frac{1}{2} = \frac{1}{4} p_N.
\end{align*}
Then
\begin{align}
    h_N'(0) &= \int_\Sigma \phi'_N(X)X^{(1)} \ge \int_{\{X\in B_N\}} \phi'_N(X) cX \nonumber\\
    &\ge c(M(N)-r\rho^N)) \mathbb{P}\{X\in B_N\} 
    \ge \frac{c}{4}p_N \frac{M(N)}{2} 
    \ge \frac{c}{8} \frac{1}{\sqrt{N}}.\label{Section: Non-diff Example; estimate of phiN'(0)}
\end{align}
The first inequality above is due to Lemma \ref{Section: Non-diff Example; Lemma of X(1) bigger than X}. In the last inequality we have used \eqref{Section: Non-diff Example; Definition of M(N)}. Finally, the penultimate inequality holds for all sufficiently large $n$, because of the following

    \textbf{Claim}: $\DS \lim_{N\to\infty} M(N) = \infty$.

    \textit{Proof of Claim.} Otherwise there exists $M_0 \in\mathbb{N}$ and a sequence $\{N_k\}$ going to infinity such that $M_k:= M(N_k) < M_0$. We aim to reach a contradiction. By definition, for each $k$ we have
    $$\text{either }\, M_0 \mathbb{P}\{X_{N_k} \ge M_0\} <\frac{1}{\sqrt{N_k}} \text{ or } \, \mathbb{P}\{X_{N_k}\ge M_0\} < 2e^{-\delta N_k}.$$
    But since $X_{N_k}$ converges to $X$ in distribution and $\mu_{\lambda_1,\lambda_2}$ is non-atomic \cite[Proposition~3.1.1]{BSSBook}
    \begin{equation*}
         \begin{cases}
        \DS \lim_{k\to \infty}M_0 \mathbb{P}\{X_{N_k} \ge M_0\} = M_0 \mathbb{P}\{X \ge M_0\}>0\\
        \DS \lim_{k\to \infty} \mathbb{P}\{X_{N_k} \ge M_0\} = \mathbb{P}\{X \ge M_0\}>0
    \end{cases}
    \end{equation*}
   
    and
  $\DS
        \lim_{k\to\infty} \frac{1}{\sqrt{N_k}} = \lim_{k\to \infty} 2e^{-\delta N} = 0,$
        which is contradiction. The claim is proved.

        Now inequality \eqref{Section: Non-diff Example; estimate of phiN'(0)} immediately gives $\DS \sum_{N} \phi_N'(0) = \infty$.

        \textbf{Proof of}$(3)$. We want to show that $\DS \phi:= \sum_N \phi_N$ is a smooth, bounded function. Note that $\phi_N'$ is supported on $[M(N) - 2r\rho^N,\infty)$, and so is $\phi_N$ itself. By the previous claim, $\DS \lim_{N\to\infty} M(N) - 2r\rho^N = \infty$. Therefore, for each $x\in\mathbb{R}$, there exists $N (x)\in\mathbb{N}$ such that $x+1 < M(N) - 2r\rho^N$. It follows that $\DS \phi = \sum_N \phi_N$ is in fact a finite sum $\DS \sum_{N\le N(x)}\phi_N$ when restricted on the interval $[-\infty , x+1)$, on which $\phi$ is therefore smooth since each $\phi_N$ is smooth. Since $x$ is arbitrary, $\phi$ is smooth. 
        
        Observe that the above argument also implies $\DS \phi' = \sum_N \phi_N'$ and $\phi(x) = \int_{\infty}^x \phi'(t)\,dt$. Since each $\phi_N'$ is non-negative, in order to prove boundedness of $\phi$, it suffices to show that 
        $$\sum_N ||\phi_N'||_{L^1}< \infty.$$
        Now for each $N$, $\phi_N'\in [0,1]$ is supported on $C_N$, the $2r\rho^N$-neighborhood of $A_N$, and it is clear that 
 $\DS   \# A_N \le 2^N p_N.$
Thus 
$$
    ||\phi_N'||_{L^1} = \int \phi_N'(t)\,dt 
    \le \mathcal{L}^1 (C_N) 
    \le  4r\rho^N \# A_N
    \le 8rp_N (2\rho)^N \le 8r (2\rho)^N.
$$
Since $\rho<\frac{1}{2}$, it follows that $\DS \sum_N ||\phi_N'||_{L^1}< \infty$ as desired. This finishes the proof of Proposition \ref{Section: Non-diff example; Intermediate Step} and thus 
Theorem \ref{Section: non-diff example; Main Theorem} is established
in the case $\lambda_1 \lambda_2<1/4$.
\end{proof}

\subsection{Proof of Proposition \ref{Section: Non-diff example; Intermediate Step} under the assumption $\log_2\lambda_2 >1+ \log_{\frac{1}{\lambda_1}} \lambda_2$} \label{Subsection: Non-diff; Prop for thm2}

\begin{proof}
    By a simple calculation, the assumption $\log_2\lambda_2 >1+ \log_{\frac{1}{\lambda_1}} \lambda_2$ implies that there exists a constant $\kappa>0$ such that 
    \begin{equation}
        \begin{cases}
            \rho:=\lambda_2 \lambda_1^\kappa <1,\\
            \lambda_2 > 2^{1+\kappa}.
        \end{cases}
    \end{equation}

    For each $N\in\mathbb{N}$, define $\kappa(N) = \lceil\kappa N\rceil$, where $\lceil x\rceil$ stands for the least integer that is larger than $x$. Define $$x_N := \sum_{m=0}^{N+\kappa(N)}\lambda_{j_1}\cdots\lambda_{j_m}
    \quad\mathrm{where}\quad     
     \begin{matrix}   
      (j_1 ,j_2,\cdots,j_{N+\kappa(N)}) =  (\underbrace{2,2\cdots,2} , \underbrace{1,1,\cdots,1}).\\
    \qquad\qquad\qquad\qquad\qquad\;\;\,  N   \quad\quad\quad \kappa(N)
        
  \end{matrix}$$
  Note that with this choice of $(j_1,\cdots,j_{N+\kappa(N)})$, we have $x_N > \lambda_2^N$ and 
  $\DS \prod_{m=1}^{N+\kappa(N)} \lambda_{j_m} <\rho^N$.

  Fix $r>0$ so that $\mathbb{P}\{X-1 \le r\} = \frac{1}{2}$ and let $B(x_N , \rho^N r)$ be the ball centered at $X_N$ with radius $\rho^N r$. Define $\phi_N$ to be a smooth function with $\phi_N(-2r) = 0$ and $\phi_N'$ is such that
  \begin{itemize}
    \item $0\le \phi_N\le 1 $;
    \item $\phi_N|_{B(x_N,\rho^N r)} \equiv 1$; 
    \item supp$\phi_N$ is contained in the $2r\rho^N$-neighborhood of $x_N$. That is, $\phi'_N$ is zero outside $B(x_N , 2r\rho^N)$. 
    \end{itemize}

Note that since supp$\phi_N'$ is contained in $[-2r,\infty)$, we have $\phi_N (x) = \int_{-\infty}^x \phi_N'(t)\,dt$. As before, we verify that the family $\{\phi_N\}$ meets the requirements in Proposition \ref{Section: Non-diff example; Intermediate Step}. 

\textbf{Proof of}$(1)$. This step is the same as in the proof of Theorem \ref{Section: non-diff example; Main Theorem}.

\textbf{Proof of}$(2)$. Since $\phi_N'$ is compactly supported, by Lemma \ref{Section: Non-diff Example; Lemma on first order derivative}, we have
\begin{align}
    h_N'(0) &= \int_{\Sigma} \phi_N'(X) X^{(1)}\,d\nu 
    \ge \int_{A_N} cX\,d\mu_{\lambda_1,\lambda_2}\quad c \text{ is from Lemma \ref{Section: Non-diff Example; Lemma of X(1) bigger than X}} \nonumber \\
    &\ge c(\lambda_2^N - r\rho^N) \mathbb{P}\{X\in B(x_N,\rho^N r)\}\nonumber\\
    &\ge \frac{c}{2}\lambda_2^N\mathbb{P}\{X\in B(x_N,\rho^N r)\} \quad\text{ if N is sufficiently large} \label{Section: Non-diff Example; Estimate 1 in Main theorem 2} .
\end{align}
Moreover, we have
$$
    \mathbb{P}\{X\in B(x_N,\rho^N r)\} = \mathbb{P}\{X_{N+\kappa(N)}+\Lambda_{N+\kappa(N)} Y_{N+\kappa(N)}\in B(x_N,\rho^N r)\} 
    $$$$
    \ge \mathbb{P}\{X_{N+\kappa(N)} = x_N, \Lambda_{N+\kappa(N)} < \rho^N ,Y_{N+\kappa(N)}  \le r \}
    \ge 2^{-N-\kappa(N)} \frac{1}{2} \ge \frac{1}{4}2^{-(1+\kappa)N}.
$$
This together with inequality \eqref{Section: Non-diff Example; Estimate 1 in Main theorem 2} implies 
$\DS h_N'(0)\ge \frac{c}{8} (2^{-1-\kappa}\lambda_2 )^N.$
Since by construction $\lambda_2 > 2^{1+\kappa}$, it is clear that  $\DS \sum_N h_N'(0) = \infty$ as desired.

\textbf{Proof of}$(3)$. By construction, the support of $\phi_N'$ is contained in $[\lambda_2^N - r\rho^N,\infty)$.
Thus the support of $\phi_N$ is also contained in  $[\lambda_2^N - r\rho^N,\infty)$. The right endpoint goes to $\infty$ as $N\to \infty$. Therefore the same argument as before applies to guarantee that
$\DS \phi:= \sum_N \phi_N$ is smooth. We only need to prove that
$\DS \sum_N ||\phi_N'||_{L^1}<\infty$
in order to show that $\phi$ is bounded. But clearly 
$\DS ||\phi_N'||_{L^1} \le 4r\rho^N $.
Since $\rho<1$, the theorem is proved.
\end{proof}

%% file: Differentiability_via_Moment_and_Derivative.tex
\section{Differentiability via Moment Condition and Bound on Derivative}
\label{Section: Differentiability via Moment Condition and Bound on Derivative }

\subsection{First derivative formula.}

In this section we start the proof of \eqref{Differentiability of Stationary measure}, by induction on $r$. We only prove the base case $r=1$ here, and postpone the complete proof to  \S \ref{ScHD},
as we will take advantage of several estimates from \S \ref{Section: Smoothness of Moments}.

Let us state Theorem \ref{Differentiability of Stationary measure} in the case of $r=1$.
\begin{thm}\label{Differentiability of Stationary measure, base case}
    Suppose $0<\lambda_1<1<\lambda_2$ are such that $\frac{\lambda_1+\lambda_2}{2}<1$ and $\phi\in C^1(\mathbb{R})$ is such that $||\phi'||_{C^0}<\infty$. Consider $\DS h(\epsilon) = h_\phi (\epsilon):= \int \phi\,d\mu_{\lambda_1+\epsilon,\lambda_2}$ which is defined for $\epsilon$ so small that $(\lambda_1+\epsilon)\lambda_2<1$. Then $h$ is differentiable at $0$ and the derivative is given by
    $$h'(0) = \int_{\Sigma} h'(X)X^{(1)}\,d\nu.$$
    Here $X^{(1)}$ is a random variable defined as 
    $$X^{(1)} = \frac{1}{\lambda_1} \sum_{m=1}^{\infty} \lambda_{i_1}\cdots\lambda_{i_m}o(m),$$
    where $o(m):= \#\{1\le s\le m:i_s = 1\}.$
\end{thm}

Since we are free to make a change of variable, we may assume without loss of generality, that $d_1=1,d_2=1$, see the discussion after Theorem \ref{infinite differentiability for compactly supported test function} in \S \ref{Section: Infinite Differentiability for Smooth Compactly Supported Test Function}. We only need to observe that $\phi\circ c$ still has bounded derivative.

The idea of proving Theorem \ref{Differentiability of Stationary measure, base case} is similar to that of Theorem \ref{infinite differentiability for compactly supported test function}. We shall approximate the actual integral by a finite truncate, and then calculate the derivative using this finite truncate. More precisely, we have the following two lemmas, which are virtually identical to those in \S \ref{Section: Infinite Differentiability for Smooth Compactly Supported Test Function}.

\begin{lm}\label{Lemma on Finite Truncate; Base case}
    Let $\DS  X_n(\epsilon) := \sum_{j=0}^{n}  \lambda_{i_1}(\epsilon)\cdots \lambda_{i_{j-1}}(\epsilon)$ and write $X_n$ for $X_n(0)$. There exists constants $0<\theta<1$ and $C>0$ such that
        $$\left|\int_{\Sigma}\phi(X)\,d
        \nu - \int_{\Sigma}\phi(X_n)\,d\nu\right| < C\theta^n.$$
        Moreover, the constants $C,\theta$ depends continuously on $\lambda_1,\lambda_2$.
\end{lm}

\begin{lm}\label{Lemma on Finite difference; Base case}
    Let $\epsilon\in \mathbb{R}$ and $n = n(\epsilon)\in \mathbb{N}$ be such that there exists some constant $d>0$ satisfying $\DS \frac{1}{d} \log|\frac{1}{\epsilon}| \le n\le d\log|\frac{1}{\epsilon}|$. Then the following limit holds:
    $$\lim_{\epsilon\to 0} \frac{\int_{\Sigma} \phi(X_n(\epsilon))\,d\nu-\int_{\Sigma}\phi(X_n)\,d\nu }{\epsilon}  = \int_{\Sigma} h'(X)X^{(1)}\,d\nu.$$
\end{lm}

In the following subsections, we first show how to deduce Theorem \ref{Differentiability of Stationary measure, base case} assuming the above two lemmas, which is a simple version of the one we did at the end of \S \ref{Section: Infinite Differentiability for Smooth Compactly Supported Test Function} . Then we turn to the proof of the two lemmas.

We end this subsection with a binomial identity to be used in the proof. 

\begin{lm}\label{Binomial Identity}
    For any $t\le j\in\mathbb{N}$, it holds
    $$\sum_{k=t}^j k(k-1) \cdots(k-t+1)\binom{j}{k} \lambda_1^{k-t}\lambda_2^{j-k}  = j(j-1)\cdots(j-t+1) (\lambda_1+\lambda_2)^{j-t}.$$
\end{lm}

\begin{proof}
        By the Binomial identity, 
    $\DS (\lambda_1+\lambda_2)^{j}=\sum_{k=0}^j \binom{j}{k} \lambda_1^k\lambda_2^{j-k}.$
    Differentiating both sides $t$ times with respect to $\lambda_1$, we get the desired identity.
\end{proof}

\subsection{Proof of Theorem \ref{Differentiability of Stationary measure}, assuming Lemma \ref{Lemma on Finite Truncate; Base case} and \ref{Lemma on Finite difference; Base case}}

\begin{proof}
    Since in Lemma \ref{Lemma on Finite Truncate; Base case}, the constants $C$ and $\theta$ depend continuously on the parameters $\lambda_1 , \lambda_2$, we may choose a pair of $C,\theta$ such that the lemma  holds for all sufficiently small $\epsilon$.

Next, choose $n = n(\epsilon) = 2\log_{\theta} |\epsilon|$ in Lemma \ref{Lemma on Finite difference; Base case}. Note that 
$\displaystyle\lim_{\epsilon\to 0} \frac{c\theta^n}{\epsilon} = 0$.
Hence 
\begin{align*}
    &\lim_{\epsilon\to 0} \frac{h(\epsilon) - h(0)}{\epsilon}\\
    &= \lim_{\epsilon\to 0} \frac{\int_{\Sigma}\phi(X(\epsilon))\,d\nu- \int_{\Sigma}\phi(X_n(\epsilon))\,d\nu  }{\epsilon}+\lim_{\epsilon\to 0} \frac{\int_{\Sigma}\phi(X_n)\,d\nu- \int_{\Sigma}\phi(X)\,d\nu  }{\epsilon}\\
    &+\lim_{\epsilon\to 0} \frac{\int_{\Sigma}\phi(X_n)\,d\nu- \int_{\Sigma}\phi(X_n(\epsilon))\,d\nu  }{\epsilon}.
\end{align*}

Note that the first and the second limits are equal to $0$, by Lemma \ref{Lemma on Finite Truncate; Base case} and the choice of $n$, and the third limit is equal to $\DS \int_{\Sigma} h'(X)X^{(1)}\,d\nu$, thanks to Lemma \ref{Lemma on Finite difference; Base case}. The theorem is proved.
\end{proof}

\subsection{Proof of Lemma \ref{Lemma on Finite Truncate; Base case}}
\begin{proof}
     Note that
    \begin{align*}
        |\int\phi \,d\mu - \int\phi\,d\mu_n| &=| \int_{\Sigma} \phi(X)-\phi(X_n)\,d\nu|
        \le \int_{\Sigma} ||\phi'||_{C^0} (X-X_n) \,d\nu =\\
        &  ||\phi'||_{C^0} \int_{\Sigma} \sum_{m=n+1}^{\infty}  \lambda_{i_1}\cdots\lambda_{i_m}  \,d\nu 
        = ||\phi'||_{C^0} \sum_{m=n+1}^{\infty} \left(\frac{\lambda_1+\lambda_2}{2}\right)^{m} 
        =  C \theta^n, 
    \end{align*}
 where $C = ||\phi'||_{C^0} \frac{\lambda_1+\lambda_2}{1-\lambda_1-\lambda_2} <\infty$ and $\theta =\frac{\lambda_1+\lambda_2}{2}$. Clearly these two expressions depend continuously on $\lambda_1,\lambda_2$. The lemma is proved.
\end{proof}

\subsection{Proof of Lemma \ref{Lemma on Finite difference; Base case}}
\begin{proof}
    We first compute
\begin{equation}\label{finite difference}
    X_n(\epsilon) - X_n = \sum_{t=1}^{n-1} a_{n,t} \epsilon^t,
\end{equation}
where
\begin{equation}
    a_{n,t} = \frac{1}{\lambda_1^t}\sum_{j=1}^{n} \lambda_{i_1}\cdots \lambda_{i_{j-1}} \binom{o(j-1)}{t} ,\quad 
    t = 1,\dots, n-1, 
\end{equation}
\begin{equation}
     o(k) := \# \{s: 1\le s\le k , i_s = 1\} ,\; k=1,\dots, n-1
\end{equation}
and $\binom{n}{k} := \frac{n!}{k!(n-k)!}$ denotes the binomial coefficients
(we set $\binom{n}{k} =0$ if $k>n$).

By Taylor's theorem, there exists $\xi_n\in[-1,1]$ such that
\begin{align}
&\int_\Sigma  \frac{\phi(X_n(\epsilon)) - \phi(X_n)}{\epsilon} \,d\nu\nonumber\\
 &  = \int \phi'(X_n + \xi_n (X_n'-X_n))a_{n,1} \,d\nu + \int \phi'(X_n + \xi_n (X_n(\epsilon)-X_n)) \sum_{t=2}^{n-1} a_{n,t}\epsilon^{t-1} \,d\nu \nonumber\\
 &:= J_1 + J_2.
\end{align}

We claim that 
\begin{equation}
    \lim_{\epsilon\to 0}J_1 =\int_{\Sigma} h'(X)X^{(1)}\,d\nu \quad\text{and} \quad \lim_{\epsilon\to 0} J_2 = 0,
\end{equation}
from which the lemma follows immediately.

We first deal with $J_2$. Fix any $t\ge 2$, 
$$|\phi'\left(X_n+ \xi_n (X_n(\epsilon) - X_n)\right) a_{n,t}\epsilon^{t-1}|\le M|\epsilon|^{t-1} a_{n,t},$$
and
\begin{align*}
    \int a_{n,t} &= \frac{1}{\lambda_1^t} \sum_{j=t+1}^n\int\lambda_{i_1}\cdots \lambda_{i_{j-1}} \binom{o(j-1)}{t} .
\end{align*}
Moreover, for any $j\in[t+1,n]$,
\begin{align*}
 & \int\lambda_{i_1}\cdots \lambda_{i_{j-1}} \binom{o(j-1)}{t}  = 
 \left(\frac{1}{2}\right)^{j-1}\sum_{k=t}^{j-1} \binom{j-1}{k}\lambda_1^k \lambda_2^{j-1-k} \binom{k}{t}\\
  &= \frac{\lambda_1^t}{2^{j-1}} \frac{k!}{t!(k-t)!k(k-1)\cdots (k-t+1)}  \sum_{k=t}^{j-1} \binom{j-1}{k} k(k-1)\cdots(k-t+1) \lambda_1^{k-t}\lambda_2^{j-1-k}\\
  &= \frac{\lambda_1^t}{2^{j-1}} \frac{k!}{t!(k-t)!k(k-1)\cdots (k-t+1)}  (j-1)\cdots(j-t) (\lambda_1+\lambda_2)^{j-t-1} \;\text{ by Lemma }\ref{Binomial Identity}\\
  &= \frac{\lambda_1^t (\lambda_1+\lambda_2)^{j-t-1}}{2^{j-1}} \binom{j-1}{t}.
\end{align*}
It follows that
\begin{align*}
    M|\epsilon|^{t-1}\int a_{n,t}&= M|\epsilon|^{t-1}\sum_{j=t+1}^{n}  \frac{ (\lambda_1+\lambda_2)^{j-t-1}}{2^{j-1}} \binom{j-1}{t}
    = M|\epsilon|^{t-1}\sum_{j=0}^{n-t-1}  \frac{ (\lambda_1+\lambda_2)^{j}}{2^{j+t}} \binom{j+t}{t}\\
    &\le M\sum_{j=0}^{2\log_\theta |\epsilon|} (\frac{\lambda_1+\lambda_2}{2})^{j} (|\epsilon|^{\frac{t-1}{t}} (j+t))^t  
    \quad \text{ since }\binom{j+t}{t}\le (j+t)^t\\
    &\le C|\epsilon|^{\frac{1}{3}t}, \quad \text{ if } |\epsilon| \text{ is small enough,}
\end{align*}
where $C$ is some constant independent of $t$ and $\epsilon$. Note that in the last inequality, we have used the fact that $j,t$ are at most of order $-\log |\epsilon|$, $\DS \frac{\lambda_1+\lambda_2}{2}<1$ and $\DS \frac{t-1}{t}\ge \frac{1}{2}$. 
Therefore
\begin{equation}
    |J_2| \le CM\sum_{t=2}^{n-1} |\epsilon|^{\frac{t}{3}} \to 0 \text{ as } \epsilon\to 0.
\end{equation}

Next we consider $J_1$. Observe that
$$\phi'(X_n+\xi_n (X_n(\epsilon) - X_n))a_{n,1} = \phi'(X_n+\xi_n (X_n(\epsilon) - X_n))\frac{1}{\lambda_1}\sum_{j=1}^{n-1} \lambda_{i_1}\cdots\lambda_{i_{j-1}} o(j-1)$$
converges a.e. to 
$$\phi'(X)X^{(1)}$$
as $n\to \infty$ (or equivalently, $\epsilon\to 0$). Moreover, 
$$|\phi'(X_n+\xi_n (X_n(\epsilon) - X_n))\frac{1}{\lambda_1}\sum_{j=1}^{n-1} \lambda_{i_1}\cdots\lambda_{i_{j-1}} o(j-1)|\le \frac{M}{\lambda_1}\sum_{j=1}^{n-1} \lambda_{i_1}\cdots\lambda_{i_{j-1}} o(j-1)$$
and the right hand side is convergent both a.e. and in $L^1$, by Lemma \ref{Binomial Identity} and the Monotone convergence theorem. Therefore, by Lemma \ref{Generalized Dominated Convergence Lemma}, 
$$\phi'(X_n+\xi_n (X_n(\epsilon) - X_n))\frac{1}{\lambda_1}\sum_{j=1}^{n-1} \lambda_{i_1}\cdots\lambda_{i_{j-1}} o(j-1)$$
converges to
$$\phi'(X)X^{(1)}$$
also in $L^1$. This is the desired claim for $J_1$.
\end{proof}

%% file: Smoothness_of_Moments.tex
\section{Smoothness  of  moments}\label{Section: Smoothness of Moments}
\subsection{Main result and sketch of proof.}
The main objective of this section is to prove Theorem \ref{Smoothness for Finite Moment Condition} 
which we now recall.
\vskip2mm

\noindent {\bf Theorem \ref{Smoothness for Finite Moment Condition}}.
  {\em   Suppose $d_1=1,d_2=1$, $ 0<\lambda_1 <1<\lambda_2$ and $t>0$ are such that $\frac{\lambda_1^t+ \lambda_2^t}{2}<1$. Note that this implies $\lambda_1\lambda_2 <1$ and the stationary measure $\mu = \mu_{\lambda_1,\lambda_2}$ is well defined. Take the test function to be $\phi(x) = x^t$ and define $h(\epsilon) = \int \phi(x) \,d\mu_{\lambda_1 + \epsilon , \lambda_2}$ for $\epsilon$ small. Then $h$ is infinitely differentiable at $0$ and the derivative coincide with the one obtained by taking formal derivative:
    \begin{align*}
        h^{(l)}(0) &= \int_{\Sigma} \frac{\,d^l}{\,d \epsilon^l} (\phi\circ X(\epsilon))|_{\epsilon=0}\,d\nu \\
        &= \int_{\Sigma} \sum_{k_1 + 2k_2 + \cdots lk_l = l} L(k_1, \cdots, k_l) \phi^{(k)} (X) \prod_{j=1}^{l} (X^{(j)})^{k_j}\,d\nu,\quad \forall l\in \mathbb{N}.
    \end{align*}}

The proof will be induction on $l$, and the inductive step will follow the same idea as in Theorem \ref{infinite differentiability for compactly supported test function}, studying the finite truncate of the integral. In particular, we again use two lemmas that are virtually identical to Lemma \ref{finite approximation} and 
Lemma \ref{second step:finite difference converges to derivative}.

\begin{lm}\label{moment condition: finite truncate lemma}
    
  There exists $c>0 $ and $0<\theta<1$ that depends on $\lambda_1,\lambda_2$, $t$ and $l$ such that 
    \begin{align*}
        |&\int \sum_{k_1 + 2k_2 + \cdots lk_l = l} L(k_1, \cdots, k_l) \phi^{(k)} (X) \prod_{j=1}^{l} (X^{(j)})^{k_j} \\
        &-\int \sum_{k_1 + 2k_2 + \cdots lk_l = l} L(k_1, \cdots, k_l) \phi^{(k)} (X_n) \prod_{j=1}^{l} (X^{(j)}_n)^{k_j}| < c\theta ^n.
    \end{align*}

    Here, $X_n$ and $X_n^{(j)}$'s are finite truncates of $X,X^{(j)}$'s at $n$. That is,
    $$X_n =  \sum_{m=1}^n \lambda_{i_1 }\cdots \lambda_{i_m},\quad
    X_n^{(j)} = \frac{j!}{\lambda_1^j}\sum_{m=1}^{n} \lambda_{i_1} \cdots \lambda_{i_m} \binom{o(m)}{j} .$$
\end{lm}

\begin{lm}\label{Moment Condition: Finite Difference Converges to Derivative}

    For any $\epsilon\neq0$, let $n = n(\epsilon)\in\mathbb{N} $ be such that there exists some constant d satisfying $\frac{1}{d} \log\frac{1}{|\epsilon|} \le n \le d \log\frac{1}{|\epsilon|},\forall \epsilon\neq0$. Then 
    \begin{align*}
        \lim_{\epsilon\to 0} \frac{1}{\epsilon} &\left(\int \sum_{k_1 + 2k_2 + \cdots lk_l = l} L(k_1, \cdots, k_l) \phi^{(k)} (X_n(\epsilon)) \prod_{j=1}^{l} (X^{(j)}_n (\epsilon))^{k_j}\right. \\
        &-\left.\int \sum_{k_1 + 2k_2 + \cdots lk_l = l} L(k_1, \cdots, k_l) 
        \phi^{(k)} (X_n) \prod_{j=1}^{l} (X^{(j)}_n)^{k_j}\right) \\
        &=\int \sum_{k_1 + 2k_2 + \cdots (l+1)k_l = l+1} L(k_1, \cdots, k_{l+1}) \phi^{(k)} (X) \prod_{j=1}^{l+1} (X^{(j)})^{k_j}.
    \end{align*}
    Here $\displaystyle X_n (\epsilon):= \sum_{m=1}^n \lambda_{i_1}(\epsilon) \cdots \lambda_{i_m}(\epsilon)$ and 
    $$X_n^{(j)} (\epsilon) := \frac{\,d^j}{\,d \epsilon^j}X_n (\epsilon) =  \frac{j!}{\lambda_1^j}\sum_{m=1}^{n} \lambda_{i_1}(\epsilon) \cdots \lambda_{i_m}(\epsilon) \binom{o(m)}{j}.$$
 \end{lm}

 Repetition of the proof  in \S \ref{Smoothness for Compactly Supported Function: Final Proof} implies that Theorem \ref{Smoothness for Finite Moment Condition} holds if the above two lemmas are true. 
The plan for the rest of this section is the following. First we 
 prepare some auxiliary lemmas in \S \ref{Subsection: Smoothness via Moment condition; Lemmas}. 
 Then we prove  Lemma \ref{moment condition: finite truncate lemma} and \ref{Moment Condition: Finite Difference Converges to Derivative} in \S \ref{Subsection: smoothness of moment condition; proof of truncate lemma} and \S \ref{Subsection: smoothness of moment condition; proof of convergence lemma} respectively. Once Theorem \ref{Smoothness for Finite Moment Condition} is established, we proceed to finish Theorem \ref{Differentiability of Stationary measure} in \S \ref{Subsection: smoothness of moment condition; proof of induction step} by exploiting several estimates in its proof.

 \subsection{Auxiliary estimates} \label{Subsection: Smoothness via Moment condition; Lemmas}
 We begin with several 
preliminary results.

\begin{lm}\label{Moment Condition: Power of Sum is less than Sum of Power}
    For any $0\le\alpha\le1$, $n\in \mathbb{N}$ and $a_1,\cdots a_n \ge0$, 
    $$(a_1 + \cdots a_n)^\alpha \le a_1^\alpha + \cdots +a_n^{\alpha}. $$
\end{lm}
\begin{proof}
    If $\displaystyle\sum_{i=1}^n a_i = 0$ then the inequality is trivially true. Otherwise, the lemma is equivalent to
    $$1\le \sum_{i=1}^n \left(\frac{a_i}{a_1+\cdots +a_n}\right)^\alpha.$$
    Note that $(\frac{a_i}{a_1+\cdots +a_n})^\alpha\ge \frac{a_i}{a_1+\cdots +a_n},i=1,\cdots,n$, since for each $i$, $\frac{a_i}{a_1+\cdots +a_n} \le 1$ and $\alpha\in [0,1]$. This immediately implies the lemma.
\end{proof}

\begin{lm}\label{Moment Condition: Lemma on Diference of a-Power}
    For any $a\in \mathbb{R}$ and any $X> Y>0$, we have
    \begin{equation}
        X^a-Y^a \le \max\{a,1\} X^{a-1} (X-Y), \text{ if }a\ge 0
    \end{equation}
and
\begin{equation}
    Y^a-X^a \le \max\{|a|,1\} (X-Y) Y^a X^{-1},\text{ if }a <0. 
\end{equation}

\end{lm}

\begin{proof}
    If $0\le a\le 1$, then $a-1\le0$ and thus $ -Y^a\le -X^{a-1}Y $. It follows that
    $$X^a-Y^a\le X^a - X^{a-1}Y =   X^{a-1} (X-Y)$$
    as desired.

    If $a>1$, then by Taylor's theorem, there exists $\xi\in (Y,X)$ such that
    $$X^a-Y^a = a\xi^{a-1} (X-Y) \le aX^{a-1}(X-Y) $$
    and the first inequality is proved.

    If $a<0$, apply the lemma in the case of $a\ge 0$ to get 
   
   $$
        Y^{a} - X^{a} = (Y^{-1})^{-a} - (X^{-1})^{-a}
        \le \max\{|-a|,1\} (Y^{-1})^{-a-1} (Y^{-1} - X^{-1})
   $$  \hskip75mm   $= \max\{|a|,1\} (X-Y) Y^a X^{-1}.
    $
\end{proof}

 \begin{lm}\label{Moment Condtion: Lemma for Binomial Difference}
     Let $a,b,j$ be any non-negative integers. Then
     $$\binom{a}{j} \le 3^j \left(\binom{a-b}{j}+1\right)(b^j+1).$$
     \end{lm}
     Note that we adopt the convention  $\binom{u}{v} = 0$ if $v>u$.

\begin{proof}
    If $a<j$ or $b=0$ or $j=0$, then the inequality is trivially true. So let's assume $a\ge j$ and $b,j\ge 1$.

If $a-b \le j+1$, then 
$$
    \binom{a}{j} = \frac{a(a-1)\cdots (a-j+1)}{j!} 
    \le \frac{(b+j+1)(b+j)\cdots(b+2)}{j(j-1) \cdots1} 
    \le 3^j b^j< 3^j \left[\binom{a-b}{j}+1\right](b^j+1)
$$
    where the penultimate
    inequality is because $\frac{b+k+1}{k} \le 3b$ for any $k\!\!\in \!\! \{1,\cdots,j\}$ and $b\ge 1$.

    If $a-b\ge j+2$, then 
    $$
        \frac{\binom{a}{j}}{\binom{a-b}{j}} \!=\! \frac{a(a-1)\cdots(a-j+1)}{(a-b)(a-b-1)\cdots(a-b-j+1)} 
        =\!\!\!\! \prod_{x=a-b-j+1}^{a-b} \!\!\!\left(1+\frac{b}{x}\right) 
        \le (2b)^j\!\!<\!\!  3^j \left[\binom{a-b}{j}+1\right](b^j+1)
    $$
    where the penultimate inequality is true since $1+\frac{b}{x}\le 2b$, for any $x\ge 3$ and $b\ge 1$. The lemma if proved.
\end{proof}

\subsection{Proof of Lemma \ref{moment condition: finite truncate lemma}.} \label{Subsection: smoothness of moment condition; proof of truncate lemma}
 \begin{proof}
     By the same argument as in Lemma \ref{finite approximation}, we may reduce the problem to estimating, for any tuple of integers $(k_1 , \cdots , k_l)$,
     $$|\int  \phi^{(k)} (X) \prod_{j=1}^{l} (X^{(j)})^{k_j} -\int  \phi^{(k)} (X_n) \prod_{j=1}^{l} (X^{(j)}_n)^{k_j}| ,$$

    which is less or equal to
   \begin{align*}
         &|\int  \phi^{(k)} (X) (\prod_{j=1}^{l} (X^{(j)})^{k_j}-\prod_{j=1}^{l} (X^{(j)}_n)^{k_j})| 
        +|\int  (\phi^{(k)} (X)-\phi^{(k)} (X_n)) \prod_{j=1}^{l} (X^{(j)}_n)^{k_j}|\\
        &\le \int  \phi^{(k)} (X) (\prod_{j=1}^{l} (X^{(j)})^{k_j}-\prod_{j=1}^{l} (X^{(j)}_n)^{k_j}) 
        +\int  |(\phi^{(k)} (X)-\phi^{(k)} (X_n)) \prod_{j=1}^{l} (X^{(j)}_n)^{k_j}|  \\
        &:= J_1 + J_2 .
   \end{align*}
      
Here is the outline of the proof. We will show that $J_1,J_2$ decay exponentially in $n$. For each $J_i, i =1,2$, we will discuss the two cases $t\ge k$ and $t<k$ separately. The estimate of $J_1$ in the case $t\ge k$ will be further split into two cases. 
All other estimates in this subsection  can be reduced to situations similar to those that have already appeared in the calculation of $J_1, t\ge k$.

    Now we estimate $J_1$. The first part of the calculation follows  the proof of Lemma ~\ref{finite approximation}. We repeat it once for the readers' convenience and omit similar estimates later on.
    
    As before $\displaystyle X^{(j)} - X_n^{(j)} = \frac{j!}{\lambda_1^j} \sum_{m=n+1}^{\infty} \lambda_{i_1}\cdots \lambda_{i_m}\binom{o(m)}{j} =: R_n^{(j)}$, and we compute
    \begin{align*}
        &\prod_{j=1}^{l} (X^{(j)})^{k_j} - \prod_{j=1}^{l} (X_n^{(j)})^{k_j} =\prod_{j=1}^{l} (X_n^{(j)} + R_n^{(j)})^{k_j} - \prod_{j=1}^{l} (X_n^{(j)})^{k_j}\\
        &=(\sum_{s_1=0}^{k_1} \binom{k_1}{s_1} (X_n^{(1)})^{k_1-s_1} (R_n^{(1)})^{s_1} )\cdots(\sum_{s_l=0}^{k_l} \binom{k_l}{s_l} (X_n^{(l)})^{k_l-s_l} (R_n^{(l)})^{s_l} ) - \prod_{j=1}^{l} (X_n^{(j)})^{k_j}\\
        &= \sum_{\begin{matrix}
            s_1,\cdots,s_l \\
            0\le s_i \le k_i,\forall i\\
            s_1 + \cdots +s_l >0
        \end{matrix}}  \prod_{j=1}^{l}\binom{k_j}{s_j} (X_n^{(j)})^{k_j-s_j}(R^{(j)}_n)^{s_j}
    \end{align*}
  Again, note that in the last summation, the number of summands is uniformly bounded by some constant depending on $l$, and the coefficients $\displaystyle \prod_{j=1}^{l}\binom{k_j}{s_j}$ are also uniformly bounded in terms of $l$. Therefore we can further reduce  the problem to  estimating
  $$\int \phi^{(k)}(X) \prod_{j=1}^{l} (X_n^{(j)})^{k_j-s_j}(R^{(j)}_n)^{s_j},$$
which is essentially
\begin{equation}\label{Similar Estimate 1 in moment and bouded derivative}
    \int X^{t-k} \prod_{j=1}^{l} (X_n^{(j)})^{k_j-s_j}(R^{(j)}_n)^{s_j}.
\end{equation}

  Now the proof differs from that of Lemma \ref{finite approximation} and we use the condition $\frac{\lambda_1^t + \lambda^t_2}{2} <1$. We discuss the cases $k\le t$ and $k>t$ separately. In the remainder of this section, the expression $A \preceq B$ will mean $A\le D B$ for some constant $D>0$ that is independent of $n,k,s$, while it could depend on $l,\lambda_1,\lambda_2,t$.
  
  If $k\le t$, we have
  $$
      \prod_{j=1}^{l} (X_n^{(j)}) ^{k_j-s_j} \le \prod_{j=1}^l (\sum_{m=1}^n \lambda_{i_1} \cdots \lambda_{i_m} K_l(m))^{k_j-s_j} 
      = (\sum_{m=1}^n \lambda_{i_1} \cdots \lambda_{i_m} K_l(m))^{k-s},
  $$
  where $\displaystyle K_l(m):= \sum_{r=0}^l \binom{o(m)}{r}\le C o(m)^{l+1}$ for some constant $C>0$ depending only on $l$ and $s:= s_1 + \cdots s_l$. 
  
  Similarly,
  \begin{align*}
      \prod_{j=1}^{l} (R_n^{(j)}) ^{s_j} \le \left(\sum_{m=n+1}^{\infty} \lambda_{i_1} \cdots \lambda_{i_m} K_l(m)\right)^{s}.
  \end{align*}
We would like to further split the right hand side into two independent parts. One of them only involves $i_1,\cdots,i_n$ and the other depends on the future $i_{m},m\ge n+1$. To do that, we apply Lemma
\ref{Moment Condtion: Lemma for Binomial Difference} to get, for any $m\ge n+1$, 
$$
    K_l(m) \!\!=\!\! \sum_{j=0}^l \binom{o(m)}{j} 
    \!\!\le \!\! \sum_{j=0}^{l} 3^j \left(\binom{o(m)-o(n)}{j}+1\right)(o(n)^j+1)
    \!\!\le\!\! \sum_{j=0}^{l} 3^l \left(\binom{o(m)-o(n)}{j}+1\right)(n^l+1)
    $$$$
    \le 3^l(n^{l}+1)\sum_{j=0}^{l} \left[\binom{o(m)-o(n)}{j}+1\right]
    \preceq (n^{l}+1)((o(m)-o(n))^{l+1}+1).
$$
Therefore, 
\begin{equation}\label{Moment Condition: Estimate of Tail}
   \sum_{m=n+1}^{\infty} \lambda_{i_1} \cdots \lambda_{i_m} K_l(m)\preceq A_n B_n
\end{equation}
where 
$$A_n:= (n^l+1)\lambda_{i_1} \cdots\lambda_{i_n}, \quad  B_n := \sum_{m=n+1}^{\infty}\lambda_{i_{n+1}}\cdots\lambda_{i_m}((o(m)-o(n))^{l+1}+1).$$ Note that $o(m)-o(n)$ equals the number of $1$'s among $i_{n+1} , \cdots i_m$, therefore $B_n$ depends only on $i_m$ for $m\ge n+1$. 

For later use, we also state the following consequence of H\"older inequality. \\

\textbf{FACT}: For any $0<\rho<1$ and any $p,p^* >1$ satisfying $\frac{1}{p}+\frac{1}{p^*} = 1$, we have
\begin{equation}\label{Holder inequality}
    \sum_{m=1}^n \lambda_{i_1 }\cdots\lambda_{i_m} K_l(m) \le 
    \left[\sum_{m=1}^n \left(\frac{\lambda_{i_1}}{\rho}\cdots \frac{\lambda_{i_m}}{\rho}K_l(m)\right)^p  \right]^{\frac{1}{p}} (\sum_{m=1}^n 
  \rho^{mp*})^{\frac{1}{p*}}.
\end{equation}
  
Since $t\ge k$, we also have
\begin{align*}
    X^{t-k} &= (\sum_{m=1}^n \lambda_{i_1}\cdots\lambda_{i_m} +  \sum_{m=n+1}^{\infty} \lambda_{i_1}\cdots\lambda_{i_m})^{t-k}\\
    &\le (\sum_{m=1}^n \lambda_{i_1}\cdots\lambda_{i_m}K_l (m) +  \sum_{m=n+1}^{\infty} \lambda_{i_1}\cdots\lambda_{i_m}K_l (m) )^{t-k} \\
    &\le C (\sum_{m=1}^n \lambda_{i_1}\cdots\lambda_{i_m}K_l (m))^{t-k} +  C(\sum_{m=n+1}^{\infty} \lambda_{i_1}\cdots\lambda_{i_m}K_l (m) )^{t-k},
\end{align*}
where in the last step we used the fact that for any $\alpha\ge 0$, there exists a constant $C = C(\alpha)$ such that
$$(a+b)^{\alpha} \le C(a^\alpha + b^{\alpha}),\quad\forall a,b>0.$$
Moreover, since $C$ depends continuously on $\alpha$ and $0\le t-k\le t$, we may choose a uniform $C$ that works for all $t-k$. It follows that
\begin{align}\notag
    X^{t-k}  \prod_{j=1}^{l} (X_n^{(j)}) ^{k_j-s_j} (R_n^{(j)}) ^{s_j} &\le
   X^{t-k} (\sum_{m=1}^{n} \lambda_{i_1} \cdots \lambda_{i_m} K_l(m))^{k-s}(\sum_{m=n+1}^{\infty} \lambda_{i_1} \cdots \lambda_{i_m} K_l(m))^{s}       \\
    &\preceq (\sum_{m=1}^{n} \lambda_{i_1} \cdots \lambda_{i_m} K_l(m))^{(t-k) +(k-s)}(\sum_{m=n+1}^{\infty} \lambda_{i_1} \cdots \lambda_{i_m} K_l(m))^{s} \notag\\
    &+  (\sum_{m=1}^{n} \lambda_{i_1} \cdots \lambda_{i_m} K_l(m))^{k-s}(\sum_{m=n+1}^{\infty} \lambda_{i_1} \cdots \lambda_{i_m} K_l(m))^{t-k+s} \notag \\
    &=: Q_1+Q_2. \label{Moment Condition: estimate 1}
\end{align}

We first deal with $Q_1$. 

\begin{align}\label{Moment Condition: Estimate of Q_1, first Step}
    Q_1& = \left(\sum_{m=1}^{n} \lambda_{i_1} \cdots \lambda_{i_m} K_l(m)\right)^{t-s}
    \left(\sum_{m=n+1}^{\infty} \lambda_{i_1} \cdots \lambda_{i_m} K_l(m)\right)^{s}\\
    &\le (\sum_{m=1}^{n} \lambda_{i_1} \cdots \lambda_{i_m} K_l(m))^{t-s} A_n^s
  B_n^s. \notag
\end{align}

We claim that $\int (\sum_{m=1}^{n} \lambda_{i_1} \cdots \lambda_{i_m} K_l(m))^{t-s} A^s_n B^s_n\le C\theta^n$ for some 
$C\!\!=\!\! C(\lambda_1,\lambda_2,t,l)\!>\!0$ and $0<\theta = \theta(\lambda_1,\lambda_2,t,l)<1$. Since $B_n$ depends on the future after $n$, the claim follows if we prove that the integral of $B_n$ is uniformly bounded and the integral of the rest decays exponentially.

\noindent
More precisely, we first show that there exist $C \!\!=\!\! C(\lambda_1,\lambda_2,t,l)\!\!>\!\!0$ and 
$\theta\!\! =\!\! \theta(\lambda_1,\lambda_2,t,l)\in (0,1)$ such that for all $n\in \mathbb{N}$,
$$ \int (\sum_{m=1}^{n} \lambda_{i_1} \cdots \lambda_{i_m} K_l(m))^{t-s} A^s_n = $$
\begin{equation}\label{Moment Condition: Estimate 3}
    (n^l+1)^s\int (\sum_{m=1}^{n} \lambda_{i_1} \cdots \lambda_{i_m} K_l(m))^{t-s} (\lambda_{i_1} \cdots\lambda_{i_n})^s \le C\theta^n.
\end{equation}

 Note that $0\le t-s\le t-1$. The calculation is further divided into two cases.

If $t-s\le1$, then 
\begin{align*}
    (\sum_{m=1}^{n} \lambda_{i_1} \cdots \lambda_{i_m} K_l(m))^{t-s}&\le \sum_{m=1}^{n} \lambda_{i_1}^{t-s} \cdots \lambda_{i_m}^{t-s} K^{t-s}_l(m) 
     \le \sum_{m=1}^{n} \lambda_{i_1}^{t-s} \cdots \lambda_{i_m}^{t-s} K_l(m).
\end{align*}
Here the first inequality follows from Lemma \ref{Moment Condition: Power of Sum is less than Sum of Power}, and the second is because $K_l(m)\ge 1$ and $t-s\le 1$. Thus
$$(\sum_{m=1}^{n} \lambda_{i_1} \cdots \lambda_{i_m} K_l(m))^{t-s} (\lambda_{i_1} \cdots\lambda_{i_n})^s \le \sum_{m=1}^{n} \lambda_{i_1} \cdots \lambda_{i_m} K_l(m)(\lambda_{i_{m+1}} \cdots\lambda_{i_n})^s.$$

A simple calculation shows that for $j = 0,1,\cdots ,l$,
\begin{equation}\label{Moment Condition: Estimate on the Expectation with Polynomial Distortion}
    \int\sum_{m=1}^{n} \lambda_{i_1} \cdots \lambda_{i_m} \binom{o(m)}{j} \lambda^s_{i_{m+1}}\cdots\lambda^s_{i_n} 
    =  \sum_{m=1}^{n}\left(\frac{\lambda_1^s + \lambda_2^s}{2}\right)^{n-m} \int \lambda_{i_1} \cdots \lambda_{i_m} \binom{o(m)}{j}
\end{equation} 
\begin{align*}
    &=\sum_{m=1}^{n}\left(\frac{\lambda_1^s + \lambda_2^s}{2}\right)^{n-m} \frac{1}{2^m} \lambda_1^j (\lambda_1+\lambda_2)^{m-j} \binom{m}{j} \\
    &= \left(\frac{\lambda_{1}}{\lambda_1 + \lambda_2}\right)^j \sum_{m=1}^{n} 
    \left(\frac{\lambda_1^s + \lambda_2^s}{2}\right)^{n-m}\left(\frac{\lambda_1+\lambda_2}{2}\right)^m \binom{m}{j}\\
    &\le  \sum_{m=1}^n \theta_1^{n-m}\theta_1^m \left(\frac{\lambda_1+\lambda_2}{2\theta_1}\right)^m \binom{m}{j}
    \le C_1\theta_1^n.
\end{align*}
Above $0<\theta_1<1$ is any real number such that 
$\DS  \theta_1>\kappa:=\max_{u\in [1, t]}\left\{\frac{\lambda_1^u + \lambda_2^u}{2}\right\}$.
Since $\frac{\lambda_1^t +\lambda_2^t}{2}<1$, we have
$\kappa<1$  by  monotonicity of the function $x\mapsto \frac{\lambda_1^x + \lambda_2^x}{2} , x>0$. Moreover, $\DS C_1:=\sum_{m=1}^\infty \left(\frac{\lambda_1+\lambda_2}{2\theta_1}\right)^m \binom{m}{j} <\infty$ since $\frac{\lambda_1+\lambda_2}{2\theta_1} <1$. 
Summing  over $j$ yields 
\begin{equation}\label{Moment Condition: Estimate 2}
    \int (\sum_{m=1}^{n} \lambda_{i_1} \cdots \lambda_{i_m} K_l(m))^{t-s} A^s_n\le (n^l+1)^s Cl\theta^n \le C' (\theta')^n
\end{equation}
proving the claim.\\

If $t\!-\!s\!\!>\!\!1$, fix a real number $\rho$ such that $0<\rho<1$ and $\frac{\lambda_1^t + \lambda^t_2}{2\rho^{t-s}} <1$.  By H\"older inequality
\begin{align}\label{Moment Condition: Holder Trick}
    (\sum_{m=1}^{n} \lambda_{i_1} \cdots \lambda_{i_m} K_l(m))^{t-s}&\le (\sum_{m=1}^{n} (\frac{\lambda_{i_1}}{\rho} \cdots \frac{\lambda_{i_m}}{\rho} K_l(m))^{t-s})(\sum_{m=1}^n \rho^{m(t-s)^*})^{\frac{t-s}{(t-s)^*}}\\
    &\le I(\rho) \sum_{m=1}^{n} \left(\frac{\lambda_{i_1}}{\rho}\right)^{t-s} \cdots 
    \left(\frac{\lambda_{i_m}}{\rho}\right)^{t-s} K^{t-s}_l(m).
\end{align}
It follows that
\begin{align*}
      \int (\sum_{m=1}^{n} \lambda_{i_1} \cdots \lambda_{i_m} K_l(m))^{t-s} A^s_n&\le
      I(\rho)(n^l+1) \int    \sum_{m=1}^{n} \frac{\lambda_{i_1}^t}{\rho^{t-s}} \cdots \frac{\lambda_{i_m}^t}{\rho^{t-s}} K^{t-s}_l(m) \lambda_{i_{m+1}}^{s} \cdots\lambda_{i_{n}}^s.
\end{align*}
Here, $(t-s)^*>1$  is such that $\frac{1}{t-s}+\frac{1}{(t-s)^*}=1$ and $I(\rho) = (\sum_{m=1}^\infty \rho^{m(t-s)^*})^{\frac{t-s}{(t-s)^*}}<\infty$. Note that we can choose $\rho$ such that $I(\rho)$ is  uniformly bounded in terms of $\lambda_1,\lambda_2,t$.

Moreover, since the term $K^{t-s}_l(m)\preceq m^{t(l+1)}$, a calculation similar to \eqref{Moment Condition: Estimate on the Expectation with Polynomial Distortion} yields that if we fix a $\theta_2$ such that $1>\theta_2>\max\{\kappa,\frac{\lambda_1^t+\lambda_2^t}{2\rho^{t-s}}\}$ and a $\theta>0$ such that $1>\theta>\theta_2$, then there exists $C_2 = C_2(\lambda_1,\lambda_2,t,l)>0$ and $C = C(\theta_2,\theta) = C(\lambda_1,\lambda_2,t,l)>0$ such that
\begin{align*}
    I(\rho)(n^l+1) \int    \sum_{m=1}^{n} \frac{\lambda^{t}_{i_1}}{\rho^{t-s}} \cdots \frac{\lambda^{t}_{i_m}}{\rho^{t-s}} K^{t-s}_l(m) \lambda_{i_{m+1}}^{s} \cdots\lambda_{i_{n}}^s &\le I(\rho)(n^l+1)   C_2 \theta_2^n
    \le C\theta^n.
\end{align*}
This finishes the proof of  \eqref{Moment Condition: Estimate 3}.\\

Next, we show that there exists $M = M(\lambda_1,\lambda_2,t,l)<\infty$ such that
$$\int B_n^s \le M,\quad\forall n\in \mathbb{N}.$$
Recall that $\DS B_n := (\sum_{m=n+1}^{\infty}\lambda_{i_{n+1}}\cdots\lambda_{i_m}((o(m)-o(n))^{l+1}+1))$ and that $o(m)-o(n)$ is exactly the number of $1$'s  among $i_{n+1} , \cdots i_m$. Therefore the above is equivalent to

\begin{equation}\label{Moment Condition: Estimate 4}
    \int  (\sum_{m=1}^{\infty}\lambda_{i_{1}}\cdots\lambda_{i_m}(o(m)^{l+1}+1))^s \le M,
\end{equation}
which is readily seen. Indeed, note that $t\ge k\ge s\ge1$ and $\frac{\lambda_1^s + \lambda_2^s}{2}<1$. If $s=1$, then since $(o(m)^{l+1}+1)^s\le Do(m)^{l(l+1)}$ is of polynomial growth,  a calculation similar to \eqref{Moment Condition: Estimate on the Expectation with Polynomial Distortion} yields  the result.
If $s>1$, then we first fix some suitable $\rho<1$ satisfying $\frac{\kappa}{\rho^s}<1$ and apply H$\ddot{o}$der inequality to get an estimate similar to \eqref{Moment Condition: Holder Trick}:
$$\int(\sum_{m=1}^{\infty}\lambda_{i_{1}}\cdots\lambda_{i_m}(o(m)^{l+1}+1))^s\le I(\rho) \int \sum_{m=1}^{\infty}\frac{\lambda^s_{i_{1}}}{\rho^s}\cdots\frac{\lambda^s_{i_m}}{\rho^s}(o(m)^{l+1}+1)^s.$$
Again because $(o(m)^{l+1}+1)^s$ is uniformly of polynomial growth and $\frac{\lambda_1^s + \lambda^s_2}{2\rho^s}\le \frac{\kappa}{\rho^s}<1$ , a \eqref{Moment Condition: Estimate on the Expectation with Polynomial Distortion} type argument finishes the proof.

Combining \eqref{Moment Condition: Estimate 3} and \eqref{Moment Condition: Estimate 4}  we obtain
  $\DS   \int Q_1 
  \le MC \theta^n.$

This shows that $\int Q_1$ is exponentially decaying in $n$ in the case $t\ge k$. 

Next, we prove $\int Q_2 \le C \theta^n$ for the same $C$ and $\theta$ that work for $\int Q_1$. Indeed, since $t-k+s > 0$, 
we apply \eqref{Moment Condition: Estimate of Tail} to get
\begin{align}
    Q_2 &=(\sum_{m=1}^{n} \lambda_{i_1} \cdots \lambda_{i_m} K_l(m))^{k-s}(\sum_{m=n+1}^{\infty} \lambda_{i_1} \cdots \lambda_{i_m} K_l(m))^{t-k+s}\\
    &\preceq (\sum_{m=1}^{n} \lambda_{i_1} \cdots \lambda_{i_m} K_l(m))^{k-s} A_n^{t-k+s} B_n^{t-k+s}.
\end{align}
Note that $0\le k-s\le k-1\le t-1$, so this is exactly the same situation as we faced in estimating $Q_1$, see 
\eqref{Moment Condition: Estimate of Q_1, first Step}, except that now $k-s$ plays the role of $t-s$. We can again discuss the cases $k-s\le 1$ can $k-s>1$ separately and follow the same reasoning as above, which will lead to
$\DS \int Q_2 \le MC\theta^n$
for the same $M,\theta,C$ that works for $Q_1$. Then it is immediate that
\begin{equation}\label{Moment Condition: Estimate of J_1, case t ge k}
    J_1 \preceq  \int Q_1 +\int Q_2 \le 2CM \theta^n, 
\end{equation}
in the case $t\ge k$. \\

Next, we estimate $ J_1$ in the case $t<k$. Define $\beta = \frac{t}{k}\in (0,1)$ and note that
\begin{align}
    &  X^{t-k} ( \sum_{m=1}^{n} \lambda_{i_1} \cdots \lambda_{i_m} K_l(m))^{k-s}(\sum_{m=n+1}^{\infty} \lambda_{i_1} \cdots \lambda_{i_m} K_l(m))^{s} \notag \\
    &=\frac{(\sum_{m=1}^{n} \lambda_{i_1} \cdots \lambda_{i_m} K_l(m))^{k-s}(\sum_{m=n+1}^{\infty} \lambda_{i_1} \cdots \lambda_{i_m} K_l(m))^{s}}{X^{(1-\beta)(k-s)} X^{(1-\beta)s}} \notag \\
    &= \left(\sum_{m=1}^{n} \frac{\lambda_{i_1} \cdots \lambda_{i_m}}{X^{1-\beta}} K_l(m)\right)^{k-s}
    \left(\sum_{m=n+1}^{\infty} \frac{\lambda_{i_1} \cdots \lambda_{i_m}}{X^{1-\beta}} K_l(m)\right)^{s} \notag \\
    &\le \left(\sum_{m=1}^{n} \lambda^\beta_{i_1} \cdots \lambda^\beta_{i_m} K_l(m)\right)^{k-s}
    \left(\sum_{m=n+1}^{\infty} \lambda^\beta_{i_1} \cdots \lambda^\beta_{i_m} K_l(m)\right)^{s}\quad \text{ since } X\ge \lambda_{i_1}\cdots\lambda_{i_m},\forall m \notag \\
    &\preceq \left(\sum_{m=1}^{n} \lambda^\beta_{i_1} \cdots \lambda^\beta_{i_m} K_l(m)\right)^{k-s} 
    \left(\lambda^\beta_{i_1} \cdots \lambda^\beta_{i_m}\right)^s 
    \left(\sum_{m=n+1}^{\infty} \lambda^\beta_{i_{n+1}} \cdots \lambda^\beta_{i_m} ((o(m)-o(n))^l+1)\right)^{s},
    \label{5.16A}
\end{align}
where the last inequality is due to \eqref{Moment Condition: Estimate of Tail}. Define $\tilde{\lambda}_1 = \lambda^\beta_1$ and $\tilde{\lambda}_2  = \lambda^\beta_2$. Then the above expression becomes
\begin{equation}
    (\sum_{m=1}^{n} \tilde{\lambda}_{i_1} \cdots \tilde{\lambda}_{i_m} K_l(m))^{k-s} (\tilde{\lambda}_{i_1} \cdots \tilde{\lambda}_{i_m})^s (\sum_{m=n+1}^{\infty} \tilde{\lambda}_{i_{n+1}} \cdots \tilde{\lambda}_{i_m} ((o(m)-o(n))^l+1))^{s}. 
\end{equation}

The rest of the calculation is similar to that of $J_1$ in the case $t\ge k$. We claim that there exists $C>0,\theta\in (0,1)$ and $M>0$ depending on $\lambda_1,\lambda_2,t,l$ such that
\begin{equation}\label{Moment Condition: Estimate of J_1, case t<k, first n term}
    \int (\sum_{m=1}^{n} \tilde{\lambda}_{i_1} \cdots \tilde{\lambda}_{i_m} K_l(m))^{k-s} (\tilde{\lambda}_{i_1} \cdots \tilde{\lambda}_{i_m})^s \le C\theta^n
\end{equation}
and
\begin{equation}\label{Moment Condition: Estimate of J_1, case t<k, tail term}
    \int (\sum_{m=n+1}^{\infty} \tilde{\lambda}_{i_{n+1}} \cdots \tilde{\lambda}_{i_m} ((o(m)-o(n))^l+1))^{s} \le M.
\end{equation}
 Now we are in the same situation as before. More precisely, since $\frac{\tilde{\lambda}_1^k + \tilde{\lambda}_2^k}{2} = \frac{\lambda_1^t + \lambda_2^t}{2}<1$, the proof of \eqref{Moment Condition: Estimate of J_1, case t<k, first n term} follows from the same reasoning as in \eqref{Moment Condition: Estimate 3}. Moreover, since 
 $$\frac{\tilde{\lambda}_1^s + \tilde{\lambda}_2^s}{2} = \frac{\lambda_1^{\frac{st}{k}} + \lambda_2^{\frac{st}{k}}}{2}\le \kappa':= \max\left\{ \frac{\lambda_1^{\frac{u}{v}t} + \lambda_2^{\frac{u}{v}t}}{2} \right\} <1,$$ 
 where the maximum is taken among $\{(u,v)\in \mathbb{Z}^2: t<v\le l, 1\le u \le v\}$, the proof of \eqref{Moment Condition: Estimate of J_1, case t<k, tail term} is the same as \eqref{Moment Condition: Estimate 4}. Combining the two, we have
 $\DS 
     J_1 \le MC\theta^n,\quad\forall n\in \mathbb{N}.    
$
  This proves that $ J_1$ decays exponentially in the case $t<k$ and the estimate of $J_1$ is complete. \\
  
  We now turn to $J_2$. Again, we discuss the two cases $k\le t$ and $k>t$.
  
  If $k\le t$, we can further assume $t>k$, because otherwise  $\phi^{(k)}(X) - \phi^{(k)}(X_n)$ is zero and $J_2=0$. Now 
  \begin{align}\label{Similar estimate 2 in moment condition and bounded derivative}
      |(\phi^{(k)} (X)-\phi^{(k)} (X_n)) \prod_{j=1}^{l} (X^{(j)}_n)^{k_j}|& \preceq (X^{t-k} - X_n^{t-k}) \prod_{j=1}^{l} (X^{(j)}_n)^{k_j}\nonumber \\
      &\preceq X^{t-k-1}(X-X_n) \prod_{j=1}^{l} (X^{(j)}_n)^{k_j}\nonumber\\
      &\le X^{t-k-1} (\sum_{m=n+1}^{\infty} \lambda_{i_1} \cdots \lambda_{i_m}K_l(m))(\sum_{m=1}^{n} \lambda_{i_1} \cdots \lambda_{i_m} K_l(m))^k  
  \end{align}
where in the second inequality we have used Lemma \ref{Moment Condition: Lemma on Diference of a-Power}.  

 Now there are two sub-cases. If $t\ge k+1$, then we are in the same situation
as in the second term in \eqref{Moment Condition: estimate 1}, with $k$ replaced with $k+1$ and $s$ replaced by $1$. Therefore, there exists $C= C(\lambda_1,\lambda_2,t,l)>0$ and $\theta = \theta(\lambda_1,\lambda_2,t,l)\in (0,1)$ such that

\begin{equation}\label{Moment Condition: Estimate of J_2, last step}
    J_2 \le C\theta^n,
\end{equation}
proving that $J_2$ is exponentially decaying in the case $k\le t$.

If $k< t<k+1$, write $\beta := t-k\in(0,1)$. We have
\begin{align}
    &X^{t-k-1}(X-X_n)  (\sum_{m=1}^{n} \lambda_{i_1} \cdots \lambda_{i_m} K_l(m))^k \\
    &= (\sum_{m=1}^{n} \lambda_{i_1} \cdots \lambda_{i_m} K_l(m))^{t-\beta}  \frac{\lambda_{i_1}\cdots\lambda_{i_n}(\sum_{m=n+1}^{\infty} \lambda_{i_{n+1}} \cdots\lambda_{i_m})}{X^{1-\beta}}\\
    &\le  (\sum_{m=1}^{n} \lambda_{i_1} \cdots \lambda_{i_m} K_l(m))^{t-\beta}  (\lambda_{i_1}\cdots\lambda_{i_n})^{\beta}(\sum_{m=n+1}^{\infty} \lambda^{\beta}_{i_{n+1}} \cdots\lambda^{\beta}_{i_m}).
\end{align}
 Moreover, the right hand side is exponentially small if we can show that $$\int (\sum_{m=1}^{n} \lambda_{i_1} \cdots \lambda_{i_m} K_l(m))^{t-\beta}  (\lambda_{i_1}\cdots\lambda_{i_n})^{\beta} \text{  is exponentially small }$$ and $$\int(\sum_{m=n+1}^{\infty} \lambda^{\beta}_{i_{n+1}} \cdots\lambda^{\beta}_{i_m}) \text{ is uniformly bounded.}$$
 Now the former can be proved exactly the same way as  \eqref{Moment Condition: Estimate 3}, and the latter is obvious since $\frac{\lambda_1^\beta + \lambda_2^\beta}{2}<1$ by monotonicity of the function $x\mapsto \frac{\lambda_1^x+\lambda_2^x}{2}$. Thus $J_2$ is exponentially small in the case $t\ge k$.

If $k>t$, then 
 \begin{align}
      |(\phi^{(k)} (X)-& \phi^{(k)} (X_n)) \prod_{j=1}^{l} (X^{(j)}_n)^{k_j}| \preceq (X_n^{t-k} - X^{t-k}) \prod_{j=1}^{l} (X^{(j)}_n)^{k_j} \notag \\
      &\preceq X_n^{t-k}X^{-1} (X-X_n)  (\sum_{m=1}^{n} \lambda_{i_1} \cdots \lambda_{i_m} K_l(m))^k \notag \\
      &\le X_n^{t-k} \left(\sum_{m=1}^{n} \lambda_{i_1} \cdots \lambda_{i_m} K_l(m)\right)^k  \frac{\sum_{m=n+1}^{\infty}\lambda_{i_1}\cdots\lambda_{i_m}  }{X}
      \label{K<TIn1}
  \end{align}
  where the middle inequality uses Lemma \ref{Moment Condition: Lemma on Diference of a-Power}. Define $\beta = \frac{t}{k+1}\in (0,1)$ and rewrite \eqref{K<TIn1} as
\begin{align*}
      &\frac{(\sum_{m=1}^{n} \lambda_{i_1} \cdots \lambda_{i_m} K_l(m))^{k}}{X_n^{(1-\beta)k} } \frac{\sum_{m=n+1}^{\infty} \lambda_{i_1}\cdots\lambda_{i_m}}{X^{1-\beta}} \frac{X_n^{\beta}}{X^{\beta}} \\
    &\le (\sum_{m=1}^{n} \frac{\lambda_{i_1} \cdots \lambda_{i_m}}{X_n^{1-\beta}} K_l(m))^{k} \sum_{m=n+1}^{\infty} \lambda^{\beta}_{i_1} \cdots\lambda^{\beta}_{i_m} 
   \le (\sum_{m=1}^{n} \lambda^\beta_{i_1} \cdots \lambda^\beta_{i_m} K_l(m))^{k} \lambda^{\beta}_{i_1}\cdots\lambda^{\beta}_{i_n}
 \sum_{m=n+1}^{\infty} \lambda^{\beta}_{i_{n+1}} \cdots\lambda^{\beta}_{i_m}.
\end{align*}
 since  $X\ge \lambda_{i_1}\cdots\lambda_{i_m}$ for all $m$ and  $X>X_n.$

Note that $$\frac{\lambda_1^{(k+1)\beta} + \lambda_2^{(k+1)\beta}}{2} = \frac{\lambda^t_1 + \lambda^t_2}{2}<1\quad\text{and}\quad \frac{\lambda_1^{\beta}+\lambda_2^{\beta}}{2} \le 
\kappa'':= \max_{u\in\mathbb{Z}: t+1< u\le l}\left\{\frac{\lambda_1^{\frac{t}{u}}+\lambda_2^{\frac{t}{u}}}{2}\right\}<1.$$ 
Therefore, the above right hand side is almost the same as the 
right hand side in 
\eqref{5.16A}, with $k$ replaced with $k+1$ and $s$ replaced with $1$. 
The only difference is that here the tail term is $\DS B_n' :=\sum_{m=n+1}^{\infty} \lambda^{\beta}_{i_{n+1}} \cdots\lambda^{\beta}_{i_m}$ versus
$\DS B_n =\sum_{m=n+1}^{\infty} \lambda^{\beta}_{i_{n+1}} \cdots\lambda^{\beta}_{i_m} (1+(o(m) - o(n))^l)$ in 
\eqref{5.16A}.
But clearly $B_n'\le B_n$, so by the same reasoning as in \eqref{Moment Condition: Power of Sum is less than Sum of Power}, there exists $C = C(\lambda_1,\lambda_2,t,l)>0$, $M = M(\lambda_1,\lambda_2,t,l)>0$ and $\theta = \theta(\lambda_1,\lambda_2,t,l)\in (0,1)$ such that 
$$\int (\sum_{m=1}^{n} \lambda^\beta_{i_1} \cdots \lambda^\beta_{i_m} K_l(m))^{k} \lambda^{\beta}_{i_1}\cdots\lambda^{\beta}_{i_n}
 \sum_{m=n+1}^{\infty} \lambda^{\beta}_{i_{n+1}} \cdots\lambda^{\beta}_{i_m} \le MC\theta^n.$$
 This immediately implies $J_2 \le MC\theta^n$ as desired. The proof is complete.
 \end{proof}

\subsection{Proof of Lemma \ref{Moment Condition: Finite Difference Converges to Derivative}.}\label{Subsection: smoothness of moment condition; proof of convergence lemma}
 \begin{proof}
 As in the proof of Lemma \ref{second step:finite difference converges to derivative}, we wish to show, for any given tuple of integers $(k_1,\cdots k_l)$ satisfying $\DS \sum_{i=1}^l ik_i = l$, 
      \begin{equation}\label{Moment Condition: Finite Difference converges to derivative, step 1}
          \lim_{\epsilon\to 0} \frac{1}{\epsilon} \left(\int  \phi^{(k)} (X_n(\epsilon)) \prod_{j=1}^{l} (X^{(j)}_n (\epsilon))^{k_j} -\int  \phi^{(k)} (X_n) \prod_{j=1}^{l} (X^{(j)}_n)^{k_j}\right) 
          \end{equation}
        $$=\int  \phi^{(k+1)} (X)X^{(1)} \prod_{j=1}^{l} (X^{(j)})^{k_j} + 
        \int  \phi^{(k)} (X) \sum_{j=1}^{l} (k_j (X^{(j)})^{k_j -1} X^{(j+1)}\prod_{i\neq j} (X^{(i)})^{k_i}) ,$$

where $\DS k:=\sum_{i=1}^{l} k_i$. As before, we divide the difference into two parts.
\begin{align*}
    &\int  \phi^{(k)} (X_n(\epsilon)) \prod_{j=1}^{l} (X^{(j)}_n (\epsilon))^{k_j}-\int  \phi^{(k)} (X_n) \prod_{j=1}^{l} (X^{(j)}_n)^{k_j})\\
    =&\int  (\phi^{(k)} (X_n(\epsilon))-\phi^{(k)}(X_n))
    \prod_{j=1}^{l} (X^{(j)}_n )^{k_j} 
    +\int  \phi^{(k)} (X_n(\epsilon)) (\prod_{j=1}^{l} (X^{(j)}_n (\epsilon))^{k_j}-\prod_{j=1}^{l} (X^{(j)}_n)^{k_j})) \\
    =:& G_1 + G_2.
\end{align*}
 We first deal with $\DS \lim_{\epsilon\to 0} \frac{1}{\epsilon}G_2$. By a simple calculation, for $j=0,\cdots,l$ we have
 $$X_n^{(j)}(\epsilon) = X_n^{(j)} + \sum_{r=1}^{n} b^{(j)}_{n,r} \epsilon^r=: X^{(j)}_n + T^{(j)}_n,$$
 where
 \begin{equation}\label{Moment Condition: Coefficient for epsilon^r}
     b^{(j)}_{n,r}:=\frac{j!}{\lambda_1^{r+j}} \sum_{m=1}^n \lambda_{i_1}\cdots\lambda_{i_m} \binom{o(m)-j}{r} \binom{o(m)}{j}.
 \end{equation}
 
By the same reasoning as in Lemma \ref{second step:finite difference converges to derivative}, it suffices
to calculate the limit 

\begin{equation}
    \lim_{\epsilon\to 0} \frac{1}{\epsilon}\int X^{t-k}_n \prod_{j=1}^l (X^{(j)}_n)^{k_j-s_j} (T^{(j)}_n)^{s_j} 
\end{equation}

for any given integer tuple $(s_1,\cdots ,s_l)$ with $0\le s_i\le k_i , i=1,\cdots,l$ and $\DS \sum_{i=1}^l s_i >0$. 
Note that we can write $X_n^{t-k}$ instead of $X^{t-k}_n(\epsilon)$ since $\frac{X_n}{X_{n}(\epsilon)}\to 1$ as $\epsilon\to 0$. To see this, note that $n = \mathcal{O}(\log\frac{1}{|\epsilon|})$ and $o(m)\le m$. Thus if $|\epsilon|$ is small enough, we have for any $1\le m\le n$,

$$\frac{\lambda_{i_1}(\epsilon) \cdots\lambda_{i_m}(\epsilon)}{\lambda_{i_1} \cdots\lambda_{i_m}} = \frac{(\lambda_1 + \epsilon)^{o(m) }\lambda_2^{m-o(m)}}{\lambda_1 ^{o(m) }\lambda_2^{m-o(m)}} = 
\left(1+\frac{\epsilon}{\lambda_1}\right)^{o(m)}\to 1,\text{ as }\epsilon\to 0.$$
Hence $\frac{X_n(\epsilon)}{X_n}\to 1$ as $\epsilon\to 0$.

We claim that 
\begin{align}\label{Similar estimate 3 of moment condition and bounded derivative}
   & \lim_{\epsilon\to 0}\frac{1}{\epsilon} \int X^{t-k}_n \prod_{j=1}^l (X^{(j)}_n)^{k_j-s_j}(T^{(j)}_n)^{s_j} \\
   &=
    \begin{cases}
        0, \text{ if }\DS \sum_{j=1}^l s_j \ge 2 \\
        \sum_{1\le j\le l,k_j \ge 1} \int X^{t-k} \prod_{i\neq j}(X^{(i)})^{k_i} k_j (X^{(j)})^{k_j-1} X^{(j+1)}, \text{ if }\DS \sum_{j=1}^l s_j = 1,
    \end{cases} \notag
\end{align}
which agrees with taking formal derivative.

Note that for any $j\in \{1,\cdots,l\}$ and $r\in \{1,\cdots,n\}$,
$$
    b^{(j)}_{n,r} = \frac{j!}{\lambda_1^{r+j}} \sum_{m=1}^n \lambda_{i_1}\cdots\lambda_{i_m} \binom{o(m)-j}{r} \binom{o(m)}{j}
    $$$$
    \le \frac{l!}{\lambda_1^{l}}  \frac{1}{\lambda_1^r}\sum_{m=1}^n \lambda_{i_1}\cdots\lambda_{i_m} (o(m)-j)^r o(m)^j  
    \preceq  \frac{1}{\lambda_1^r}o(n)^{r}\sum_{m=1}^n \lambda_{i_1}\cdots\lambda_{i_m}  (1+o(m)^{l+1}).
$$

Therefore 
\begin{equation}
    |(\sum_{r=1}^{n} b^{(j)}_{n,r}  \epsilon^r)^{s_{j}}| \preceq \left(\sum_{m=1}^n \lambda_{i_1}\cdots\lambda_{i_m}  (o(m))^{l}\right)^{s_j}\left(\sum_{r=1}^{n} o(n)^r
    \left(\frac{|\epsilon|}{\lambda_1}\right)^r \right)^{s_j},\quad j=1,\cdots ,l.
\end{equation}
On the other hand, we have 
$$\prod_{j=1}^l (X^{(j)}_n)^{k_j-s_j} \le (\sum_{m=1}^{n}{\lambda_{i_1}} \cdots\lambda_{i_m}K_l(m))^{k-s}\preceq (\sum_{m=1}^{n}{\lambda_{i_1}} \cdots\lambda_{i_m}(1+o(m)^{l+1}))^{k-s},$$
where $s:=s_1+ \cdots+s_l$. Taking products, we see that
$$
    X^{t-k}_n \prod_{j=1}^l (X^{(j)}_n)^{k_j-s_j}|(T^{(j)}_n)^{s_j}|\preceq X_n^{t-k}  (\sum_{m=1}^{n}{\lambda_{i_1}} \cdots\lambda_{i_m}(1+o(m)^{l+1}))^{(k-s)+s}(\sum_{r=1}^{n} o(n)^r (\frac{|\epsilon|}{\lambda_1})^r)^s
    $$$$
    \le X_n^{t-k} \left(\sum_{m=1}^{n}X_n(1+o(m)^{l+1})\right)^{k}\left(\sum_{r=1}^{n} n^r (\frac{|\epsilon|}{\lambda_1})^r\right)^s 
    = X_n^t \sum_{m=1}^n (1+m^{l+1})^k\left(\sum_{r=1}^{n} n^r \left(\frac{|\epsilon|}{\lambda_1}\right)^r\right)^s
    $$$$
    \preceq X_n^t (1+n^{l^2+l+1}) \left(\sum_{r=1}^{n} n^r \left(\frac{|\epsilon|}{\lambda_1}\right)^r\right)^s
$$

If $s\ge 2$, it follows that
\begin{align}\label{Moment Condition: Tails of order ge 2 converges to 0}
    |\frac{1}{\epsilon} X^{t-k}_n \prod_{j=1}^l (X^{(j)}_n)^{k_j-s_j} (T^{(j)}_n)^{s_j} | &\preceq (1+n^{l^2+l+1})
    \left(\sum_{r=1}^{n} n^r \frac{|\epsilon|^{r(1-\frac{1}{rs})}}{\lambda_1^r}\right)^s\int X_n^t
\end{align}
and the right hand side converges to $0$ as $\epsilon\to 0$, since $\int X_n^t$ is uniformly bounded, $n$ is of order $\log\frac{1}{|\epsilon|}$ and $1-\frac{1}{rs}\ge \frac{1}{2},\forall r\ge 1$.\\

The remaining terms in $G_2$ correspond to those tuples where one of $s_j$ is one and the others are zero. They sum up to
\begin{align*}
    \phi^{(k)}(X_n(\epsilon))\sum_{1\le j\le l,k_j\ge 1} k_j \prod_{i\neq j} (X_n^{(i)})^{k_i} (X_n^{(j)})^{k_j-1}  T_{n}^{(j)}.
\end{align*}
Fix any $j\in\{1,\cdots,l\}$ with $k_j \ge 1$, it suffices to show that
\begin{equation*}
    \lim_{\epsilon\to 0} \frac{1}{\epsilon}\int  \phi^{(k)}(X_n(\epsilon))k_j \prod_{i\neq j} (X_n^{(i)})^{k_i} (X_n^{(j)})^{k_j-1}  T_{n}^{(j)} = \int \phi^{(k)}(X)k_j \prod_{i\neq j} (X^{(i)})^{k_i} (X^{(j)})^{k_j-1}  X^{(j+1)}.
\end{equation*}

Observe that 
$$\frac{1}{\epsilon}T^{(j)}_n = b^{(j)}_{n,1} + \sum_{r=1}^{n-1} b^{(j)}_{n,r+1} \epsilon^r \quad \text{ and }\quad b^{(j)}_{n,1} = X_n^{(j+1)}.$$
Thus it suffices to show

\begin{equation}\label{Similar estimate 4 in moment conditiona and bounded derivative}
    \lim_{\epsilon\to 0} \int  \phi^{(k)}(X_n(\epsilon))k_j \prod_{i\neq j} (X_n^{(i)})^{k_i} (X_n^{(j)})^{k_j-1}  
    \left(\sum_{r=1}^{n-1} b^{(j)}_{n,r+1} \epsilon^r \right) =0,
\end{equation}
and 
\begin{equation}\label{Moment Condition: Estimate of first order term in G_2}
     \lim_{\epsilon\to 0} \int  \phi^{(k)}(X_n(\epsilon))k_j \prod_{i\neq j} (X_n^{(i)})^{k_i} (X_n^{(j)})^{k_j-1}  X_{n}^{(j+1)} = 
\end{equation}
$$ \int \phi^{(k)}(X)k_j \prod_{i\neq j} (X^{(i)})^{k_i} (X^{(j)})^{k_j-1}  X^{(j+1)}.$$
The first identity can be proved in the same way as \eqref{Moment Condition: Tails of order ge 2 converges to 0} so we omit its proof. For the second identity, we discuss separately the cases $t\ge k$ and $t<k$.

If $t\ge k$, then 
\begin{align*}
      \phi^{(k)}(X_n(\epsilon))k_j \prod_{i\neq j} (X_n^{(i)})^{k_i} (X_n^{(j)})^{k_j-1} X_{n}^{(j+1)}&\preceq(\sum_{m=1}^n \lambda_{i_1}\cdots\lambda_{i_m} K_l(m))^{t-k+\sum_{i\neq j}k_i + (k_j-1)+1}\\
      &= (\sum_{m=1}^n \lambda_{i_1}\cdots\lambda_{i_m} K_l(m))^t
\end{align*}
where the inequality follows by $\DS\ \lim_{\epsilon\to 0} \frac{X_n{(\epsilon)}}{X_n} = 1$, $t-k\ge 0$ and 
$\DS X_n^{(j)} \le\sum_{m=1}^n \lambda_{i_1}\cdots\lambda_{i_m} K_l(m)$ for $j=0,1\cdots ,l $. Since $\frac{\lambda_1^t+\lambda_2^t}{2}<1$, the above right hand side converges a.e. and in $L^1$ as $\epsilon\to 0$, by the monotone convergence theorem. Moreover, the left hand side converges a.e. to $\phi^{(k)}(X)k_j \prod_{i\neq j} (X^{(i)})^{k_i} (X^{(j)})^{k_j-1}  X^{(j+1)}$ as $\epsilon\to 0$. Hence the generalized dominated convergence theorem implies 
\eqref{Moment Condition: Estimate of first order term in G_2}.

If $t<k$, then define $\beta = \frac{t}{k}\in (0,1)$, we have
\begin{align*}
      \phi^{(k)}(X_n(\epsilon))k_j \prod_{i\neq j} (X_n^{(i)})^{k_i} (X_n^{(j)})^{k_j-1} X_{n}^{(j+1)}&\preceq \frac{(\sum_{m=1}^n \lambda_{i_1}\cdots\lambda_{i_m} K_l(m))^{\sum_{i\neq j}k_i + (k_j-1)+1}}{X_n^{k-t}}\\
      = \frac{(\sum_{m=1}^n \lambda_{i_1}\cdots\lambda_{i_m} K_l(m))^k}{X_n^{k(1-\beta)}} 
      &\le (\sum_{m=1}^n \lambda^{\beta}_{i_1}\cdots\lambda^{\beta}_{i_m} K_l(m))^k.
\end{align*}
Since $\frac{\lambda_1^{k\beta}+\lambda_2^{k\beta}}{2} = \frac{\lambda_1^t + \lambda_2^t}{2}<1$, the right hand side is both convergent a.e. and in $L^1$ as $\epsilon\to 0$, and similarly the generalized dominated convergence theorem leads to 
\eqref{Moment Condition: Estimate of first order term in G_2}. This finishes the calculation for $G_2$.
\\

It remains to calculate $G_1$. Note that
\begin{align*}
    &\frac{1}{\epsilon} \int  (\phi^{(k)} (X_n(\epsilon))-\phi^{(k)}(X_n))
    \prod_{j=1}^{l} (X^{(j)}_n )^{k_j} \\
    &= \int  (\phi^{(k+1)} (X_n + \xi_n(X_n(\epsilon)-X_n)) \frac{X_n(\epsilon) - X_n}{\epsilon}\prod_{j=1}^{l} (X^{(j)}_n )^{k_j} \quad \text{for some }\xi_n\in(-1,1)\\
    &= \int  (\phi^{(k+1)} (X_n + \xi_n(X_n(\epsilon)-X_n)) \prod_{j=1}^{l} (X^{(j)}_n )^{k_j}(X_n^{(1)} + \sum_{r=1}^{n-1}
    b^{(0)}_{n,r+1}\epsilon^{r}).
\end{align*}
 Here, $b_{n,r}^{(0)}$ denotes the coefficient of $\epsilon^r$ in the expression $X_n(\epsilon)-X_n$ (recall \eqref{Moment Condition: Coefficient for epsilon^r}):

   $$b_{n,r}^{(0)} = \frac{1}{\lambda_1^r}\sum_{m=1}^{n}\lambda_{i_1}\cdots\lambda_{i_m} \binom{o(m)}{r}.$$
   
   We claim that
   \begin{equation}\label{Moment Condition: Estimate of G_1, tail term}
     \lim_{\epsilon\to 0} \int  (\phi^{(k+1)} (X_n + \xi_n(X_n(\epsilon)-X_n)) \prod_{j=1}^{l} (X^{(j)}_n )^{k_j} ( \sum_{r=1}^{n-1}b^{(0)}_{n,r+1}\epsilon^{r}) = 0
   \end{equation}
   and
   \begin{equation}\label{Moment Condition: Estimate of G_1, first order term}
      \lim_{\epsilon\to 0} \int  (\phi^{(k+1)} (X_n + \xi_n(X_n(\epsilon)-X_n)) \prod_{j=1}^{l} (X^{(j)}_n )^{k_j}X_n^{(1)}  = \int  \phi^{(k+1)} (X) \prod_{j=1}^{l} (X^{(j)} )^{k_j}X^{(1)},
   \end{equation}
   which  implies the desired result. Moreover, \eqref{Moment Condition: Estimate of G_1, tail term} can be proved in the same way as \eqref{Moment Condition: Tails of order ge 2 converges to 0}, and \eqref{Moment Condition: Estimate of G_1, first order term} is proved in the same way as \eqref{Moment Condition: Estimate of first order term in G_2}. The proof is complete.
   \end{proof}

\section{Proof of Theorem \ref{Differentiability of Stationary measure} for higher derivatives.}
\label{ScHD}
\label{Subsection: smoothness of moment condition; proof of induction step}

We are ready to finish the inductive step of Theorem \ref{Differentiability of Stationary measure}. 
\begin{proof}[Continuation of proof of Theorem \ref{Differentiability of Stationary measure}]
    Recall that $\frac{\lambda_1^r+\lambda_2^r}{2}<1$ where $r\in\mathbb{N}$. We wish to show  $h(\epsilon):= \int\phi\,d\mu_{\lambda_1+\epsilon,\lambda_2}$ is differentiable at $0$ up to order $r$ and the derivatives are given by 
    \begin{align*}
        h^{(l)}(0) &= \int_{\Sigma} \frac{\,d^l}{\,d \epsilon^l} (\phi\circ X(\epsilon))|_{\epsilon=0}\,d\nu \\
        &= \int_{\Sigma} \sum_{k_1 + 2k_2 + \cdots lk_l = l} L(k_1, \cdots, k_l) \phi^{(k)} (X) \prod_{j=1}^{l} (X^{(j)})^{k_j}\,d\nu,\; l=1,2,\cdots ,r.
        \end{align*}
        The base case $r=1$ is taken care of in Section \ref{Section: Differentiability via Moment Condition and Bound on Derivative }. Now suppose $r\ge 2$ and the theorem is true for $r-1$. We would like to push it to $r$. Note that $\frac{\tilde{\lambda}_1^{r-1}+\lambda_2^{r-1}}{2}<1$, for any $\tilde{\lambda}_1$ sufficiently close to $\lambda_1$. So by inductive hypothesis $h$ is differentiable up to order $r-1$ near $0$ at sufficiently small $\epsilon$ and the $(r-1)$-th derivative is 
        \begin{equation}
             h^{(r-1)}(\epsilon) = \int_{\Sigma} \sum_{k_1 + 2k_2 + \cdots (r-1)k_{r-1} = r-1} L(k_1, \cdots, k_{r-1}) \phi^{(k)} (X(\epsilon)) \prod_{j=1}^{r-1} (X^{(j)}(\epsilon))^{k_j}\,d\nu.
        \end{equation}
     We will follow the same strategy, proving the inductive step by showing the following two lemmas.
    Recall that $X_n$ and $X_n^{(j)}$'s are the finite truncates of $X,X^{(j)}$'s at $n$:    $$X_n =  \sum_{m=1}^n \lambda_{i_1 }\cdots \lambda_{i_m},\quad
    X_n^{(j)} = \frac{j!}{\lambda_1^j}\sum_{m=1}^{n} \lambda_{i_1} \cdots \lambda_{i_m} \binom{o(m)}{j} .$$

     \begin{lm}\label{Moment condition and bound on derivative: finite truncate lemma}
    
  There exists $c>0 $ and $0<\theta<1$ that depends on $\lambda_1,\lambda_2$, $t$ and $l$ such that 
    \begin{align*}
        |&\int \sum_{k_1 + 2k_2 + \cdots (r-1)k_{r-1} = r-1} L(k_1, \cdots, k_{r-1}) \phi^{(k)} (X) \prod_{j=1}^{r-1} (X^{(j)})^{k_j} \\
        &-\int \sum_{k_1 + 2k_2 + \cdots (r-1)k_{r-1} = r-1} L(k_1, \cdots, k_{r-1}) \phi^{(k)} (X_n) \prod_{j=1}^{r-1} (X^{(j)}_n)^{k_j}| < c\theta ^n.
    \end{align*}

\end{lm}

\begin{lm}\label{Moment Condition and bound on derivative: Finite Difference Converges to Derivative}

    For any $\epsilon\neq0$, let $n = n(\epsilon)\in\mathbb{N} $ be such that there exists some constant d satisfying $\frac{1}{d} \log\frac{1}{|\epsilon|} \le n \le d \log\frac{1}{|\epsilon|},\forall \epsilon\neq0$. Then 
    \begin{align*}
        \lim_{\epsilon\to 0} \frac{1}{\epsilon} &\left(\int \sum_{k_1 + 2k_2 + \cdots (r-1)k_{r-1} = r-1} L(k_1, \cdots, k_{r-1}) \phi^{(k)} (X_n(\epsilon)) \prod_{j=1}^{r-1} (X^{(j)}_n (\epsilon))^{k_j}\right. \\
        &-\left.\int \sum_{k_1 + 2k_2 + \cdots (r-1)k_{r-1} = r-1} L(k_1, \cdots, k_{r-1}) 
        \phi^{(k)} (X_n) \prod_{j=1}^{r-1} (X^{(j)}_n)^{k_j}\right) \\
        &=\int \sum_{k_1 + 2k_2 + \cdots (l+1)k_l = l+1} L(k_1, \cdots, k_{l+1}) \phi^{(k)} (X) \prod_{j=1}^{l+1} (X^{(j)})^{k_j}.
    \end{align*}
    Here $\displaystyle X_n (\epsilon):= \sum_{m=1}^n \lambda_{i_1}(\epsilon) \cdots \lambda_{i_m}(\epsilon)$ and 
    $$X_n^{(j)} (\epsilon) := \frac{\,d^j}{\,d \epsilon^j}X_n (\epsilon) =  \frac{j!}{\lambda_1^j}\sum_{m=1}^{n} \lambda_{i_1}(\epsilon) \cdots \lambda_{i_m}(\epsilon) \binom{o(m)}{j}.$$
    
 \end{lm}

Once the above two lemmas are proved, we can conclude the proof in  the same way as we did in the proof of Theorem \ref{infinite differentiability for compactly supported test function}. Therefore, the remainder of this section will be devoted to proving the two lemmas.
\begin{proof}
[Proof of Lemma \ref{Moment condition and bound on derivative: finite truncate lemma}]
By the same reasoning as we did in its counterpart,  Lemma~\ref{moment condition: finite truncate lemma}, we need to show
that
  \begin{align*}
        & \int  \phi^{(k)} (X) (\prod_{j=1}^{l} (X^{(j)})^{k_j}-\prod_{j=1}^{l} (X^{(j)}_n)^{k_j}) 
        +\int  |(\phi^{(k)} (X)-\phi^{(k)} (X_n)) \prod_{j=1}^{l} (X^{(j)}_n)^{k_j}|  
        := J_1 + J_2 .
   \end{align*}
    is exponentially small in $n$. 

\noindent
    We first deal with $J_1$.  Recall the notation $\displaystyle X^{(j)} \!\!-\!\! X_n^{(j)} \!\!=\!\! \frac{j!}{\lambda_1^j} \!\! \sum_{m=n+1}^{\infty} \!\!\lambda_{i_1}\cdots \lambda_{i_m}\binom{o(m)}{j}\!\! =: R_n^{(j)}$. As before, we can further reduce  the problem to  estimating, for any given tuple $(k_1,\cdots,k_{r-1})$ satisfying 
    $\DS \sum_{j=1}^{r-1} jk_j = r-1$ and any tuple $(s_1,\cdots,s_{r01})$ with $0\le s_j\le k_j$, $j=1,\cdots,r-1$, the following integral
  $$\int \phi^{(k)}(X) \prod_{j=1}^{r-1} (X_n^{(j)})^{k_j-s_j}(R^{(j)}_n)^{s_j}.$$

Since $\phi^{(r)}$ is bounded and $k\le r-1$, there exists some constant $M$ such that
\begin{equation}\label{Polynomial Bound on Derivatives of phi}
    \phi^{(k)}(x) \le M x^{r-k},\;\forall x \ge 1, \;\forall\, 0\le k\le r.
\end{equation}

Since $X\ge1$ by definition, we have

$$\int \phi^{(k)}(X) \prod_{j=1}^{r-1} (X_n^{(j)})^{k_j-s_j}(R^{(j)}_n)^{s_j} \le M \int X^{r-k} \prod_{j=1}^{r-1} (X_n^{(j)})^{k_j-s_j}(R^{(j)}_n)^{s_j}.$$
 With $\frac{\lambda_1^r+\lambda_2^r}{2}<1$, we are in the same situation as in 
 \eqref{Similar Estimate 1 in moment and bouded derivative} and the rest of the proof follows verbatim.

 Next we consider $J_2$. We have
 $$      |(\phi^{(k)} (X)-\phi^{(k)} (X_n)) \prod_{j=1}^{r-1} (X^{(j)}_n)^{k_j}|
      =\phi^{(k+1)}(X_n + \xi_n (X-X_n))(X-X_n)  \prod_{j=1}^{r-1} (X^{(j)}_n)^{k_j}$$
$$      \le M (X_n + \xi_n (X-X_n))^{r-k-1}\prod_{j=1}^{r-1} (X^{(j)}_n)^{k_j}
      \le M X^{r-k-1}\prod_{j=1}^{r-1} (X^{(j)}_n)^{k_j},
$$
  where in the last inequality we have used that $r-k-1\ge 0$. Now we are in the same situation as in the second line of 
  \eqref{Similar estimate 2 in moment condition and bounded derivative}, so the proof follows by the same argument. This finishes the proof of Lemma \ref{Moment condition and bound on derivative: finite truncate lemma}.
  \end{proof}
\begin{proof}
  [Proof of Lemma \ref{Moment Condition and bound on derivative: Finite Difference Converges to Derivative}]
   As in the proof of Lemma \ref{Moment Condition: Finite Difference Converges to Derivative}, we only need to show, for any given tuple of integers $(k_1,\cdots k_{r-1})$ satisfying $\DS \sum_{i=1}^{r-1} (r-1)k_{r-1} = r-1$, that 
      \begin{equation}\label{Moment Condition: Finite Difference converges to derivative, step 1}
          \lim_{\epsilon\to 0} \frac{1}{\epsilon} \left(\int  \phi^{(k)} (X_n(\epsilon)) \prod_{j=1}^{l} (X^{(j)}_n (\epsilon))^{k_j} -\int  \phi^{(k)} (X_n) \prod_{j=1}^{l} (X^{(j)}_n)^{k_j}\right) 
          \end{equation}
        $$=\int  \phi^{(k+1)} (X)X^{(1)} \prod_{j=1}^{l} (X^{(j)})^{k_j} + 
        \int  \phi^{(k)} (X) \sum_{j=1}^{l} (k_j (X^{(j)})^{k_j -1} X^{(j+1)}\prod_{i\neq j} (X^{(i)})^{k_i}) ,$$

where $\DS k:=\sum_{i=1}^{l} k_i$. As before, the left hand side is divided into two parts.
\begin{align*}
    &\int  \phi^{(k)} (X_n(\epsilon)) \prod_{j=1}^{l} (X^{(j)}_n (\epsilon))^{k_j}-\int  \phi^{(k)} (X_n) \prod_{j=1}^{l} (X^{(j)}_n)^{k_j})\\
    =&\int  (\phi^{(k)} (X_n(\epsilon))-\phi^{(k)}(X_n))
    \prod_{j=1}^{l} (X^{(j)}_n )^{k_j} 
    +\int  \phi^{(k)} (X_n(\epsilon)) (\prod_{j=1}^{l} (X^{(j)}_n (\epsilon))^{k_j}-\prod_{j=1}^{l} (X^{(j)}_n)^{k_j})) \\
    =:& G_1 + G_2.
\end{align*}
We deal with $G_2$ first. Recall for $j=0,\cdots,l$ the notations
 $$X_n^{(j)}(\epsilon) = X_n^{(j)} + \sum_{r=1}^{n} b^{(j)}_{n,r} \epsilon^r=: X^{(j)}_n + T^{(j)}_n,$$
 and
 \begin{equation}
     b^{(j)}_{n,r}:=\frac{j!}{\lambda_1^{r+j}} \sum_{m=1}^n \lambda_{i_1}\cdots\lambda_{i_m} \binom{o(m)-j}{r} \binom{o(m)}{j}.
 \end{equation}
 
By the same reasoning as in Lemma \ref{Moment Condition: Finite Difference Converges to Derivative}, it suffices to calculate the limit 

\begin{equation}
    \lim_{\epsilon\to 0} \frac{1}{\epsilon}\int \phi^{(k)}(X_n) \prod_{j=1}^{r-1} (X^{(j)}_n)^{k_j-s_j} (T^{(j)}_n)^{s_j} 
\end{equation}
for each given tuple $(s_1,\cdots,s_{r-1})$. We claim that 
\begin{align}
   & \lim_{\epsilon\to 0}\frac{1}{\epsilon} \int \phi^{(k)}(X_n) \prod_{j=1}^{r-1} (X^{(j)}_n)^{k_j-s_j}(T^{(j)}_n)^{s_j} \\
   &=
    \begin{cases}
        0, \text{ if }\DS \sum_{j=1}^{r-1} s_j \ge 2 \\
        \sum_{1\le j\le r-1,k_j \ge 1} \int \phi^{(k)}(X) \prod_{i\neq j}(X^{(i)})^{k_i} k_j (X^{(j)})^{k_j-1} X^{(j+1)}, \text{ if }\DS \sum_{j=1}^{r-1} s_j = 1,
    \end{cases}
\end{align}
which agrees with taking formal derivative. 

The case $\DS \sum_{j=1}^{r-1} s_j\ge 2$ holds since $\phi^{(k)}(X_n)\le M X_n^{r-k}$ and the rest of the proof follows the same lines as in the first case in \eqref{Similar estimate 3 of moment condition and bounded derivative}. 

As for the case $\DS \sum_{j=1}^{r-1} s_j= 1$, similar as before we need to show that for each $k_j\ge 1$,

\begin{equation}
    \lim_{\epsilon\to 0} \int  \phi^{(k)}(X_n(\epsilon))k_j \prod_{i\neq j} (X_n^{(i)})^{k_i} (X_n^{(j)})^{k_j-1}  
    \left(\sum_{r=1}^{n-1} b^{(j)}_{n,r+1} \epsilon^r \right) =0,
\end{equation}
and 
$$
     \lim_{\epsilon\to 0} \int  \phi^{(k)}(X_n(\epsilon))k_j \prod_{i\neq j} (X_n^{(i)})^{k_i} (X_n^{(j)})^{k_j-1}  X_{n}^{(j+1)} = 
\int \phi^{(k)}(X)k_j \prod_{i\neq j} (X^{(i)})^{k_i} (X^{(j)})^{k_j-1}  X^{(j+1)}.$$
To prove the first identity, note that $\DS \phi^{(k)}(X_n(\epsilon))\le MX_n(\epsilon)^{r-k}\le 2MX_n^{r-k}$, where the second inequality holds for sufficiently small $\epsilon$, since $\frac{X_n(\epsilon)}{X_n}\to 1$ as $\epsilon\to 0$. The rest part of the proof is the same as the proof of \eqref{Similar estimate 4 in moment conditiona and bounded derivative}.

To prove the second identity, note that the integrand on the left hand side converges a.e. to the integrand on the right hand side. Moreover, since $X_n(\epsilon),X_n \ge 1$ and $\frac{X_n(\epsilon)}{X_n}\!\!\to \!\!1$ as $\epsilon\to 0$, we have
\begin{equation*}
     \phi^{(k)}(X_n(\epsilon))k_j \prod_{i\neq j} (X_n^{(i)})^{k_i} (X_n^{(j)})^{k_j-1}  T_{n}^{(j)}\le2 k_j M X_n^{r-k} | \prod_{i\neq j} (X_n^{(i)})^{k_i} (X_n^{(j)})^{k_j-1}  T_{n}^{(j)}|.
\end{equation*}
Since $\frac{\lambda_1^r+\lambda_2^2}{2}<1$, the above right hand side converges both a.e and in $L^1$, as the proof of identity \eqref{Moment Condition: Estimate of first order term in G_2} shows. Therefore, the generalized dominated convergence theorem \ref{Generalized Dominated Convergence Lemma} confirms the second identity.

Finally we consider $G_1$. As in the proof of \eqref{Moment Condition: Finite Difference Converges to Derivative}, it suffices to show
   \begin{equation}
     \lim_{\epsilon\to 0} \int  (\phi^{(k+1)} (X_n + \xi_n(X_n(\epsilon)-X_n)) \prod_{j=1}^{l} (X^{(j)}_n )^{k_j} ( \sum_{r=1}^{n-1}b^{(0)}_{n,r+1}\epsilon^{r}) = 0
   \end{equation}
   and
   \begin{equation}
      \lim_{\epsilon\to 0} \int  (\phi^{(k+1)} (X_n + \xi_n(X_n(\epsilon)-X_n)) \prod_{j=1}^{l} (X^{(j)}_n )^{k_j}X_n^{(1)}  = \int  \phi^{(k+1)} (X) \prod_{j=1}^{l} (X^{(j)} )^{k_j}X^{(1)}.
   \end{equation}
    Given the proof of\eqref{Moment Condition: Estimate of G_1, tail term} and \eqref{Moment Condition: Estimate of G_1, first order term}, it is enough to observe that 
    $$ (\phi^{(k+1)} (X_n + \xi_n(X_n(\epsilon)-X_n))\le 2MX_n^{r-k-1},\text{ for sufficiently small }\epsilon.$$
    The proof of Lemma \ref{Moment Condition and bound on derivative: Finite Difference Converges to Derivative} is complete.
\end{proof}    
Lemmata \ref{Moment condition and bound on derivative: finite truncate lemma} and \ref{Moment Condition and bound on derivative: Finite Difference Converges to Derivative}
establish
Theorem \ref{Differentiability of Stationary measure}. 
    \end{proof}

%% file: Extensions_and_open_problems.tex
\section{Extensions and open problems}
\label{Section: Extensions and open problems}

 We end by discussing several generalization of the previous results.

\subsection{Differentiability in other directions}

In this subsection we discuss differentiability with respect to $\lambda_2,d_1,d_2$.

 First, one could define $g_{\phi}(\epsilon) = g(\epsilon) = \int_{\mathbb{R}} \phi\,d\mu_{\lambda_1,\lambda_2+\epsilon}$ and study the derivative with respect to $\lambda_2$. All the above theorems remains valid with the same proof, and the formulae for the derivative is replaced with the formal derivative taken with respect to $\lambda_2$. 

Likewise, if one studies the derivative with respect to $d_1$ or $d_2$, then the same assumptions give the same conclusions in each theorem, with the derivative equal to the formal derivative. But in this case, the formal derivative becomes much simpler. For example, let us consider the derivative with respect to $d_1$. That is, define 
$$h(d) = \int_{\Sigma} \phi(\sum_{n=1}^{\infty} d_{i_n} \lambda_{i_1}\cdots\lambda_{i_{n-1}}) \,d\nu$$
where $d_{1} = d,d_2 = d_2$ and $d_2$ is given. By taking formal derivative, we can prove

$$h'(d) =  \int_{\Sigma} \phi'(\sum_{n=1}^{\infty} d_{i_n} \lambda_{i_1}\cdots\lambda_{i_{n-1}})  \left(\sum_{n=1}^{\infty} \lambda_{i_1}\cdots\lambda_{i_{n-1}}\delta_{n,1}\right)\,d\nu$$
where $\delta_{n,1} = 1$ if $i_n = 1$ and is zero otherwise. If we take higher order derivatives, whenever a derivative hits the term $\DS \sum_{n=1}^{\infty} \lambda_{i_1}\cdots\lambda_{i_{n-1}}\delta_{n,1}$ it yields zero. So, under the correct assumptions, we have
$$h^{(l)}(d) =  \int_{\Sigma} \phi^{(l)}(\sum_{n=1}^{\infty} d_{i_n} \lambda_{i_1}\cdots\lambda_{i_{n-1}})  \left(\sum_{n=1}^{\infty} \lambda_{i_1}\cdots\lambda_{i_{n-1}}\delta_{n,1}\right)^{l}\,d\nu.$$

The above discussion readily implies that, if $d_1,d_2,\lambda_1,\lambda_2$ are $C^1$ (resp. $C^l$) functions defined on some neighborhood $(-\epsilon,\epsilon)$ of $0$, then $h_\phi$ will be differentiable (resp. $l$-times differentiable) with the derivative coinciding with the formal derivative, under the same assumption as in the main theorems. For instance, if $\phi\in C_c^\infty(\mathbb{R})$, then $h_\phi$ is $C^\infty$, and 
$$h_\phi'(0) = \int_{\Sigma} \phi'(X) \left(\lambda_1'(0) \frac{\partial X}{\partial \lambda_1} +\lambda_2'(0) \frac{\partial X}{\partial \lambda_2}+d_1'(0) \frac{\partial X}{\partial d_1}+d_2'(0) \frac{\partial X}{\partial d_2}\right) \,d\nu(i).$$
Here $\DS X = X(\lambda_1,\lambda_2,d_2,d_2;i) = \sum_{m=1}^\infty  d_{i_m} \lambda_{i_1}\cdots \lambda_{i_{m-1}}$ with $i = (i_1,i_2,\cdots)\in \Sigma := \{1,2\}^{\mathbb{N}}$. And $\frac{\partial X}{\partial \lambda_j}, \frac{\partial X}{\partial d_j}, j=1,2$ are the formal derivative, taken by viewing $X$ as a power series. For higher order derivatives one gets formal derivatives of higher order.

\subsection{Multiple maps}

Next, we address the case of multiple maps $f_i(x) = \lambda_i x+d_i,i=1,2,\cdots,k$ assigned with probability $(p_1,\cdots,p_k)$. Here we assume $k\ge 2$, $\DS \sum _i p_i = 1$ and there is at least one contraction and one expansion among $f_1,\cdots, f_k$. Moreover, the system is contracting on average: $\DS \sum_i p_i \log |\lambda_i|<0$ and $f_i$'s do not share the same fixed point. 

We claim that in this case, the {\it actual derivative equals formal derivative} 
assertions in our main results remain valid, under mild change of the assumption. More precisely:

Theorem \ref{infinite differentiability for compactly supported test function} remains true if in addition we assume that each $\lambda_i,d_i>0,\forall i$. This is because the key Varadham's lemma-type of estimate (see \ref{Equation: Varadhan's Type of Estimate}) will still be true, given that 
$\DS \sum_{m=1}^N d_{i_m} \lambda_{i_1} \cdots \lambda_{i_{m-1}}$ is monotone increasing in $N$ and the systems is contracting on average.

Theorem \ref{Differentiability of Stationary measure} will hold if we replace the assumption $\DS \frac{\lambda_1^r+\lambda_2^r}{2}<1$ by $\DS \sum_i p_i |\lambda_i|^r<1$. To see this, observe that the proof of the original theorem relies critically on two facts:

\begin{enumerate}
    \item $|\phi^{(l)}(x)|\preceq |x|^{r-l}+1, l=0,1,\cdots,r$, since $||\phi^{(r)}||_{C^0}<\infty$.
    \item $\DS \frac{\lambda_1^t+\lambda_2^t}{2}<1, \forall0<t\le r$. And this is due to the behavior of the function $\DS f(t):=\log \frac{\lambda_1^t+\lambda_2^t}{2} , t\ge 0$: $f(0)=0,f'(0)<0$, $f''>0$ and $f(\infty) = \infty$. Item $(2)$ will follow from these by an elementary analysis. 
\end{enumerate}

These two facts together yield a uniform integrability estimate, which essentially finishes the proof. Now under the new setting, we have two similar facts that guarantee uniform integrability:

\begin{enumerate}
    \item $|\phi^{(l)}(x)|\preceq |x|^{r-l}+1, l=0,1,\cdots,r$, since $||\phi^{(r)}||_{C^0}<\infty$.
    \item $\DS\sum_{i=1}^k  p_i |\lambda_i|^t<1, \forall0<t\le r$. 
\end{enumerate}

The first claim is clearly true. The second claim holds since its left hand side comes from a moment generating function. 
Namely, consider a random variable $Y$ which takes values
$\log |\lambda_i|$  with probability $p_i$, $i=1,\cdots k$
and define 
$\DS f(t) = \log \mathbb{E}[e^{tY}]$ for $ t\ge 0.$
Then a direct computation gives
$$f'(t) = \frac{\mathbb{E}[e^{tY}Y]}{\mathbb{E}[e^{tY}]}, \quad f''(t) = \frac{\mathbb{E}[Y^2 e^{tY}]\mathbb{E}[e^{tY}] - \mathbb{E}[Ye^{tY}]^2}{\mathbb{E}[e^{tY}]^2}.$$

It readily follows that $f(0) = 0$, $f'(0)<0$ (by contracting on average), $f(\infty) = \infty$ (since there is at least one expansion). Moreover, $f''>0$ by Cauchy-Schwartz inequality. Thus the second item follows by simple analysis.

Finally, Theorem \ref{Smoothness for Finite Moment Condition} will also be true, if one changes the assumption $\DS \frac{\lambda_1^t+\lambda_2^t}{2}<1$ into $\DS \sum_i p_i |\lambda_i|^t<1$, by the same argument as above.

\subsection{Degenerate case: Common Fixed Point}

So far we have always assumed that the maps $f_i$ do not share a common fixed point. It turns out if there are only two maps, then  linear response still holds, even in the degenerate case.

More specifically, let  $f^t_1(x) = \lambda_1(t) x+d_1(t),f^t_2(x) = \lambda_2(t) x+d_2(t); \; p = (\frac{1}{2},\frac{1}{2})$ be a $C^\infty$ one parameter family of maps, with $0<\lambda_1<1<\lambda_2$ and $\lambda_1\lambda_2<1$ uniformly for $t\in(-\epsilon,\epsilon)$. Moreover assume that at $t=0$ the two maps have a same fixed point, or equivalently, $\DS \frac{d_1(0)}{1-\lambda_1(0)} = \frac{d_2(0)}{1-\lambda_2(0)}$. We also assume this degeneracy is isolated, so that $\DS \frac{d_1(t)}{1-\lambda_1(t)} \neq \frac{d_2(t)}{1-\lambda_2(t)}, \;\forall t\in(-\epsilon,\epsilon)\setminus \{0\}$.

Denote the stationary measure for $\{f_1^t,f_2^t;p\}$ by $\mu_t$. The goal is to study the differentiability of $h(t) = h_\phi(t) = \int\phi \,d\mu_t$ at $t=0$.

We claim that Theorems \ref{infinite differentiability for compactly supported test function}, \ref{Smoothness for Finite Moment Condition} and \ref{Differentiability of Stationary measure} remain true. The reason is that the proofs of these three theorems are based on the following two facts:
\begin{enumerate}
    \item $\mu_t$ is the law of  $X(t) = X(t;i) = \sum_{m=1}^\infty d_{i_m}(t)\lambda_{i_1}(t)\cdots \lambda_{i_{m-1}}(t),\; i = (i_1,i_2,\cdots)\in \Sigma $. Here recall that $\Sigma = \{1,2\}^\mathbb{N}$, $\nu = p^\mathbb{N}$, and $ i = (i_1,i_2,\cdots)\in \Sigma \text{ has law } p$. 
    \item After a change of variable, one can make $d_1,d_2$ to have the same sign, so that $\sum_{m=1}^N d_{i_m}(t)\lambda_{i_1}(t)\cdots \lambda_{i_{m-1}}(t)$ is monotone in $N$.
\end{enumerate}

In the present case, fact $1$ survives because when $t=0$, $\mu_t = \delta_{x_0}$ is the delta mass on the fixed point $x_0 = \frac{d_1(0) - d_2(0)}{\lambda_2(0) - \lambda_1(0)}$, which is indeed equal to the law of 
$\DS \sum_{m=1}^\infty d_{i_m}(0)\lambda_{i_1}(0)\cdots \lambda_{i_{m-1}}(0)$. To see this, note that for any fixed $N\in\mathbb{N}$,

$$
    \sum_{m=1}^N d_{i_m}(0)\lambda_{i_1}(0)\cdots \lambda_{i_{m-1}}(0) = \sum_{m=1}^N d_{i_m}(0)\lambda_{i_1}(0)\cdots \lambda_{i_{m-1}}(0) + \lambda_{i_1}(0) \cdots\lambda_{i_N}(0)x_0 - \lambda_{i_1}(0) \cdots\lambda_{i_N}(0)x_0 
    $$$$
    =f_{i_1}^0\circ \cdots f_{i_N}^0 (x_0) -  \lambda_{i_1}(0) \cdots\lambda_{i_N}(0)x_0\\
    =x_0 - \lambda_{i_1}(0) \cdots\lambda_{i_N}(0)x_0.
$$
As $N\to\infty$, the above discrete measure converges weakly to $\delta_{x_0}$ because $\mathbb{E}[\log\lambda_i(0)]\!\!<\!\!0$.

The second fact can be compensated as follows. Define a family of change of variables $c^t(x) = x + \frac{d_2(t)}{1-\lambda_2(t)}$. One can easily check that $f_i^t \circ c^t = c^t \circ \tilde{f}_i^t$ for $i =1,2$ where $\tilde{f}_1^t(x) = \lambda_1(t)x + \epsilon(t) , \tilde{f}_2^t(x) = \lambda_2(t)x$ and $\epsilon(t) = \frac{d_1(t)(1-\lambda_2(t)) - d_2(t)(1-\lambda_1(t))}{1-\lambda_2(t)}$. 

Denote by $\tilde{\mu}_t$ the stationary measure for $\{\tilde{f}_1^t,\tilde{f}_2^t;p\}$. By the same argument at the beginning of \S \ref{Section: Infinite Differentiability for Smooth Compactly Supported Test Function}, for any $\phi$, 

$$h_\phi(t) = \int\phi\,d\mu_t = \int\phi\circ c^t \,d\tilde{\mu}_t = \int\phi(x+\epsilon(t)) \,d\tilde{\mu}_t(x).$$
Moreover, for the system $\{\tilde{f}_1^t,\tilde{f}_2^t;p\}$, the translation factors $d_1(t) = \epsilon(t),d_2(t) = 0$ always have the same sign, which leads 
$$\forall M >0,\;\;\mathbb{E}[\sum_{m=N}^\infty d_{i_j}(t) \lambda_{i_1}(t)\cdots \lambda_{i_{m-1}}(t) 1_{\{|X(t;i)|\} \le M}   ] \le C\theta^N$$
for some $C=C(M,\lambda_1(0),\lambda_2(0))>0$ and $\theta = \theta(\lambda_1(0),\lambda_2(0))\in(0,1)$, uniformly for $N\in\mathbb{N}$ and $t\in(-\epsilon,\epsilon)$. This in turn yields uniform integrability and one can argue the same way as in Lemma \ref{finite approximation} and Lemma \ref{second step:finite difference converges to derivative} to show that, if $\phi\in C_c^\infty$, then $\tilde{h}_\phi(t):= \int\phi\,d\tilde{\mu_t}$ is $C^\infty$ at $0$ and the derivatives coincide with formal ones.

Therefore, linear response holds for the original system $\{f_1^t,f_2^t;p\}$ by chain rule:

$$h_\phi'(t) = \int\phi'(x+\epsilon(0)) \epsilon'(0) \,d\tilde{\mu}_0(x)+ \frac{d}{dt}\int\phi(x+\epsilon(0))\,d\tilde{\mu}_t =\int\phi'(x) \epsilon'(0) \,d\tilde{\mu}_0(x)+\tilde{h}'_\phi(0).$$
The last equality holds because $\epsilon(0)= 0$. Higher order derivatives can be computed in the same way.

\subsection{An open Question}

We end this section by pointing out a case we do not know how to deal with. Note that in Theorem \ref{infinite differentiability for compactly supported test function}, we relied on the fact that $d_1,d_2,\lambda_1,\lambda_2>0$ so that the sum $\DS \sum_{m=1}^{N} d_{i_m}\lambda_{i_1}\cdots\lambda_{i_{m-1}}$ is monotone increasing in $N$. And we are able to assume $d_1,d_2>0$ because there are two functions thus a change of variable makes them positive (see the discussion at the beginning of \S \ref{Section: Infinite Differentiability for Smooth Compactly Supported Test Function}). In general, if the IFS consists of  more than two maps, and the $\lambda_i$'s and $d_i$'s are such that 
$\DS \sum_{m=1}^\infty d_{i_j}\lambda_{i_1}\cdots\lambda_{i_{m-1}}
$ can have alternating terms, then the proof does not work. The author does not know whether Theorem \ref{infinite differentiability for compactly supported test function} will still hold in this case.